\documentclass[13pt,CJK]{article}
\usepackage{geometry}
\usepackage{amsfonts,amsmath,amssymb,amsthm,mathrsfs}
\usepackage{latexsym}
\usepackage{graphicx}
\usepackage{indentfirst}
\usepackage{color}
\usepackage{float} %设置图片浮动位置的宏包
\usepackage{subfigure} %插入多图时用子图显示的宏包
\usepackage{fancyhdr}
\usepackage[colorlinks,linktocpage,linkcolor=blue,citecolor=red]{hyperref}
\usepackage{titlesec}

\usepackage{authblk} %作者
\usepackage{pgf}
\usepackage{tikz}
\usepackage[utf8]{inputenc}
\usetikzlibrary{arrows,automata}
\usetikzlibrary{positioning}

\tikzset{
    state/.style={
           rectangle,
           rounded corners,
           draw=black, very thick,
           minimum height=2em,
           inner sep=2pt,
           text centered,
           },
}

\tikzset{global scale/.style={
    scale=#1,
    every node/.append style={scale=#1}
  }
}

\definecolor{darkgreen}{rgb}{0.00,0.50,0.25}
\definecolor{purple}{rgb}{0.54,0.17,0.89}
\definecolor{purple1}{rgb}{1.00,0.00,1.00}

\titleformat*{\subsection}{\bfseries}

\theoremstyle{plain}                       % default
\newtheorem{lemma}{Lemma}[section]
\newtheorem{theorem}[lemma]{Theorem}
\newtheorem{corollary}[lemma]{Corollary}
\newtheorem{remark}[lemma]{Remark}

\newtheorem{proposition}[lemma]{Proposition}

\theoremstyle{remark}

\numberwithin{equation}{section}

\newcommand{\pll}{\kern 0.56em/\kern -0.8em /\kern 0.56em}

\def\Xint#1{\mathchoice
  {\XXint\displaystyle\textstyle{#1}}%
  {\XXint\textstyle\scriptstyle{#1}}%
  {\XXint\scriptstyle\scriptscriptstyle{#1}}%
  {\XXint\scriptscriptstyle\scriptscriptstyle{#1}}%
  \!\int}
\def\XXint#1#2#3{{\setbox0=\hbox{$#1{#2#3}{\int}$}
  \vcenter{\hbox{$#2#3$}}\kern-.5\wd0}}

\def\dashint{\Xint-}
\begin{document}
\allowdisplaybreaks
%\pagestyle{myheadings}
%\markboth{$~$ \hfill {\rm Q. Xu,} \hfill $~$}
%{$~$ \hfill {\rm  } \hfill$~$}
%\author{Li Wang
%%\thanks{Email: lwang10@lzu.edu.cn.}
%\quad Qiang Xu
%\thanks{Corresponding author}
%\thanks{Email: xuqiang09@lzu.edu.cn.}
%\quad Peihao Zhao
%%\thanks{Email: phzhao@lzu.edu.cn.}
%\\
%School of Mathematics and Statistics, Lanzhou University, \\
%Gansu, 730000, PR China.
%\vspace{0.5cm}
%}

\author[a,b]{Li Wang
}

\author[a]{Qiang Xu\thanks{Email: xuq@lzu.edu.cn.}
}

\author[b]{Zhifei Zhang \thanks{Email: zfzhang@math.pku.edu.cn.}
}
\affil[a]{School of Mathematics and Statistics, Lanzhou University, Lanzhou, 730000, China.}

\affil[b]{School of Mathematic Sciences,
Peking University, Peking 100871, China.}

%

%\author{Weiren Zhao
%\thanks{Email: xuqiang@math.pku.edu.cn.}
%\thanks{This work was supported by the National Natural Science Foundation of China (Grant No. 11471147).}

\title{\textbf{Corrector estimates and homogenization errors\\
of unsteady flow ruled by Darcy's law}}

\maketitle
\begin{abstract}
Focusing on Darcy’s law incorporating memory effects, this paper studies non-stationary Stokes equations on perforated domains.
We establish a sharp homogenization error for both velocity and pressure
in terms of the energy norm. The main challenge lies in gauging the boundary layers induced by the incompressibility condition.
To address this, we construct boundary-layer correctors using Bogovskii's operator. Also,
the present work provides detailed regularity estimates for these    correctors, where a significant difficulty arises from the incompatibility between initial and boundary values. The methodologies developed herein hold a great potential for tackling the same issue in other evolutionary models beyond a homogenization setting.

\end{abstract}

\tableofcontents

\section{Introduction}

\subsection{\centering Motivation and main results}
\noindent

We start from introducing the main object of this paper.
Let the ratio of the period to
the overall size of the porous medium be denoted by a parameter $\varepsilon$, which allowed to approach zero, and the porous medium is contained in a bounded domain $\Omega$.
Its fluid part is represented by $\Omega_\varepsilon$,
which is also referred to as a perforated domain. Precisely,
let $\Omega\subset \mathbb{R}^d$ (with $d\geq 2$) be a bounded domain
with $C^{2}$ boundary, and set $Y:=[-\frac{1}{2},\frac{1}{2})^d\cong\mathbb{R}^d/\mathbb{Z}^d$ to be the elementary cell for the lattice $\mathbb{Z}^d$, made of two complementary parts:
the solid part $Y_{s}$ and
the fluid part $Y_{f}$.
Assume that $Y_{s}\subsetneqq Y$ is a connected subset of $Y$ with
$C^{2}$ boundary and a strictly positive Lebesgue
measure in $\mathbb{R}^d$, and we are interest in the case
where $Y_f:=Y\setminus \overline{Y}_s$ is connect set.
We now define the configuration and the perforated domain, respectively.
\begin{equation}\label{RA}
\omega:=\bigcup_{z\in\mathbb{Z}^d}(Y_{f}+z);
\qquad\quad
  \Omega_{\varepsilon} := \Omega\cap \varepsilon\omega =
  \Omega\setminus\bigcup_{k}\varepsilon(\overline{Y}_{s}+z_{k}),
\end{equation}
where $z_k\in \mathbb{Z}^d$ and the union is taken over those $k's$ such that
$\varepsilon(Y+z_k)\subset\Omega$ and
$\text{dist}(\partial\Omega,\partial\Omega_{\varepsilon}
\setminus\partial\Omega)\geq \kappa_0\varepsilon$ with
$\kappa_{0}\geq 2$ (which is not essential). These assumptions
will be utilized without stating in later sections.

The incompressible fluid movement in $\Omega_\varepsilon$
can be described by the unsteady Stokes equations:
for $T>0$,
\begin{equation}\label{pde1.1}
\left\{
\begin{aligned}
\partial_t u_{\varepsilon}-\varepsilon^2
\mu\Delta u_{\varepsilon}+\nabla p_{\varepsilon}
&=&f\qquad&\text{in}\quad \Omega_{\varepsilon}\times(0,T];\\
\nabla\cdot u_{\varepsilon}&=&0\qquad&\text{in}\quad \Omega_{\varepsilon}\times(0,T];\\
u_{\varepsilon}&=&0\qquad&\text{on}\quad \partial\Omega_{\varepsilon}\times (0,T];\\
u_{\varepsilon}|_{t=0}&=&0\qquad&\text{on}\quad \Omega_{\varepsilon}.
\end{aligned}\right.
\end{equation}
We denote by $u_\varepsilon$ and $p_\varepsilon$ the velocity and
pressure of the fluid, respectively, while $f$ represents the density of forces acting
on the fluid, where $u_\varepsilon$ and $f$ are vector-valued functions but
$p_\varepsilon$ is a scalar. The fluid viscosity $\mu$ is a constant and
we assume $\mu=1$ throughout the paper for simplicity.
In terms of the scaling $\varepsilon^2$ of the viscosity in the equations $\eqref{pde1.1}$, it is not a simple change of variable as in the stationary case because the density in front of the inertial
term has been scaled to 1.
The present scaling in $\eqref{pde1.1}$ will precisely lead to a limit problem depending on time in a nonlocal manner (i.e., memory effects),
which is the critical case shown in \cite{Sandrakov97}.

The limit behavior of the equations $\eqref{pde1.1}$ was pioneeringly investigated by J.-L. Lions \cite{Lions81} through a formal asymptotic expansion argument,
demonstrating that the effective equation is governed by Darcy's law with memory. Then, under the geometry assumption $\eqref{RA}$,
G. Allaire \cite{Allaire92} gave a rigorous proof
by using the two-scale convergence method,
while A. Mikeli\'c \cite{Mikelic94} developed a similar result independently.
Also, G. Sandrakov \cite{Sandrakov97} studied the same type models by
considering different scalings on viscosity.

Since the solution $(u_\varepsilon,p_\varepsilon)$ to the equations
$\eqref{pde1.1}$ is not defined in a fixed domain, a mehtod of extending the solution to the whole domain $\Omega$ is required to state the homogenization result.
As in the stationary case, the velocity
$u_\varepsilon$ naturally owns a zero-extension,
denoted by $\tilde{u}_\varepsilon$.
Let $Y_s^k := Y_s+z_k$ with $z_k\in \mathbb{Z}^d$. Let  $\tilde{p}_{\varepsilon}$ be the extension of $p_{\varepsilon}$ to $\Omega$, given as follows:
\begin{equation}\label{eq1.1}
\tilde{p}_{\varepsilon}(x,t):=
\left\{
\begin{aligned}
&p_{\varepsilon}(x,t)&\quad&\text{if}\quad x\in\Omega_{\varepsilon};\\
&\dashint_{\varepsilon(Y_f+z_{k})}p_{\varepsilon}(\cdot,t)&\quad&\text{if}
\quad x\in\varepsilon Y_{s}^k \quad\text{and}\quad
\varepsilon Y^k_s\subset\Omega \text{~for~some~}z_{k}\in \mathbb{Z}^d.
\end{aligned}\right.
\end{equation}

We now present the main results in the qualitative homogenization theory.

\begin{theorem}[homogenization theorem \cite{Allaire92,Allaire-Mikelic97,Lions81,Mikelic94,Sandrakov97}]
\label{thm3}
Let $0<T<\infty$, and $d\geq 2$. There exists an extension $(\tilde{u}_\varepsilon,\tilde{p}_\varepsilon)$ of the solution $(u_\varepsilon,p_\varepsilon)$ to $\eqref{pde1.1}$, which
weakly converges in $L^2(0,T;L^2(\Omega)^d)\times
L^2(0,T;L^2(\Omega)/\mathbb{R})$ to the unique solution
$(u_0,p_0)$ of the homogenized problem (known as Darcy's law with memory):
\begin{equation}\label{pde1.2}
\left\{
\begin{aligned}
&u_0(x,t)=\int_{0}^tds A(t-s)(f-\nabla p_0)(x,s)\qquad&\text{in}&\quad\Omega\times(0,T);\\
&\nabla\cdot u_0=0\qquad&\text{in}&\quad\Omega\times(0,T);\\
&u_0\cdot \vec{n}=0\qquad&\text{on}&\quad\partial\Omega\times (0,T),
\end{aligned}
\right.
\end{equation}
where $\vec{n}$ is the unit outward normal vector of $\partial\Omega$.
The quantity $A=(A_{ij})$ is a symmetry, positive defined,
and time-dependent
permeability tensor,  determined by
\begin{equation}\label{pde2.2}
  A_{ij}(t)=\int_{Y_f}dy W_j(y,t)\cdot e_i,
\end{equation}
where $e_{i}=(0,\cdots,1,\cdots,0)$ with 1 in the $i^{th}$ place,
and $(W_{j},\pi_{j})\in
L^2(0,T;H^{1}_{\text{per}}(Y_f)^d)\times
H^{-1}
({0,T;L^2_{\text{per}}(Y_f)}/\mathbb{R})$ is the corrector associated with $e_j$ by the following equations:
\begin{equation}\label{pde2.1}
\left\{
\begin{aligned}
\partial_t W_j-\Delta W_{j}+
\nabla \pi_{j}&=0&\qquad&\text{in}\quad
\omega\times(0,T];\\
\nabla\cdot W_j&=0&\qquad&\text{in}\quad
\omega\times(0,T];\\
W_j&=0&\qquad&\text{on}\quad\partial\omega
\times(0,T],\\
W_j|_{t=0}&=e_j&\qquad&\text{on}\quad\omega.
\end{aligned}\right.
\end{equation}
Moreover, there holds the strong convergence 
\begin{equation}\label{pri:1.2}
\begin{aligned}
& \int_{0}^{T}dt\big\|\tilde{u}_\varepsilon(\cdot,t)
    -\int_{0}^{t}ds W(\cdot/\varepsilon,t-s)(f-\nabla p_0)(\cdot,s)\big\|_{L^2(\Omega)}^2
  \to 0,\qquad\text{as}\quad \varepsilon\to 0.
\end{aligned}
\end{equation}
\end{theorem}

\medskip

The main purpose of this paper is to quantify the strong convergence  as described in $\eqref{pri:1.2}$. For the reader's convenience, we introduce the following notations:
\begin{equation*}
A\ast f(\cdot,t) := \int_{0}^{t}ds A(t-s)f(\cdot,s);\qquad
\Omega_{\varepsilon,T}:=\Omega_{\varepsilon}\times(0,T].
\end{equation*}

\begin{theorem}[error estimates]\label{thm1}
Let $0<T<\infty$, and $d=2$ or $3$. Assume that
$\Omega\subset\mathbb{R}^d$ is a bounded $C^{2}$ domain, and
the perforated one $\Omega_\varepsilon$
satisfies the geometrical hypothesis $\eqref{RA}$.
Given $f\in L^2(0,T;C^{1,\frac{1}{2}}
(\bar{\Omega})^d)$,
let $(u_{\varepsilon},p_{\varepsilon})
\in L^2(0,T;H^1_0(\Omega_{\varepsilon})^d)
\times L^2(0,T;L^2(\Omega_{\varepsilon})/\mathbb{R})$
be the weak solution of \eqref{pde1.1}.
Then, we have
\begin{equation}\label{es1.1}
\begin{aligned}
\|u_{\varepsilon}-W({\cdot}/{\varepsilon})
&\ast
(f-\nabla p_{0})\|_{
L^{2}(\Omega_{\varepsilon, T})}\\
&+\|\varepsilon \nabla u_{\varepsilon}-
\nabla W({\cdot}/{\varepsilon})\ast(f-\nabla p_0)
\|_{L^2(\Omega_{\varepsilon, T})}
\leq C\varepsilon^{1/2}
\|f\|_{L^2(0,T;C^{1,1/2}(\bar\Omega))},
\end{aligned}
\end{equation}
and there further holds
\begin{equation}\label{pri:1.3}
\begin{aligned}
\|\partial_t u_\varepsilon
-\partial_tW(\cdot/\varepsilon)*(f-
\nabla p_0)
&-W(0)(f
-\nabla p_0)\|_{L^2(\Omega_{\varepsilon,T})}\\
&+\inf_{c\in\mathbb{R}^d}\|p_\varepsilon-p_0
 -c\|_{L^2(\Omega_{\varepsilon,T})}
\leq C\varepsilon^{1/2}
\|f\|_{L^2(0,T;C^{1,1/2}(\bar\Omega))},
\end{aligned}
\end{equation}
where the constant $C$ depends only on $d$, $T$, $|Y_f|$, and the characters of $\Omega$ and $Y_s$.
\end{theorem}

\begin{corollary}\label{cor:2}
Assume the same conditions as in Theorem $\ref{thm1}$. Let $\tilde{p}_\varepsilon$ be
the extension of $p_\varepsilon$ as
in $\eqref{eq1.1}$.
Then, there holds the following estimate
\begin{equation}\label{pri:1.4}
   \inf_{c\in\mathbb{R}^d}\Big(\int_{0}^{T}dt
   \int_{\Omega}|\tilde{p}_{\varepsilon}(\cdot,t)-p_0(\cdot,t)-c|^2\Big)^{1/2}
   \leq C \varepsilon^{1/2}\|f\|_{L^2(0,T;C^{1,1/2}(\bar\Omega))},
\end{equation}
where the constant $C$ depends only on $d$, $T$, $|Y_f|$, and the characters of $\Omega$ and $Y_s$.
\end{corollary}

Roughly speaking, the primary contribution of this paper is the derivation of the following asymptotic expansion for the velocity:
\begin{equation}\label{expan:1}
u_\varepsilon \approx \underbrace{W(\cdot/\varepsilon)\ast (f-\nabla p_0)}_{\text{oscillating~part}}
~~~+
\underbrace{\hat{\xi}_\varepsilon + \hat{\eta}_\varepsilon}_{\text{boundary~layer part}}
~~~+ ~~~O(\varepsilon).
\end{equation}
While the form of the above expansion should be well-known to experts in qualitative analysis (see e.g. \cite{Lions81}), how to specifically construct the boundary layer part
and make a well quantitative analysis remains a significant challenge in this field. In this work,  we have successfully constructed the boundary-layer correctors, denoted by  $(\hat{\xi}_\varepsilon,\hat{\eta}_\varepsilon)$, for the boundary layer part.
Furthermore, we derived that
\begin{equation}\label{k-2}
\|\hat{\xi}_\varepsilon\|_{L^2(\Omega_{\varepsilon,T})}
+\varepsilon\|\nabla\hat{\xi}_\varepsilon\|_{L^2(\Omega_{\varepsilon,T})}
+ \|\hat{\eta}_\varepsilon\|_{L^2(\Omega_{\varepsilon,T})}
+\varepsilon\|\nabla\hat{\eta}_\varepsilon\|_{L^2(\Omega_{\varepsilon,T})} = O(\sqrt{\varepsilon}),
\end{equation}
which are optimal even for the stationary case. Consequently,
with respect to the convergence rate,
the error estimates $\eqref{es1.1}$ and
$\eqref{pri:1.3}$ are sharp, given the $O(\sqrt{\varepsilon})$ loss
of the boundary-layer correctors.

Beyond the challenges posed by the boundary layer in spatial variables, Theorem 1.1 has highlighted the considerable difficulties associated with the temporal variable when considering higher-order regularity estimates.
\begin{itemize}
  \item[Q1.] Given $f\in L^2(0,T;C^{1,1/2}(\bar\Omega)^d)$, how to establish the well-posedness of the equations
  $\eqref{pde1.2}$?

Within the framework of Hilbert spaces, A. Mikeli\'c \cite{Mikelic91}
and G. Sandrakov \cite{Sandrakov97} developed well-posedness results  using
Laplace's transform methods. However, these methods are difficult to be applied to general Bochner spaces. Inspired by J.-L. Lions \cite[pp.170]{Lions81},
this work establishes a rigorous well-posedness theory
for the integral-differential equations $\eqref{pde1.2}$ as follows.

\begin{proposition}[well-posedness for Darcy's law with memory]\label{P:1}
Let $1\leq q\leq \infty$, $m\geq 1$ and $\alpha\in(0,1)$.
Given $0<T<\infty$ and $d=2$ or $3$, suppose that
$f\in L^q(0,T;C^{m,\alpha}(\bar{\Omega}
)^d)$ and $\partial\Omega\in C^{m+1,\alpha}$.
Then, there exists a unique $p_0\in L^q(0,T;C^{m+1,\alpha}(\bar{\Omega}))$
to the equations \eqref{pde1.2}
with the condition $\int_{\Omega}p_0(\cdot,t)=0$ for a.e. $t\geq 0$. Moreover, let
the permeability tensor $A$ be given as in $\eqref{pde2.2}$.
Then, for any $0<\beta<(2/21)$, one can derive that
\begin{equation}\label{pri:3.1-1}
|\partial_tA|\in L^{1+\beta}(0,T),
\end{equation}
As a result, there holds the regularity estimate
\begin{equation}\label{pri:3.1}
\|p_0\|_{L^q(0,T;C^{m+1,\alpha}(\bar{\Omega}))}\leq C
\|f\|_{L^q(0,T;C^{m,\alpha}(\bar{\Omega}))},
\end{equation}
where the constant $C$ depends on $d$, $|Y_f|$, $\Omega$, and $T$.
\end{proposition}
To the best of the authors' knowledge, the most recent and relevant results (see \cite[pp.127]{Sandrakov97}) indicate that
$|\partial_tA|\in L^1(0,T)$; and
the estimate $\eqref{pri:3.1}$ is only valid for $L^2(0,T;H^1(\Omega))$.
Although the upper bound of $\beta$ may not be optimal,
the aforementioned estimate $\eqref{pri:3.1-1}$, combined with the definition of $A$ given in $\eqref{pde2.2}$, has already shown that a more refined regularity estimate for the correctors is required. This naturally leads to another key issue addressed in this article.

  \item[Q2.] How can one derive a refined regularity estimate for $(W,\pi)$ to
  $\eqref{pde2.1}$ \textbf{without compatibility conditions}
  between the initial and boundary data?

  In fact, the methodology employed for constructing boundary-layer correctors renders us highly dependent on obtaining more accurate estimates for the corrector $(W,\pi)$. To the best authors' knowledge, there are currently no established theoretical results available for reference concerning the issue of incompatible initial and boundary conditions. This article provides a more in-depth   investigation into this problem, employing three key tools:
  \begin{itemize}
    \item An interior Caccioppoli-type inequality (as well as,  interior higher-order regularity estimates);
    \item Decay estimates of Stokes semigroup;
    \item Interpolations.
  \end{itemize}
  In this regard, the paper achieves the crucial improvement\footnote{Instead of establishing time-decay estimates
as time goes to infinity, we are interested in the ``decay'' as time approaches zero.
In this sense, we regard the estimate $\eqref{k-1}$ as an improvement over original semigroup estimates.\label{footnote1}} in the semigroup estimates as follows: for any $(1/3)<\theta\leq (1/2)$, there exists a constant $C_\theta$ such that
\begin{equation}\label{k-1}
\|\nabla W_j(\cdot, t)\|_{L^2(Y_f)}
\leq C_{\theta} t^{-\theta},
\qquad\forall t>0.
\end{equation}
The time weight $t^\alpha$, introduced below, is utilized to characterize the singularity of the velocity field and pressure near the initial moment. The intensity of this singularity is quantified by the exponent $\alpha$ of the weight function. Leveraging the estimate $\eqref{k-1}$, we have established
\begin{equation}\label{k-3}
\|t^{\alpha}\partial_t W_j\|_{L^2(0,T;L^{2}(Y_f))}
+\|t^{\alpha}\pi_j\|_{L^2(0,T;L^{2}(Y_f)/\mathbb{R})}\leq C_{\alpha},
\qquad  \forall\alpha>(1/3).
\end{equation}
We are uncertain whether the lower bound of $(1/3)$ is optimal; however, it suffices for establishing the desired results in this paper, including the previously stated estimates $\eqref{k-2},\eqref{pri:3.1-1}$, and $\eqref{pri:3.1}$. It is important to note that the technique developed herein primarily pertains to the regularity theory of PDEs, which constitutes an area of independent interest in the current literature.
\end{itemize}

\begin{remark}
\emph{The fundamental reason for restricting the dimension to 2D or 3D in this paper is based on the technique of the corrector's estimate. As mentioned earlier, we are going to employ the interior Caccioppoli-type estimates, which utilizes the basic properties of the curl operator in three-dimensional case (see Lemma $\ref{lemma3.3}$). The high-dimensional generalization of these basic properties involves the redefinition of the curl operator.
We do not intend to discuss this issue in the present work. }
\end{remark}

\subsection{\centering Relation to other works}

The pioneer literature on homogenization errors was contributed by
E. Maru\v{s}i\'{c}-Paloka, A. Mikeli\'{c}, L. Paoli  \cite{Marusi-Mikelic96,Mikelic-Paoli99}, where they obtained an $O(\varepsilon^{1/6})$-error
for steady Stokes problems and an $O(\varepsilon)$-error for non-stationary incompressible Euler's equations,
respectively. These results were obtained in dimension two.
Recently, the quantitative estimates on homogenization of Stokes equations  in porous medias have seen rapid development (see e.g. \cite{Balazi-Allaire-Omnes24,Jankowiak-Lozinski24,Jing20,Jing-Lu-Prange24,
Lu21,Shen20,Shen22,Shen23}).
Regarding the optimal homogenization error for Darcy's law, the main challenge lies in accurately gauging the boundary layers created by the incompressibility condition combined with the discrepancy of boundary values between the solution and the leading term in the related asymptotic expansion $\eqref{expan:1}$.
To date, apart from our approach, there are two other methods to address this difficulty, as outlined below:
\begin{itemize}
  \item[1.] Z. Shen \cite{Shen20} derived the optimal error estimate for steady Stokes
systems for $d\geq 2$. The primary approach involves reducing all challenges to a boundary value problem, followed by defining
boundary correctors with tangential boundary data and normal boundary ones, respectively, to achieve a sharp estimate. The analysis therein  relies on some advanced concepts such as non-tangential maximal function and Rellich estimates.

  \item[2.]
  The approach proposed by
  L. Balazi, G. Allaire, and P. Omnes \cite{Balazi-Allaire-Omnes24}
  is also noteworthy\footnote{G. Jankowiak, A. Lozinski \cite{Jankowiak-Lozinski24} developed a similar strategy in dimension two.}.  In essence, their method involves first correcting the discrepancy between the boundary values of the solution and the leading term. Subsequently, the problem is reduced to correcting the error term to satisfy the incompressibility condition.
  Regarding the techniques in their framework, they smartly concealed a cut-off function within the curl operator, based upon
  the result stated in \cite[Lemma 4.10]{Balazi-Allaire-Omnes24}, previously developed in \cite{Dacorogna02}.
\end{itemize}

E. Maru\v{s}i\'{c}-Paloka, A. Mikeli\'{c} in  \cite[pp.2]{Marusi-Mikelic96}
ever noted that ``the cut-off argument
does not work for the Stokes and Navier-Stokes systems because of the incompressibility condition, which creates a boundary
layer in the neighbourhood of $\partial\Omega$ destroying the estimate.''
By introducing the concept of radial cut-off function,
this work directly addresses this issue. Unlike the method
developed by L. Balazi et al. \cite{Balazi-Allaire-Omnes24,Jankowiak-Lozinski24},
our strategy reduces the problem to constructing boundary-layer correctors associated with Bogovskii's operator (see Subsection \ref{subsec:1.2}).

Regarding the geometric modeling of the porous medium, the isolated obstacles condition $\text{dist}(\partial Y,\partial Y_s) > 0$ appeared in \cite{Shen20}, is not necessary in the present work.  For a detailed description of the notion of isolated (or connected) obstacles, we refer the reader to \cite[Section 2]{Balazi-Allaire-Omnes24} or
\cite[Section 14]{Chechkin-Piatnitski-Shamaev2007} for the details.

In the end, we would like to highlight two points. Firstly, the work of L. Balazi et al. \cite{Balazi-Allaire-Omnes24,Jankowiak-Lozinski24} convinced us that the idea of introducing a radial cut-off function was a viable approach. This inspired us to revisit and significantly improve our previous work \cite{Wang-Xu-Zhang22} (see Remark $\ref{remark:4.2}$). Secondly, it is important to note that our work is not a continuation of the research presented in  \cite{Balazi-Allaire-Omnes24,Jankowiak-Lozinski24,Shen20}
in any sense,
since our approach is independent of the most important tool mentioned above in their analysis. Additionally, the idea of the radial cut-off function has already been applied in recent studies, such as \cite{Jing-Lu-Prange24},
demonstrating its potential for broader applications.

\subsection{\centering Structure of the paper}

In Section $\ref{sec1}$, we introduce the main ideas leading to the conclusion stated in Theorem $\ref{thm1}$. The detailed motivation for boundary-layer correctors is presented in Subsection $\ref{subsec:1.2}$. To provide readers with an overview of the proof for Theorem $\ref{thm1}$, we outline the main steps and key estimates required in Subsection $\ref{subsec:1.4}$.

In Section $\ref{sec2}$, we present the main results on correctors in Proposition $\ref{prop2.1}$. To achieve the required estimates, we employ two rounds of regularity enhancement. Subsection $\ref{subsec:3.1}$ focuses on the first round of regularity improvement, as detailed in $\eqref{k-1}$; Subsection $\ref{subsec:3.2}$ addresses the second round. Subsection $\ref{subsec:3.3}$ mainly provides the proof of weighted estimates $\eqref{k-3}$.

In Section \ref{sec4}, the estimates of boundary-layer correctors are presented in Propositions \ref{lemma3.6-1}, \ref{lemma3.6-2}, and \ref{lemma3.6-3}. Subsection \ref{subsec:4.2} establishes the existence of the desired radial cut-off function. Due to the truncation, we categorize boundary-layer correctors into co-layer type and layer type, with the corresponding proofs provided in Subsections \ref{subsec:4.3} and \ref{subsec:4.4}, respectively.

In Section $\ref{sec5}$, we present the proof of Theorem $\ref{thm1}$. Section $\ref{sec3*}$ serves as an appendix to this paper, providing the proof of Proposition $\ref{P:1}$ for the reader's convenience.

\subsection{\centering Notations}\label{notation}

\begin{enumerate}
\item Notation related to domains

We denote the co-layer part of $\Omega$ by $\Sigma_{\varepsilon}:=\{x\in\Omega,\text{dist}
(x,\partial\Omega)\geq \varepsilon\}$, while the region
$\Omega\setminus\Sigma_\varepsilon$ is known as
the layer part of $\Omega$.
There are two type cut-off functions used throughout the paper: One is
the so-called radial cut-off function $\psi_\varepsilon$
(defined in Lemma $\ref{lemma:cut-off}$); The other
is a general cut-off function, denoted by $\varphi_{\varepsilon}$, satisfying that
$\varphi_{\varepsilon}\in C^{1}_{0}(\Omega)$ and $\varphi_\varepsilon = 1$ on
$\Sigma_{\varepsilon/2}$ with
$|\nabla \varphi_{\varepsilon}|\lesssim 1/\varepsilon$.
Let supp$(f)$ represent the support of $f$, and we
denote supp$(\nabla\psi_\varepsilon)$ by $O_\varepsilon$.

\item Notation for special quantities

\begin{itemize}
  \item Let $F:=f-\nabla p_0$ be given with $p_0$ satisfying \eqref{pde1.2}, and
$G:=S_{\delta}(\varphi_{\varepsilon}F)$ with $\delta=\frac{\varepsilon}{4}$,
where $S_{\delta}$ is a smoothing operator\footnote{
Fix a nonnegative function $\zeta\in C_0^\infty(B(0,1/2))$ with
$\int dx\zeta(x) = 1$. For any $f\in L^p(\mathbb{R}^d)$ with $1\leq p<\infty$, we define the smoothing operator as
$S_\delta(f)(x):=\int dy f(x-y)\zeta_\delta(y)$,
where $\zeta_\delta(y)=\delta^{-d}
\zeta(y/\delta)$.
}, and $\varphi_{\varepsilon}$ is a general cut-off function defined above.
  \item Let $\phi^{\varepsilon}(x,t):=\phi(\frac{x}{\varepsilon},t)$,
$W^{\varepsilon}(x,t):=W(\frac{x}{\varepsilon},t)$, and
$\pi^\varepsilon(x,t):=\pi(\frac{x}{\varepsilon},t)$,
in which correctors $(W,\pi)$ and $\phi$ are defined in
the equations $\eqref{pde2.1}$ and $\eqref{divc1*}$, respectively;
We also introduce two important quantities:
\begin{equation}\label{pde3.8}
  \begin{aligned}
  J_1&:=\nabla \psi_{\varepsilon}\cdot\big[(W^{\varepsilon}-A)\ast G\big];\\
  J_2&:=\nabla \psi_{\varepsilon}\cdot (A\ast G)
  +\varepsilon\nabla \psi_{\varepsilon}\cdot
  \big(\phi^{\varepsilon}
  \ast_{2} \partial G\big)
  +\psi_{\varepsilon}\frac{A}{|Y_f|}\ast_{2} \partial G
  +\varepsilon\psi_{\varepsilon}\phi^{\varepsilon}
  \ast_{3} \partial^2 G.
  \end{aligned}
\end{equation}
(We recommend that readers initially skip the derivation of the
above expression and treat it merely as a notation. After reviewing Subsection $\ref{subsec:1.2}$ and Remark $\ref{remark:3.1}$, readers will naturally encounter the source of the expression at the beginning of Section $\ref{sec4}$.)
\end{itemize}

\item Notation for derivatives

Spatial derivatives:
  $\nabla v = (\partial_1 v, \cdots, \partial_d v)$ (or denoted by $\partial v$) is the gradient of $v$, where
  $\partial_i v = \partial v /\partial x_i$ denotes the
  $i^{\text{th}}$ derivative of $v$.
  $\nabla^2 v$ (or $\partial^2 v$)  denotes the Hessian matrix of $v$;
  $\nabla\cdot \vec{u}=\sum_{i=1}^d \partial_i u_i$
  denotes the divergence of $\vec{u}$, where
  $\vec{u} = (u_1,\cdots,u_d)$ is a vector-valued function.
Also, $\text{curl~}\vec{u}$ is identified with the 2-tensor
$\omega_{ij}=\partial_iu_j-\partial_j u_i$ for all $d\geq 2$. In particular, we have the following convention
\begin{equation*}
  \text{curl~}\vec{u}=\nabla\times \vec{u} \qquad \text{as}\quad d=3.
\end{equation*}
Temporal  derivatives: $\partial_t v:= \partial v/\partial t$ represents
the derivative with respect to the temporal variable. 
In the Appendix, the time derivative of $A$ is also denoted as 
$A'$.

\item Notation for spaces

We mention that this paper merely involves some simple Bochner spaces whose definition
is standard (see e.g. \cite[Subsection 5.9.2]{Evans10}); Note that $X_{per}$ represents its element being a periodic object,
where $X$ could be any Sobolev or H\"older space;
Also, $X^d$ indicates that its collected elements are d-dimensional, and $X/K$ represents a quotient space (see e.g. \cite{Chechkin-Piatnitski-Shamaev2007,Shen18}).

\item Notation for estimates

$\lesssim$ and $\gtrsim$ stand for $\leq$ and $\geq$ up to a multiplicative constant, which may depend on some given parameters introduced in the paper, but never on $\varepsilon$;  We write $\sim$ when both $\lesssim$ and $\gtrsim$ hold; We use $\ll$ instead of
$\lesssim$ when the inverse of multiplicative constant
is much larger than 1.

\item Notation for convolutions

Let $a,b$ be vectors, $A$ be a matrix, and $C,D$ be tensors (higher than second order); For the ease of the statement, we introduce the following notations:
\begin{equation}\label{notation:4.2}
\begin{aligned}
&a\ast_{1} b(t) := \int_{0}^{t}ds a(t-s)\cdot b(s);
&\quad&
A\ast b(t) := \int_{0}^{t}ds A(t-s)b(s); \\
&C\ast_{2} A(t) := \int_{0}^{t}ds C(t-s):A(s);
&\quad&
C\ast_{3} D(t) := \int_{0}^{t}ds C_{ijk}(t-s)D_{ijk}(s),
\end{aligned}
\end{equation}
where the notation ``$:$'' represents the tensor's inner product of second order, and Einstein's summation convention for repeated indices is used throughout.
Together with the convention on derivatives\footnote{
If the components of the gradient are involved in the inner product of a tensor
and the convolution operation at the same time,
we will use $\partial$ instead of $\nabla$ to stress this point,
such as shown in $\eqref{notation:4.1}$.}, we list the following quantities frequently appeared in this paper:
\begin{equation}\label{notation:4.1}
\begin{aligned}
&\big(W^\varepsilon\ast G\big)(x,t)
:= \int_{0}^{t}dsW(x/\varepsilon,t-s)G(x,s);
&\qquad&(\text{vector})\\
&\big(\pi^{\varepsilon}\ast_{1} G\big)(x,t)
:= \int_{0}^{t}ds\pi_{j}(x/\varepsilon,t-s) G_j(x,s); &\qquad&(\text{scalar})\\
&\big(\phi^\varepsilon\ast_2\partial G\big)(x,t)
:= \int_{0}^{t}ds\phi_{k,j}(x/\varepsilon,t-s)\partial_kG_j(x,s);
&\qquad&(\text{vector})\\
&\big(\phi^{\varepsilon}\ast_3\partial^2G\big)(x,t)
:=\int_{0}^{t}ds\phi_{ik,j}(x/\varepsilon,t-s)\partial_{ik}^2G_j(x,s);
&\qquad&(\text{scalar})\\
&\big(\phi^\varepsilon\ast_2\nabla\partial G\big)(x,t)
:=\int_{0}^{t}ds\phi_{k,j}(x/\varepsilon,t-s)\otimes\nabla\partial_kG_j(x,s),
&\qquad&(\text{matrix})
\end{aligned}
\end{equation}
where the notation ``$\otimes$'' represents the tensor product. 

\end{enumerate}

\section{Sketch of the proof}\label{sec1}

\subsection{\centering On the error term (main ideas)}\label{subsec:1.2}

\noindent
This subsection clarifies the rationale for the form of the expansion in $\eqref{expan:1}$, particularly from an error analysis perspective, which is fundamental to the paper. In particular, in the following statements we simply treat the time variable as a parameter.

%This subsection is committed to explaining the reason why the composition of the expansion is of the form  from the error side, which is fundamentally important throughout the paper.
%
%This subsection aims to elucidate the rationale behind the specific form of the expansion as presented in equation (1), focusing on the error analysis perspective, which is of fundamental importance throughout the paper.

\medskip

Firstly, the qualitative result $\eqref{pri:1.2}$
suggests that
the error term should be the form of
\begin{equation}\label{expan:2}
  w_\varepsilon^{(1)} = u_\varepsilon - W^{\varepsilon}*F.
\end{equation}
(See the notations $F$ and $G$ in Subsection $\ref{notation}$.)
However, this kind of error term leads to the following inhomogeneous conditions:
%which are \textbf{never small}\footnote{The bold word ``small''
%in this subsection refers to the quantity related to $\varepsilon$.}.
\begin{equation*}
\begin{aligned}
\nabla\cdot w_\varepsilon^{(1)}&=-W^\varepsilon*_2\partial F
\text{\quad in~} \Omega_{\varepsilon}\times(0,T);\\
w_\varepsilon^{(1)} &=-W^\varepsilon* F
\text{\qquad on~} \partial\Omega\times(0,T), \qquad T>0.
\end{aligned}
\end{equation*}

\medskip

Secondly, to correct the inhomogeneous conditions above,
a more precise formulation of $\eqref{expan:1}$ is as follows:
%\pagestyle{empty}
%\begin{figure}[H]
%  \centering

\tikzstyle{every picture}+=[remember picture]

% By default all math in TikZ nodes are set in inline mode. Change this to
% displaystyle so that we don't get small fractions.
\everymath{\displaystyle}

\begin{itemize}
    \item [] \footnotesize zero-order expansion term
        \tikz\node [fill=blue!20,draw,circle] (n1) {};
\hspace{6cm} first-order expansion term
        \tikz\node [fill=green!20,draw,circle] (n3) {};
\end{itemize}
% Below we mix an ordinary equation with TikZ nodes. Note that we have to
% adjust the baseline of the nodes to get proper alignment with the rest of
% the equation.
\begin{equation*}
w_\varepsilon = u_{\varepsilon}-
        \tikz[baseline]{
            \node[fill=blue!20,anchor=base] (t1)
            {$\psi_{\varepsilon}W^{\varepsilon}\ast G$};
        } +
        \tikz[baseline]{
            \node[fill=red!20,anchor=base] (t2)
            {$\hat{\xi_{\varepsilon}}+\hat{\eta_{\varepsilon}}$};
        } -
        \tikz[baseline]{
            \node[fill=green!20,anchor=base] (t3)
            {$\varepsilon\psi_{\varepsilon}\phi^\varepsilon \ast_{2}\partial G$.};
        }
\end{equation*}
\begin{itemize}
    \item [] \footnotesize boundary layer terms
        \tikz\node [fill=red!20,draw,circle] (n2) {};
\end{itemize}
% Now it's time to draw some edges between the global nodes. Note that we
% have to apply the 'overlay' style.
\begin{tikzpicture}[overlay]
        \path[->] (n1) edge [bend left] (t1);
        \path[->] (n2) edge [bend right] (t2);
        %\path[->] (n3) edge [out=0, in=-90] (t3);
        \path[->] (n3) edge [bend left] (t3);
\end{tikzpicture}
The analogous pattern was first identified by J.-L. Lions \cite[pp. 147]{Lions81} for stationary Stokes equations in a specific domain. We will elucidate this error expansion through a structured ``questions $\&$ answers'' format.

\everymath{}

\begin{itemize}
  \item Why do we introduce the \emph{radial} cut-off function $\psi_\varepsilon$?
   \begin{itemize}
     \item To homogenize the inhomogeneous boundary condition induced  by $W^{\varepsilon}\ast G+\varepsilon
\phi^\varepsilon\ast_{2} \partial G$ on $\partial\Omega$;
     \item By virtue of the property: $\nabla \psi_{\varepsilon}=-|\nabla \psi_{\varepsilon}|\vec{n}$ near $\partial\Omega$, we can derive that
\begin{equation*}
\|\nabla\psi_\varepsilon \cdot u_0\|_{L^2(\Omega)} = O(1)
\end{equation*}
from the boundary condition $\vec{n}\cdot u_0 = 0$ on $\partial\Omega$. The advantage of introducing $\psi_\varepsilon$
becomes evident when compared to a general cut-off function $\varphi_\varepsilon$:
\begin{equation*}
  \|\nabla\varphi_\varepsilon \cdot u_0\|_{L^2(\Omega)}
  =O(\varepsilon^{-1/2})\qquad \rightsquigarrow\qquad
  \|\nabla\psi_\varepsilon \cdot u_0\|_{L^2(\Omega)}=O(1)
\end{equation*}
This structural benefit has been crucial in boundary-layer corrector estimates  (see Lemma $\ref{lemma:5.1}$).
   \end{itemize}
  \item Why do we introduce the corrector of Bogovskii's operator $\phi$?
  \begin{itemize}
    \item Roughly speaking, it was used to correct inhomogeneous divergence part of $w_\varepsilon^{(1)}$. Using the divergence-free condition $A*_2\partial F = \nabla\cdot u_0 = 0$ in $\Omega$, we have
\begin{equation}\label{k-4}
 \nabla\cdot w_\varepsilon^{(1)}
 = -W^\varepsilon*_2\partial F
 = \big(|Y_f|^{-1}A -W^\varepsilon\big)*_2\partial F.
\end{equation}
This leads to the following equation:
\begin{equation}\label{divc1*}
\left\{
\begin{aligned}
\nabla\cdot\phi &=-W
+|Y_f|^{-1} A &\qquad\text{in}&\quad\omega\times(0,T);\\
\phi&=0&\qquad\text{on}&\quad\partial\omega\times(0,T).
\end{aligned}
\right.
\end{equation}
Then, one can substitute the right-hand side of $\eqref{k-4}$ for
$\varepsilon(\nabla\cdot\phi)^\varepsilon*_{2}\partial F$, further deriving that
\begin{equation*}
  \nabla\cdot \big( w_\varepsilon^{(1)}
  -\varepsilon
  \phi^\varepsilon*_{2}\partial F \big) =
  -\varepsilon\phi^\varepsilon*_{2}\partial^2 F.
\end{equation*}
Theoretically, this expansion can provide a first-order convergence rate. But for the sake of simplicity, the above formula does not take the cut-off function into account. A more detailed discussion on this is scheduled for Remark $\ref{remark:3.1}$, where we will further touch on the concept of the flux corrector.
%
%
%    In view of the equations $\eqref{divc1}$, it plays an important role in the following improvement
%\begin{equation*}
%\text{\textbf{never~small}!}
%\qquad
%  W^\varepsilon*_2\partial G
%\quad \rightsquigarrow \quad
%\frac{A}{|Y_f|}*_2 \partial G
%= \frac{A}{|Y_f|}*_2 \partial (G - F)
%%\qquad \text{\textbf{small}!}
%\end{equation*}
%by
  \end{itemize}
  \item Why do we introduce boundary-layer correctors associated with Bogovskii's  operator?
  \begin{itemize}
    \item To see this, let $w_\varepsilon^{(2)} = u_{\varepsilon}-\psi_{\varepsilon}
(W^{\varepsilon}\ast G+\varepsilon
\phi^\varepsilon\ast_{2} \partial G)$. There holds
\begin{equation}\label{pde:B}
\left\{\begin{aligned}
-\nabla\cdot w_\varepsilon^{(2)}&=J_1+J_2
&~&\text{\quad in~} \Omega_{\varepsilon}\times(0,T); \\
w_\varepsilon^{(2)} &=0
&~&\text{\quad on~} \partial\Omega_\varepsilon\times(0,T),
\end{aligned}\right.
\end{equation}
where $J_1$ and $J_2$ are defined in $\eqref{pde3.8}$;
    \item \textbf{The crucial idea} is to introduce a magical quantity
\begin{equation*}
   \sum_{i} (\dashint_{O_{\varepsilon}^i}J_1)
1_{O_{\varepsilon}^i},
\end{equation*}
where $1_{O_{\varepsilon}^i}$ is the indicator function of $O_{\varepsilon}^i$,
and $\{O_{\varepsilon}^i\}$ is a family of the non-overlap subsets of $O_\varepsilon =\text{supp}(J_1)$, such that
\begin{equation}\label{k-5}
\begin{aligned}
&O_\varepsilon = \bigcup_{i}O_\varepsilon^{i}
\quad\text{and}\quad
|O_\varepsilon^i| \sim \varepsilon^d,
\end{aligned}
\end{equation}
in which $O_\varepsilon^{i}$ is approximately an $d$-dimensional cube obtained by cutting the $O_\varepsilon$ in the normal direction $\vec{n}$ associated with $\partial\Omega$ (see Remark $\ref{remark:4.1}$).
Then, we have
\begin{equation*}
  J_1+J_2 = \underbrace{J_1 - \sum_{i} (\dashint_{O_{\varepsilon}^i}J_1)
1_{O_{\varepsilon}^i}}_{\Pi}
+ \underbrace{\sum_{i} (\dashint_{O_{\varepsilon}^i}J_1)
1_{O_{\varepsilon}^i} + J_2}_{\mathcal{H}}.
\end{equation*}
Thereupon, it is possible to split the equation $\eqref{pde:B}$
to have some meaningful estimates, since $\Pi$ and $\mathcal{H}$ satisfy the compatibility conditions:
$\int_{O_\varepsilon}\Pi = 0$ and
$\int_{\Omega_\varepsilon}\mathcal{H} = 0$, respectively;

    \item Taking into account of temporal variable,
    for any $t\geq 0$, we first construct
solutions $(\xi_{\varepsilon},\eta_{\varepsilon})$ to
\begin{equation}\label{pde:C}
(1)\left\{\begin{aligned}
 \nabla\cdot \xi_{\varepsilon}(\cdot,t) &= \partial_t\mathcal{H}(\cdot,t) &~&\text{in}~\Omega_\varepsilon;\\
  \xi_{\varepsilon} &= 0 &~&\text{in}~\partial \Omega_\varepsilon,
\end{aligned}\right.
\quad \text{and} \quad
(2)\left\{\begin{aligned}
 \nabla\cdot \eta_{\varepsilon}(\cdot,t) &= \partial_t\Pi(\cdot,t) &~&\text{in}~O_\varepsilon;\\
  \eta_{\varepsilon} &= 0 &~&\text{in}~\partial O_\varepsilon.
\end{aligned}\right.
\end{equation}
Then, let $(\hat{\xi_{\varepsilon}},\hat{\eta_{\varepsilon}})$ be the primitive function of
$(\xi_{\varepsilon},\eta_{\varepsilon})$, defined as follows:
\begin{equation}\label{pde3.11}
  \hat{\xi_{\varepsilon}}(\cdot,t)=\int_{0}^{t}\xi_{\varepsilon}(\cdot,s)ds;\quad
  \hat{\eta_{\varepsilon}}(\cdot,t)=\int_{0}^{t}\eta_{\varepsilon}(\cdot,s)ds.
\end{equation}
In this regard, one can
verify that $-(\hat{\xi_{\varepsilon}}+\hat{\eta_{\varepsilon}})$ satisfies
the equation $\eqref{pde:B}$. Therefore, we have the error term
precisely in the form of
$  w_\varepsilon = w_\varepsilon^{(2)}+ \hat{\xi_{\varepsilon}}+\hat{\eta_{\varepsilon}}$.
\end{itemize}
\end{itemize}

\begin{remark}
\emph{In terms of the equation $\eqref{pde:B}$,
if $w_\varepsilon^{(2)}$ is regarded as an unknown vector field,
the problem ``$\nabla\cdot v =g$'' has no uniqueness of the solution  
(see e.g. \cite[Theorem III.3.1]{Galdi11}).
Thus, there is no contradiction with that the constructed vector field $-(\hat{\xi_{\varepsilon}}+\hat{\eta_{\varepsilon}})$ satisfies
the same equation $\eqref{pde:B}$ as $w_\varepsilon^{(2)}$ does.
In fact, the intractable problem is that we cannot simply grasp satisfactory  estimates from the equation $\eqref{pde:B}$. Instead, by the constructed object $(\xi_{\varepsilon},\eta_{\varepsilon})$, the problem has been reduced to studying each of them in $\eqref{pde:C}$, where 
the so-called Bogovskii's operator is involved (see e.g. \cite{Bogovskii79} 
or \cite[Section III.3]{Galdi11}).   
We finally mention that the methods of constructing solutions
for (1) and (2) in $\eqref{pde:C}$ are different (see Section $\ref{sec4}$).}
\end{remark}

%\begin{remark}
%\emph{
%Although there is no direct relationship
%between $(\hat{\xi_{\varepsilon}},\hat{\eta_{\varepsilon}})$ and its counterparts given by Z. Shen \cite{Shen20} for stationary cases,
%they are still analogous in the sense of the role played in the error estimates. With a slightly different point of view, we prefer to understand it from the perspective
%of the boundary-layer correctors associated with Bogovskii's operator\textsuperscript{\ref{footnote}}, and
%it seems to be reasonable in the homogenization theory of fluid mechanics. }
%\end{remark}

\subsection{\centering Outline the proof of Theorem $\ref{thm1}$}\label{subsec:1.4}

\noindent
We start from introducing some terminologies used throughout the paper to distinguish the different types of correctors:
\begin{enumerate}
  \item Corrector, denoted by
$(W,\pi)$, defined in $\eqref{pde2.1}$;
%  \item Flux corrector, denoted by
%$\Phi$, defined in $\eqref{anti-sym}$;
  \item Corrector of Bogovskii's  operator\footnote{
%Abbreviated as Corrector of B.o..
In \cite{Balazi-Allaire-Omnes24,Shen20}, the same type corrector
was constructed by using Stokes equations. To the authors' best knowledge,
this type corrector was originally introduced by
E. Maru\v{s}i\'{c}-Paloka, A. Mikeli\'{c} \cite{Marusi-Mikelic96}.}, denoted by $\phi$, defined in $\eqref{divc1*}$;
  \item Boundary-layer correctors, denoted by
$(\hat{\xi}_{\varepsilon},\hat{\eta}_{\varepsilon})$, defined in
$\eqref{pde3.11}$, respectively.
\end{enumerate}

\begin{itemize}
  \item \textbf{Main steps:}
  \begin{itemize}
    \item \textbf{Step 1.} \emph{Find a good formula on the first-order expansion.}
        Precisely speaking, we derive the following error term
        associated with velocity and pressure, respectively.
\begin{equation}\label{eq:A}
\left\{
  \begin{aligned}
&w_{\varepsilon}=u_{\varepsilon}-\psi_{\varepsilon}
\big(W^{\varepsilon}\ast G+\varepsilon
\phi^\varepsilon\ast_{2} \partial G\big)
+\hat{\xi_{\varepsilon}}
+\hat{\eta_{\varepsilon}};\\
&q_{\varepsilon}=p_\varepsilon-p_0-
\varepsilon\psi_{\varepsilon}\pi^{\varepsilon}\ast_{1} G,
\end{aligned}\right.
\end{equation}
in which $(\hat{\xi_{\varepsilon}},\hat{\eta_{\varepsilon}})$ are 
boundary-layer correctors, and $\psi_\varepsilon$ is known as the radial cut-off function (see Lemma $\ref{lemma:cut-off}$ for the existence).
Also, the notations $W^\varepsilon$, $G$, $*_2$, etc
have been defined in Subsection $\ref{notation}$. A heuristic explanation on $\eqref{eq:A}$ has been presented in Subsection $\ref{subsec:1.2}$.

 \item  \textbf{Step 2.}  \emph{Derive the equations that error term $(w_\varepsilon,q_\varepsilon)$ satisfies.}
Plugging the error term $\eqref{eq:A}$
into unsteady Stokes operators in $\eqref{pde1.1}$, there holds
\begin{equation}\label{pde:A}
\left\{
\begin{aligned}
\partial_t w_{\varepsilon}
-\varepsilon^2\Delta w_{\varepsilon}+\nabla q_{\varepsilon}&=I_1+\varepsilon I_2+\varepsilon^2I_3+\varepsilon^3 I_4,&\text{in}&\quad\Omega_{\varepsilon}\times(0,T];\\
\nabla\cdot w_{\varepsilon}&=0,&\text{in}&\quad
\Omega_{\varepsilon}\times(0,T],
\end{aligned}
\right.
\end{equation}
with zero initial-boundary data and $T>0$,
where $I_1$, $I_2$, $I_3$ and $I_4$ have been formulated by
\begin{equation*}
\begin{aligned}
I_1:=&f-\nabla p_0-\psi_{\varepsilon}W^{\varepsilon}(\cdot,0)G
+\xi_{\varepsilon}+\eta_{\varepsilon};\\
I_2:=&-\psi_{\varepsilon}
\Big[
\partial_t\phi^{\varepsilon}\ast_2\partial G+
\phi^{\varepsilon}(\cdot,0):\partial G
\Big]+2\psi_{\varepsilon}(\partial W)^{\varepsilon}\ast_2\partial G
-
\nabla\psi_{\varepsilon}
\pi^{\varepsilon}\ast_1 G
-\psi_{\varepsilon}\nabla G\ast\pi^{\varepsilon};\\
I_3:=&\psi_{\varepsilon}
\nabla\cdot
\Big\{
(\nabla\phi)^{\varepsilon}\ast_2\partial G
\Big\}
+\nabla\cdot
\Big\{
\nabla\psi_{\varepsilon}\otimes \big(W^{\varepsilon}\ast G\big)
\Big\}
+\nabla\big(W^{\varepsilon}\ast G\big)\nabla
\psi_{\varepsilon}
+\psi_{\varepsilon} W^{\varepsilon}\ast\Delta G
-\Delta\hat{\xi_{\varepsilon}}
-\Delta\hat{\eta_{\varepsilon}};\\
I_4:=&\nabla
\big(\phi^{\varepsilon}\ast_2\partial G\big)\nabla\psi_{\varepsilon}+
\nabla\cdot\Big\{
\nabla\psi_{\varepsilon}\otimes
\big(\phi^{\varepsilon}\ast_2\partial G\big)
\Big\}
+\psi_{\varepsilon}
\nabla\cdot\Big\{\phi^{\varepsilon}\ast_2\nabla\partial G\Big\}.
\end{aligned}
\end{equation*}
(See the details in the proof Lemma $\ref{lemma:4.3}$, which can be temporarily skipped for an initial reading.)

\item \textbf{Step 3.}
\emph{Rewriting the right-hand side of $\eqref{pde:A}$
as the form of}
\begin{equation*}
\Theta+\varepsilon\nabla\cdot \Lambda+
\varepsilon\psi_{\varepsilon}
\nabla\cdot \Xi,
\end{equation*}
it follows from the energy estimate (see e.g. \cite{Temam79}, or
Lemma $\ref{lemma3.1}$) that
\begin{equation*}
\begin{aligned}
\|w_{\varepsilon}\|_{L^{\infty}
(0,T;L^2(\Omega_{\varepsilon}))}
+\varepsilon
\|\nabla w_{\varepsilon}\|_{L^{2}
(0,T;L^2(\Omega_{\varepsilon}))}
\lesssim
%\bigg\{
\|(\Theta,\Lambda,\psi_{\varepsilon}\Xi)\|_{L^2(0,T;L^2(\Omega_{\varepsilon}))}.
%+\|\Lambda\|_{L^2(0,T;L^2(\Omega_{\varepsilon}))}
%+\|\Xi\|_{L^2(0,T;L^2(\text{supp}
%(\psi_{\varepsilon})))}.
\end{aligned}
\end{equation*}
\emph{Then, reduce the error estimates
to show}:
\begin{equation}\label{pri:1}
\|(\Theta,\Lambda,\psi_\varepsilon\Xi)\|_{L^2(0,T;L^2(\Omega_{\varepsilon}))}
%+\|\Lambda\|_{L^2(0,T;L^2(\Omega_{\varepsilon}))}
%+\|\Xi\|_{L^2(0,T;L^2(\text{supp}
%(\psi_{\varepsilon})))}
= O(\varepsilon^{1/2}).
\end{equation}
(See the details in Subsection $\ref{subsec:5.2}$.)

\item \textbf{Step 4.}
\emph{The desired estimate $\eqref{pri:1}$ consequently relies on
the following three type estimates}:
\begin{itemize}
  \item Smoothness of the correctors, i.e., for any $1<q<\infty$,
  \begin{equation}\label{pri:7}
\|(\partial_t W_j,\nabla \pi_j)\|_{L^1(0,T;L^q(Y_f))}
+\|(\partial_t\Phi,\partial_t\phi)\|_{L^1(0,T;W^{1,q}(Y_f))}\lesssim 1,
  \end{equation}
where $\Phi$ is referred to the flux corrector, introduced in Proposition $\ref{prop2.1}$.

(See Propositions $\ref{prop2.1}$ and $\ref{prop3.1}$.)
  \item Regularity estimates on boundary-layer correctors,
i.e.,
  \begin{equation}\label{pri:6}
  \|(\xi_{\varepsilon},\eta_{\varepsilon})\|_{L^2(0,T;L^2(\Omega_\varepsilon))}
  +\varepsilon\|(\nabla\widehat\xi_{\varepsilon},\nabla\widehat\eta_{\varepsilon})\|_{L^2(0,T;L^2(\Omega_\varepsilon))}
  =O(\varepsilon^{1/2}).
  \end{equation}
(See Propositions
$\ref{lemma3.6-1}$, $\ref{lemma3.6-2}$ and
$\ref{lemma3.6-3}$.)
\item Well-posedness of the homogenized system $\eqref{pde1.2}$,
i.e.,
  \begin{equation}\label{pri:8}
\|p_0\|_{L^q(0,T;C^{m+1,\alpha}(\bar{\Omega}))}\lesssim
\|f\|_{L^q(0,T;C^{m,\alpha}(\bar{\Omega}))},
\quad \alpha\in(0,1)\text{~and~}m\geq 1.
\end{equation}
(See Proposition $\ref{P:1}$.)
\end{itemize}

\item \textbf{Step 5.} \emph{Show the error estimates on the pressure term.}

    We first establish the error estimates on the inertial term as follows:
\begin{equation}\label{pri:2}
 \|\partial_t w_\varepsilon\|_{L^2(\Omega_{\varepsilon,T})}
 = O(\varepsilon^{1/2}).
\end{equation}
Then, base upon the estimates $\eqref{pri:1}$ and $\eqref{pri:2}$,
a duality argument consequently leads to the error on the pressure
term, i.e.,
\begin{equation*}
 \inf_{c\in\mathbb{R}}\|q_\varepsilon-c\|_{L^2(\Omega_{\varepsilon,T})}
 = O(\varepsilon^{1/2}).
\end{equation*}
(See the details in Subsection $\ref{subsec:5.3}$.)
\end{itemize}

\end{itemize}

\section{Correctors}\label{sec2}

%\emph{The key observation} is that Caccioppoli's inequality offers a good decay
%in the interior region, but produces a bad scale factor. Meanwhile,
%the semigroup estimate can dominate the boundary region, owning a relatively bad decay, but creating a good scale factor. Thus,
%the idea is to impose a parameter $\xi_{\varepsilon}$ to balance their advantages and
%disadvantages such that we can improve the decay power of semigroup estimates
%(see the proof in Subsection $\ref{subsec:3.1}$).

\begin{proposition}[corrector $\&$ flux corrector]\label{prop2.1}
Let $0<T<\infty$, $d=2$ or $3$, and $1<q<\infty$. Suppose that $(W_j,\pi_j)$ is a weak solution of \eqref{pde2.1}
with $j=1,\cdots,d$.
Then there holds the energy estimate
\begin{equation}\label{pde2.14}
\begin{aligned}
&\|W_j\|_{L^{\infty}(0,T;L^2(Y_f))}+
\|W_j\|_{L^{2}(0,T;H^{1}(Y_f))}\lesssim1,
\end{aligned}
\end{equation}
where the multiplicative constant depends only on $d$ and $T$.
Also, for any $\alpha>(1/3)$,
the weak solution possesses higher regularity estimates:
\begin{subequations}
\begin{align}
\|W_j\|_{L^2(0,T;W^{1,q}(Y_f))}
+\|t^{\alpha}\partial_t W_j\|_{L^2(0,T;L^{2}(Y_f))}
+\|t^{\alpha}\pi_j\|_{L^2(0,T;L^{2}(Y_f)/\mathbb{R})}\lesssim 1;\label{pde3.33a}\\
\|\partial_t W_j\|_{L^1(0,T;L^q(Y_f))}
+ \|\nabla^2 W_j\|_{L^1(0,T;L^q(Y_f))}
+ \|\nabla \pi_j\|_{L^1(0,T;L^{q}(Y_f))}\lesssim 1,\label{pde3.33b}
\end{align}
\end{subequations}
in which the up to constant relies on $q$, $d$, $T$, and the character of $Y_f$.  Moreover, let
$\tilde{W}_j$ be the zero-extension of $W_j$ on the solid part $Y_s$.
%and for any $2\leq r<\infty$,
%\begin{equation}\label{pde3.32}
%\|\partial_t W\|_{L^1(0,T;L^r(Y_f))}\lesssim 1.
%\end{equation}
For each $t\in(0,T]$, we can define  $b_{ij}(\cdot,t):=\tilde{W}_j(\cdot,t)\cdot e_i-A_{ij}(t)$.  Then,
there exists $\Phi=\{\Phi_{ki,j}\}_{1\leq i,j,k\leq d}$
with $\Phi_{ki,j}(\cdot,t)\in H^{1}_{\emph{loc}}(\mathbb{R}^d)$,
which is also 1-periodic and satisfies
\begin{equation}\label{anti-sym}
 \nabla_k\Phi_{ki,j}=
 b_{ij}\quad \text{and}\quad
 \Phi_{ki,j}=-\Phi_{ik,j},
\end{equation}
as well as, the following regularity estimates:
\begin{subequations}
\begin{align}
&\|\partial_t\Phi\|_{L^1(0,T;H^1(Y))}
+\|\Phi\|_{L^1(0,T;H^1(Y))}
+\|\Phi(\cdot,0)\|_{H^1(Y)}\lesssim 1;\label{pri:3.5a}\\
&\|\partial_t\Phi\|_{L^1(0,T;W^{1,q}
(Y))}+\|\Phi\|_{L^1(0,T;W^{1,q}(Y))}
+\|\Phi(\cdot,0)\|_{L^\infty(Y)}
\lesssim 1.
\label{pri:3.5b}
\end{align}
\end{subequations}
Consequently, we also have $\Phi_{ki,j}\in C([0,T];C_{per}(Y))$.
\end{proposition}

The estimate $\eqref{pde2.14}$ directly follows from
a traditional argument while the pressure is merely estimated by
$H^{-1}$-norm
with respect to the temporal variable (see e.g. \cite{Temam79}). Although
the incompatibility will bring the negative influence to the solution's
regularity,
we infer that this effect only occurs
in the region where the solution has just begun to evolve
from the initial value. \emph{The main idea} on the
estimates $\eqref{pde3.33a}$ and $\eqref{pde3.33b}$ is to
employ interior regularity estimates to improve
the associated semigroup estimates in $L^2$-norm
(see Subsections $\ref{subsec:3.1}$ and $\ref{subsec:3.2}$),
and then use interpolation arguments to get the desired results,
which, to the authors' best knowledge, seems to be new even for
the general evolutionary Stokes equations without compatibility
conditions.

\begin{remark}\label{remark:3.1}
\emph{We refer to $\Phi$ as the flux corrector, and its estimates are grounded in correctors'.
To elucidate its role, we observe that
\begin{equation}\label{}
  \nabla\cdot\big(u_\varepsilon-\psi_\varepsilon W^\varepsilon*F\big)
  =\underbrace{\nabla\psi_\varepsilon\cdot (A-W^\varepsilon)*F
  -\nabla\psi_\varepsilon\cdot A*F}_{\text{layer~part}}
  +\underbrace{\psi_\varepsilon \big(|Y_f|^{-1}A-W^\varepsilon\big)*_2\partial F}_{\text{co-layer~part}}.
\end{equation}
By definition, the flux corrector plays a crucial role in addressing the boundary layer effects. Specifically, it collaborates with a corrector of Bogovskii's operator, introduced below, to correct the inhomogeneous divergence component in the error expansion. Notably, unlike the corrector of Bogovskii's operator, the flux corrector does not appear explicitly in the final form of the error expansion. This feature is consistent with the flux corrector introduced in the context of elliptic homogenization problems. }
\end{remark}

\begin{proposition}[corrector of Bogovskii's operator]\label{prop3.1}
Let $0<T<\infty$, $d=2$ or $3$, and $1<q<\infty$.
Suppose that the corrector $W$ and the permeability tensor $A$ are given as in Theorem $\ref{thm3}$.
Then, for any $t\geq 0$, there exists at least one weak solution $\phi$ associated with $W$ and $A$ by the following equations:
\begin{equation}\label{divc1}
\left\{
\begin{aligned}
\nabla\cdot\phi_{i,j}(\cdot,t) &=-W_{ij}(\cdot,t)
+|Y_f|^{-1} A_{ij}(t)&\qquad\text{in}&\quad\omega;\\
\phi_{i,j}(\cdot,t)&=0&\qquad\text{on}&\quad\partial\omega,
\end{aligned}
\right.
\end{equation}
which is merely the component form of $\eqref{divc1*}$,
with $1\leq i,j,k\leq d$, whose solution is
1-periodic and satisfies $\phi_{ki,j} (\cdot,t)\in H^1_{\text{per}}(Y_f)$.
Moreover, there holds refined regularity estimate
\begin{equation}\label{pri:4.2}
\|\phi\|_{L^1(0,T;W^{1,q}(Y_f))}
+\|\phi(\cdot,0)\|_{W^{1,q}(Y_f)}
+\|\partial_t\phi\|_{L^1(0,T;W^{1,q}(Y_f))}\lesssim 1,
\end{equation}
and we concludes that $\partial_t\phi_{ki,j}\in L^1(0,T;C_{per}(Y_f))$
and $\phi_{ki,j}\in C([0,T];W_{\text{per}}^{1,q}(Y_f))$.
\end{proposition}

\begin{remark}
\emph{The corrector $(W_j,\pi_j)$ defined in
\eqref{pde2.1} is taken from A. Mikeli\'c \cite{Mikelic94}, which slightly differs from
that given by G. Allarie \cite{Allaire92}, i.e.,
\begin{equation}\label{pde:1.3}
\left\{
\begin{aligned}
\partial_t w_j-\Delta w_{j}+\nabla
\tilde{\pi}_{j}&=e_j&\qquad&\text{in}\quad \omega\times(0,T];\\
\nabla\cdot w_j &=0&\qquad&\text{in}\quad
\omega\times(0,T];\\
w_j&=0&\qquad&\text{on}\quad\partial\omega\times(0,T];\\
w_j|_{t=0}&=0&\qquad&\text{in}\quad\omega,
\quad T>0.
\end{aligned}
\right.
\end{equation}
We observe that there holds the equality:
$\partial_t w_j(\cdot,t)=W_j(\cdot,t)$ for any $t\geq 0$,
in the sense of Stokes semigroup representation, up to a projection\footnote{
Let $\mathcal{A}:=P(-\Delta)$ and $P$ is the Helmholtz projection
(see e.g. \cite{Tsai18}). By functional analytic approach,
we have $w_j(\cdot,t)=\int_{0}^{t}ds\exp\{-(t-s)\mathcal{A}\}e_j$,
and then using a spectral
representation, one can derive $\partial_t w_j
= \exp\{-t\mathcal{A}\}e_j$. On the other hand, by noting that $P(\nabla\kappa) = 0$, there holds
$W_i(\cdot,t)=\exp\{-t\mathcal{A}\}e_i =
\exp\{-t\mathcal{A}\}P(e_i-\nabla \kappa)$.
Thus, combining the previous two formulas, the equality $\partial_t w_j = W_j$ holds
up to the projection $P$.}.
However, this relationship does not alleviate the complexity arising
from the incompatibility among the given data in the equations \eqref{pde2.1}
or \eqref{pde:1.3}. By modifying the initial conditions,
G. Sandrakov introduced a revised version of the corrector
equations \cite[Theorem 2]{Sandrakov97}. 
However, this modification did not resolve this incompatibility between the initial and boundary data.}
\end{remark}

\subsection{\centering Semigroup estimate I}\label{subsec:3.1}

\begin{lemma}[semigroup estimate I]\label{lemma3.3}
Let $d=2$ or $3$, and $2< p<\infty$.
Suppose that $(W_j,\pi_j)$ is a weak solution of \eqref{pde2.1}
with $j=1,\cdots,d$. Then, for any $t>0$,
there holds a decay estimate\textsuperscript{\ref{footnote1}}
\begin{equation}\label{pri:3.6}
\bigg(\int_{Y_f}|\nabla W_j(\cdot, t)|^2\bigg)^{1/2}
\leq C_p t^{-\frac{p}{3p-2}},
\end{equation}
in which the constant $C_p$ depends on $d$, $p$, and the character of
$Y_f$.
\end{lemma}

\emph{The key observation} on the estimate $\eqref{pri:3.6}$ is that Caccioppoli-type inequality offers a good time-decay
in the interior region, but produces a bad spacial scale-factor. Meanwhile,
the semigroup estimate can dominate the region near boundary $\partial Y_s$, owning a relatively bad time-decay, but creating a good spacial scale-factor. Thus,
\emph{the idea} is to bring in a parameter $\rho$ to balance their advantage and disadvantage such that we can ``improve'' the decay power of Stokes semigroup estimates.

\begin{proof}
It suffices to show the estimate $\eqref{pri:3.6}$ for any $t\in(0,t_0)$,
where $t_0\in (0,1/2)$ is usually a small number,
while for the case $t>t_0$ it can be simply inferred from the semigroup estimate itself. Also, we merely establish the concrete proof in the case of $d=3$. For $d=2$, one can simply set
\begin{equation*}
\tilde{W}_j(y_1,y_2,y_3,t) = (W_j(y_1,y_2),0,t);\quad
\tilde{\pi}(y_1,y_2,y_3,t) =\pi(y_1,y_2,t);\quad
\tilde{W}_j(y_1,y_2,y_3,0) = e_j,
\quad j=1,2.
\end{equation*}
It is not hard to verify that $(\tilde{W}_j,\tilde{\pi})$ satisfies
the corrector equation $\eqref{pde2.1}$ in three dimensional case.
Therefore, once three dimensional case is proved, two
dimensional problem can be directly approached in this way.

\begin{figure}
  \centering
  \includegraphics[width=0.8\textwidth]{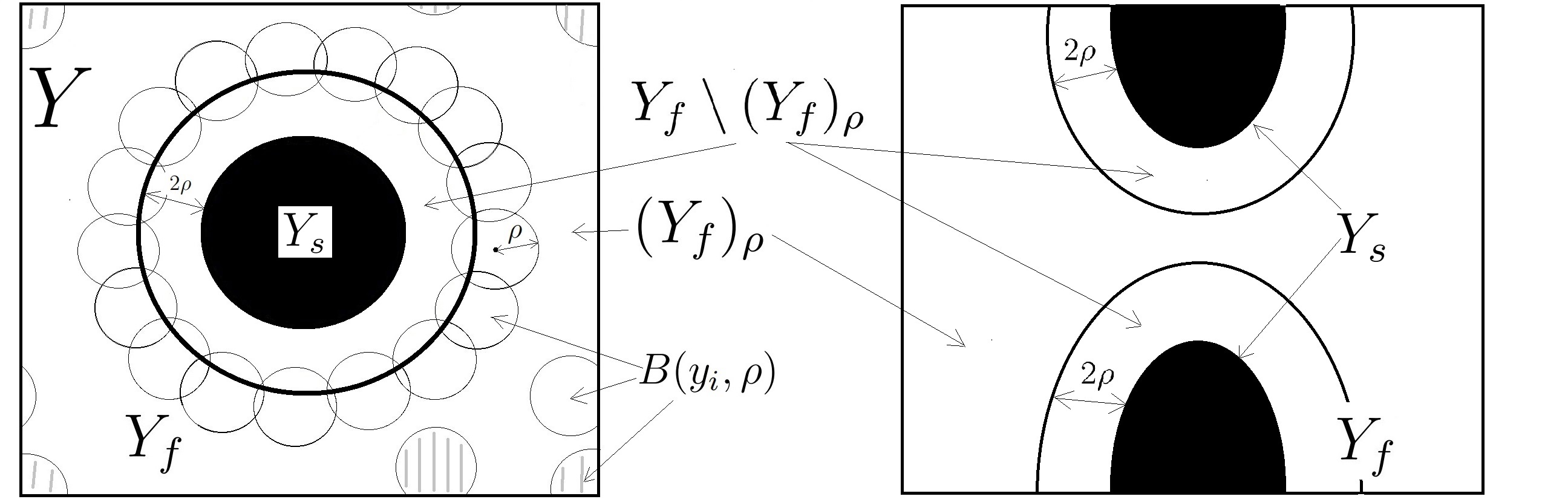}
  \caption{\small This figure provides a heuristic illustration of the decomposition of region $Y_f$ in two dimensional case. Furthermore, it is important to note that region $Y$ is a periodic domain.}\label{subfig_2}
\end{figure}

Firstly, we decompose the integral domain $Y_f$ into two parts (see Figure $\ref{subfig_2}$): $(Y_f)_{\rho}$ and $Y_f\setminus(Y_f)_{\rho}$, where $(Y_f)_{\rho}:=\{y\in Y_f:\text{dist}(y,\partial Y_s)\geq 2\rho\}$ for the parameter $\rho>0$, which will be fixed later. In this respect, the desired estimate $\eqref{pri:3.6}$ is reduced to showing: (For the ease of the statement, we omit the subscript of $W_j$ in the proof.)
\begin{equation}\label{f:3.9}
\bigg(\int_{Y_f}|\nabla W(\cdot,t)|^2\bigg)^{1/2}
\lesssim
\bigg(\int_{Y_f\setminus(Y_f)_{\rho}}|\nabla W(\cdot,t)|^2\bigg)^{1/2}
+\bigg(\int_{(Y_f)_{\rho}}|\nabla W(\cdot,t)|^2\bigg)^{1/2}=:I_1+I_2.
\end{equation}
The easier term is $I_1$, and by using
H\"older's inequality and the semigroup estimates (see e.g.
\cite[pp.82]{Tsai18}) in the order, we obtain
\begin{equation}\label{pde2.5}
I_1\lesssim \rho^{\frac{1}{2}-\frac{1}{p}}
\bigg(\int_{Y_f}|\nabla W(\cdot,t)|^p \bigg)^{1/p}
\lesssim \rho^{\frac{1}{2}-\frac{1}{p}}
t^{-\frac{1}{2}}.
\end{equation}

We now turn to the second term $I_2$, and start from giving
a family of cut-off functions, denoted by $\{\chi_i\}$, which
satisfy that
$\chi_{i}(y):=\chi_0(y+y_i)$, and $(Y_f)_{\rho}\subset\bigcup_{i}B(y_i,\rho)\subset
(Y_f)_{\frac{1}{2}\rho}$.
We assume that
\begin{itemize}
  \item $\chi_0\in C_{per}^1(Y)$ is a cut-off function;
  \item $\chi_0=1$ on $B(0,\rho/2)$ and supp$(\chi_0)\subset B(0,\rho)$  with $|\nabla\chi_0|\lesssim \frac{1}{\rho};$
  \item For any point in the region $(Y_f)_\rho$, it is covered by, at most, five elements of the covering family $\{B(y_i,\rho)\}$.
\end{itemize}
From the assumptions on $\{\chi_i\}$, it follows that  $\text{dist}\big(\partial Y_s,\text{supp}(\chi_i)\big)\geq \rho$.
Moreover, we define a family of indicator functions associated with
$\{\chi_i\}$ as follows:
\begin{equation*}
\tilde{\chi_i}=1 ~\text{in}~ B(y_i,\rho) \quad \text{and}\quad \tilde{\chi_i}=0 ~\text{outside}~ B(y_i,\rho).
\end{equation*}
Therefore, the second term in the right-hand side of $\eqref{f:3.9}$
can be estimated by
\begin{equation*}
I_2\leq \bigg(\sum_{i}\int_{Y_f}\chi_i^2|\nabla W(\cdot,t)|^2\bigg)^{1/2},
\end{equation*}
while we claim that: for each $i$ and any $t>0$, there hold
\begin{subequations}
\begin{align}
&\int_{Y_f}\chi_i^2|\nabla W(\cdot,t)|^2 \lesssim \int_{Y_f}\chi_i^2|\nabla\times W(\cdot,t)|^2+\frac{1}{\rho^2}\int_{Y_f}
  \tilde{\chi_i}|W(\cdot,t)|^2; \label{f:3.10a}\\
&\int_{Y_f}\chi_i^2|\nabla\times W(\cdot,t)|^2 \lesssim
\frac{1}{\rho^2}\int_{0}^{T}\int_{Y_f}
\tilde{\chi_i}|\nabla W|^2.
\label{f:3.10b}
\end{align}
\end{subequations}

Admitting the claims $\eqref{f:3.10a}$, $\eqref{f:3.10b}$ for a moment, and combining them,  we will get
\begin{equation}\label{pde2.3}
\int_{Y_f}\chi_i^2|\nabla W(\cdot,t)|^2dy\lesssim
\frac{1}{\rho^2}\bigg\{\int_{0}^{T}\int_{Y_f}
\tilde{\chi_i}|\nabla W|^2+
\sup_{0\leq t\leq T}\int_{Y_f}
  \tilde{\chi_i}|W(\cdot,t)|^2\bigg\}.
\end{equation}
In view of the assumptions on $\chi_i$ and the above estimate \eqref{pde2.3}, we obtain that
\begin{equation}\label{pde2.4}
 \begin{aligned}
 I_2&\leq\bigg(\sum_i\int_{Y_f}\chi_i^2|\nabla W(\cdot,t)|^2\bigg)^{1/2}\\
 &\lesssim \frac{1}{\rho}\bigg(
 \int_{0}^{T}\int_{Y_f}|\nabla W|^2+\sup_{0\leq t\leq T}\int_{Y_f}
 |W(\cdot,t)|^2 \bigg)^{1/2}\lesssim^{\eqref{pde2.14}} \frac{1}{\rho}.
 \end{aligned}
\end{equation}
As a result, plugging \eqref{pde2.5} and $\eqref{pde2.4}$ back into $\eqref{f:3.9}$, and then taking $\rho=t^{\frac{p}{3p-2}}$ (which
requires $t$ to be small), one can acquire
\begin{equation*}
\begin{aligned}
\bigg(\int_{Y_f}|\nabla W(\cdot,t)|^2\bigg)^{1/2}
\lesssim \rho^{\frac{1}{2}-\frac{1}{p}}t^{-\frac{1}{2}}
+\rho^{-1}
\lesssim t^{-\frac{p}{3p-2}},
\end{aligned}
\end{equation*}
which is the desired estimate $\eqref{pri:3.6}$.

We plan to use two steps to complete the whole proof, which is originally inspired by \cite{Jin13}.

\textbf{Step 1.}
We now verify the claims $\eqref{f:3.10a}$, $\eqref{f:3.10b}$,
and we start from dealing with the estimate $\eqref{f:3.10b}$ therein.
Let $\nabla\times(\chi^2_i\nabla\times W)$ be the test function and act on both sides of \eqref{pde2.1}, and by noting the facts:  $\nabla\cdot(\nabla\times (\chi_i^2\nabla\times W))=0$, $\partial Y_s\bigcap\text{supp}\{\chi_i\}=\emptyset$ and the periodicity of $\chi_i$,
integration by parts gives us
\begin{equation}\label{pde2.6}
\begin{aligned}
0&=\int_{Y_f}(\partial_t W-\Delta W+\nabla\pi)
\cdot\nabla\times(\chi_i^2\nabla\times W)\\
&=\int_{Y_f}\partial_t W
\cdot\nabla\times(\chi_i^2\nabla\times W)-
\int_{Y_f}\Delta W\cdot\nabla\times(\chi_i^2\nabla\times W).
\end{aligned}
\end{equation}
By the basic formula:
\begin{equation}\label{k-8}
 \nabla\cdot\big(\vec{a}\times \vec{b}\big) =\vec{b}\cdot(\nabla\times \vec{a})
 -\vec{a}\cdot(\nabla\times \vec{b}),
\end{equation}
we can further derive that
\begin{equation*}
\begin{aligned}
\int_{Y_f}\partial_t W
\cdot\nabla\times(\chi_i^2\nabla\times W)
&\overset{\eqref{k-8}}{=}-\int_{Y_f}\nabla\cdot[\partial_t W\times(\chi_i^2\nabla\times W)]
+\int_{Y_f}(\nabla\times\partial_t W)\cdot
(\chi_i^2\nabla\times W) \\
&=-\int_{\partial Y_f}\vec{n}\cdot[\partial_t W\times(\chi_i^2\nabla\times W)]
+\int_{Y_f}(\nabla\times\partial_t W)\cdot
(\chi_i^2\nabla\times W) \\
&=\frac{1}{2}\frac{d}{dt}\int_{Y_f}
\chi_i^2|\nabla\times W|^2,
\end{aligned}
\end{equation*}
where we merely take 
$\vec{a}=\partial_t W$ and $\vec{b}= \chi_i^2\nabla\times W$. This together with \eqref{pde2.6} leads to
\begin{equation}\label{k-7}
\begin{aligned}
0&=\frac{1}{2}\frac{d}{dt}\int_{Y_f}
\chi_i^2|\nabla\times W|^2
-\int_{Y_f}\Delta W\cdot\nabla\times(\chi_i^2\nabla\times W)\\
&=\frac{1}{2}\frac{d}{dt}\int_{Y_f}
\chi_i^2|\nabla\times W|^2+\int_{Y_f}\chi_i^2|\Delta W|^2-
\int_{Y_f}\Delta W\cdot\nabla\chi_i^2\times(\nabla\times W),
\end{aligned}
\end{equation}
where the second equality is due to the following computation:
\begin{equation*}
\begin{aligned}
&\nabla\times (\chi_i^2\nabla\times W)=\nabla\chi^2_i\times(\nabla\times W)+\chi_i^2\nabla\times\nabla\times W\\
&\nabla\times\nabla\times W=-\Delta W+\nabla(\nabla\cdot W)=-\Delta W.
\end{aligned}
\end{equation*}

Rewriting the above equality $\eqref{k-7}$, there holds
\begin{equation}\label{pde2.8}
\begin{aligned}
\frac{1}{2}\frac{d}{dt}\int_{Y_f}
\chi_i^2|\nabla\times W|^2+\int_{Y_f}\chi_i^2|\Delta W|^2
&=\int_{Y_f}\Delta W\cdot\nabla\chi_i^2\times(\nabla\times W)\\
&\leq 2\int_{Y_f}|\chi_i||\Delta W||\nabla \chi_i||\nabla\times W|.
\end{aligned}
\end{equation}
(Note that the repeated index $i$ does not represent a sum throughout the proofs of Lemmas $\ref{lemma3.3}$, $\ref{lemma2.2}$.)
Applying Young's inequality to the right-hand side of $\eqref{pde2.8}$, it is not hard to see that
\begin{equation*}
\frac{1}{2}\frac{d}{dt}\int_{Y_f}
\chi_i^2|\nabla\times W|^2+\int_{Y_f}\chi_i^2|\Delta W|^2
\lesssim
\int_{Y_f}|\nabla\chi_i|^2|\nabla\times W|^2,
\end{equation*}
and for any $t>0$ we consequently derive that
\begin{equation*}
\begin{aligned}
\int_{Y_f}\chi_i^2|\nabla\times W(\cdot,t)|^2
\lesssim\int_{0}^{T}\int_{Y_f}|\nabla\chi_i|^2|\nabla W|^2+\int_{Y_f}\chi_i^2|\nabla\times W(\cdot,0)|^2
\lesssim\frac{1}{\rho^2}\int_{0}^{T}\int_{Y_f}
\tilde{\chi_i}|\nabla W|^2,
\end{aligned}
\end{equation*}
which gives the claim $\eqref{f:3.10b}$.

\textbf{Step 2.} We turn to study the claim $\eqref{f:3.10a}$.
Recalling the formula $\nabla\times(\chi_i^2\nabla\times W)=\nabla\chi^2_i\times(\nabla\times W)-\chi_i^2\Delta W$ in $Y_f$, it leads us to the following integral equality
\begin{equation}\label{pde2.9}
-\int_{Y_f}\chi_i^2\Delta W\cdot W =\int_{Y_f}[\nabla\times(\chi_i^2\nabla\times W)]\cdot W -2\int_{Y_f}\chi_i(\nabla\chi_i\times\nabla\times W)\cdot W.
\end{equation}
Integrating by parts, the left-hand side of \eqref{pde2.9} is equal to
\begin{equation*}
\int_{Y_f}\chi_i^2|\nabla W|^2+2\int_{Y_f}\chi_i\nabla W:\nabla\chi_i\otimes W,
\end{equation*}
and moving the second term above to the right-hand side of \eqref{pde2.9}
we then derive that
\begin{equation}\label{pde2.10}
\int_{Y_f}\chi_i^2|\nabla W|^2 \lesssim
\int_{Y_f}|\chi_i||\nabla W||\nabla\chi_i||W|+
\Big|\int_{Y_f}\nabla\times(\chi_i^2\nabla\times W)\cdot W\Big|.
\end{equation}
Continue the computation as follows:
\begin{equation*}%\label{pde2.11}
\begin{aligned}
\int_{Y_f} \nabla\times(\chi_i^2\nabla\times W)
&\cdot W
\overset{\eqref{k-8}}{=}\int_{Y_f}\nabla\cdot[(\chi_i^2\nabla\times W)\times W]
+\int_{Y_f}
(\chi_i^2\nabla\times W)\cdot (\nabla\times W)\\
&=\int_{\partial Y_f}\vec{n}\cdot[(\chi_i^2\nabla\times W)\times W]
+\int_{Y_f}
(\chi_i^2\nabla\times W)\cdot (\nabla\times W)
=\int_{Y_f}
\chi_i^2|\nabla\times W|^2.
\end{aligned}
\end{equation*}
Inserting the above equality back into \eqref{pde2.10} and using
Young's inequality again, there holds
\begin{equation*}
\int_{Y_f}\chi_i^2|\nabla W|^2\leq
\theta\int_{Y_f}\chi_i^2|\nabla W|^2 +C_{\theta}\int_{Y_f}|\nabla\chi_i|^2|W|^2
+\int_{Y_f}\chi_i^2|\nabla\times W|^2,
\qquad \theta\in(0,1),
\end{equation*}
which immediately implies the stated result $\eqref{f:3.10a}$.
We have completed all the proof.
\end{proof}

\subsection{\centering Semigroup estimate II}\label{subsec:3.2}

\noindent
In fact, we repeat the same philosophy used in Lemma $\ref{lemma3.3}$
to show the estimate $\eqref{pri:3.7}$ in $L^2$-norm, and then appeal
to an interpolation, where we adopt the stream function method to get a higher-order interior estimate.

\begin{lemma}[semigroup estimate II]\label{lemma2.2}
Let $d\geq 2$ and $p\geq 2$ be sufficiently large, and $\gamma:=\frac{p(19p-14)}{(3p-2)(7p-2)}$.
Suppose that $(W_j,\pi_j)$ is a weak solution of \eqref{pde2.1} with $j=1,\cdots,d$.
Then, for any $r\in(1,\infty)$,
there exists a constant $q\in(1,\infty)$, such that
$\lambda:=\frac{2(q-r)}{r(q-2)}$ satisfying $0\leq\lambda\leq 1$,
and there holds the decay estimate\textsuperscript{\ref{footnote1}}
\begin{equation}\label{pri:3.7}
\|\partial_t W_j(\cdot,t)\|_{L^r(Y_f)}\lesssim t^{-1+\lambda(1-\gamma)}
%t^{-1+\lambda[1-\frac{p(19p-14)}{(3p-2)(7p-2)}]},
\end{equation}
for any $t>0$, where the multiplicative constant depends on $d,p,q$, and the character of $Y_f$.
\end{lemma}

\begin{proof}
In general, we adopt a similar strategy and notation as presented in Lemma \ref{lemma3.3}, with the additional use of an interpolation argument. The entire proof is structured into four steps.

\textbf{Step 1.} \emph{One step reduction.}
The main idea is that we firstly improve the decay estimate in
$L^2$-spatial norm, and then appeal to the interpolation argument
to get a weaker improvement in terms of $L^r$-spatial norms with $r\not= 2$. Thus,
the key step is to establish the following estimate: (Here we omit the subscript of $W_j$ throughout the proof.)
\begin{equation}\label{f:3.11}
\bigg(\int_{Y_f}dy|\Delta W(y,t)|^2\bigg)^{1/2}\lesssim
t^{-\gamma}
\end{equation}
for any $t>0$. Admitting the above estimate for a while, we introduce
the infinitesimal generator $\mathcal{A}:=P(-\Delta)$, where the operator $P$ is known as the Helmholtz projection, and then it follows from the boundedness of $P$ and $\eqref{f:3.11}$ that
\begin{equation*}
\|\mathcal{A}W(\cdot,t)\|_{L^2(Y_f)}=
\|P(-\Delta)W(\cdot,t)\|_{L^2(Y_f)}\leq
\|\Delta W(\cdot,t)\|_{L^2(Y_f)}\lesssim
t^{-\gamma}.
\end{equation*}
Recalling Stokes semigroup estimates (see e.g. \cite[pp.81]{Tsai18}):
for $1<q<\infty$, we have
\begin{equation*}
\|\mathcal{A} W(\cdot,t)\|_{L^q(Y_f)}\lesssim t^{-1}.
\end{equation*}
Then, by preferring a suitable $q\in(1,\infty)$ such that
$q<r\leq 2$ or $2\leq r<q$, one can employ the interpolation inequality to get that
\begin{equation*}
\begin{aligned}
\|\mathcal{A}W(\cdot,t)\|_{L^r(Y_f)}&\leq
\|\mathcal{A}W(\cdot,t)\|^{\lambda}_{L^2(Y_f)}
\|\mathcal{A}W(\cdot,t)\|^{1-\lambda}_{L^q(Y_f)}\\
&\lesssim t^{-1+\lambda(1-\gamma)},
\end{aligned}
\end{equation*}
where $\lambda:=\frac{2(q-r)}{r(q-2)}$. Using the
representation of the solution of \eqref{pde2.1} by the semigroup theory, we have $\partial_t W = \mathcal{A} W$ in
$Y_f\times (0,\infty)$, which implies
the stated estimate $\eqref{pri:3.7}$.

\textbf{Step 2.} \emph{A further reduction}.
The remainder part of the proof is devoted to establishing
the crucial estimate $\eqref{f:3.11}$ from two important ingredients.  Introducing the parameter $\rho>0$, we decompose the following integral into two parts: the inner part and the near-boundary part, i.e.,
\begin{equation}\label{pde2.24}
\bigg(
\int_{Y_f}|\Delta W(\cdot,t)|^2
\bigg)^{1/2}\leq
\bigg(
\int_{(Y_f)_{\rho}}|\Delta W(\cdot,t)|^2
\bigg)^{1/2}+
\bigg(
\int_{Y_f\setminus(Y_f)_{\rho}}|\Delta W(\cdot,t)|^2
\bigg)^{1/2}.
\end{equation}

Let $p\geq 2$, and we claim that there hold
\begin{subequations}\label{}
\begin{align}
 \bigg(
\int_{(Y_f)_{\rho}}|\nabla^2W(\cdot,t)|^2
\bigg)^{1/2}&\lesssim \rho^{-3}t^{-\frac{p}{3p-2}};\label{pde2.25}\\
 \bigg(
\int_{Y_f\setminus(Y_f)_{\rho}}|\Delta W(\cdot,t)|^2
\bigg)^{1/2}
&\lesssim \rho^{\frac{1}{2}-\frac{1}{p}}t^{-1}. \label{pde2.26}
\end{align}
\end{subequations}

Admitting the above two estimates temporarily,
we can proceed with our analysis. Inserting \eqref{pde2.25} and \eqref{pde2.26} back into \eqref{pde2.24}, we have
\begin{equation*}
\bigg(\int_{Y_f}|\Delta W(\cdot,t)|^2\bigg)^{1/2}\lesssim \rho^{-3}t^{\frac{-p}{3p-2}}+
\rho^{\frac{1}{2}-\frac{1}{p}}t^{-1},
\end{equation*}
and minimizing the right-side above by taking $\rho>0$ such that $\rho^{-3}t^{\frac{-p}{3p-2}}=
\rho^{\frac{1}{2}-\frac{1}{p}}t^{-1}$, there holds
\begin{equation*}
\bigg(\int_{Y_f}|\Delta W(\cdot,t)|^2\bigg)^{1/2}\lesssim
t^{\frac{-p(19p-14)}{(3p-2)(7p-2)}},
\end{equation*}
which gives us the core estimate $\eqref{f:3.11}$ by setting
$\gamma:=\frac{p(19p-14)}{(3p-2)(7p-2)}$.

\textbf{Step 3.} Show the estimate $\eqref{pde2.25}$.
We start from showing
\begin{equation}\label{pde2.31}
\int_{Y_f}\chi_i^2|\nabla^2 W(\cdot,t)|^2\lesssim
\|\nabla\chi_i\|^2_{L^{\infty}(Y)}\bigg\{
\int_{0}^{T}\int_{Y_f}\tilde{\chi_i}|\nabla^2 W|^2+\int_{Y_f}\tilde{\chi_i}|\nabla W(\cdot,t)|^2
\bigg\}.
\end{equation}
To see this, it is known that
$\partial_k W$ with $k=1,\cdots,d$ satisfies the equations in
$Y_f\times (0,T]$ as the same as $W$ does, and therefore it follows
from the estimate \eqref{pde2.3} that
\begin{equation*}
\int_{Y_f}\chi_i^2|\nabla \partial_k W(\cdot,t)|^2\lesssim
\|\nabla\chi_i\|^2_{L^{\infty}(Y)}\bigg\{
\int_{0}^{T}ds\int_{Y_f}\tilde{\chi_i}|\nabla^2 W(\cdot,s)|^2
+\int_{Y_f}\tilde{\chi_i}|\nabla W(\cdot,t)|^2
\bigg\},
\end{equation*}
which immediately gives the desired estimate $\eqref{pde2.31}$.
The desired estimate is reduced to estimating the first term in the bracket above. We also note that the interior estimate is translation-invariant within $Y_f$, and thus the center position of the estimated region can be disregarded. Therefore, we only need to establish the following estimate:
\begin{equation}\label{pde2.23}
\begin{aligned}
\int_{0}^Tdt\|\nabla^2 W(\cdot,t)\|^2_{L^2(B/2)}
%&\lesssim
%\|\nabla \chi\|^2_{L^{\infty}}\int_{0}^{T}\|\nabla W(\cdot,t)\|^2_{L^2(B)}dt+\int_{0}^{T}\|\nabla v\|^2_{L^2(B)}dt\\
%&\lesssim\|\nabla^2\chi\|^2_{L^{\infty}}
%\bigg\{\int_{0}^{T}\|\nabla W(\cdot,t)\|^2_{L^2(B)}dt+\int_{0}^{T}\|\nabla \times W\|^2_{L^2(2B)}dt\bigg\}\\
&\lesssim\|\nabla^2\chi\|^2_{L^{\infty}(Y)}
\int_{0}^{T}dt\|\nabla W(\cdot,t)\|^2_{L^2(2B)},
\end{aligned}
\end{equation}
where $B:=B(y,2\rho)$ with $2B\subset \omega$ is arbitrary, and
$\chi\in C^{\infty}_{per}(Y)$ is a cut-off function
satisfying $\chi = 1$ in $B$ and supp$(\chi)\subset 2B$
with $|\nabla\chi|\lesssim \frac{1}{\rho}$. Consequently,
plugging the estimate $\eqref{pde2.23}$ back into $\eqref{pde2.31}$,
we have
\begin{equation*}
\bigg(\int_{Y_f}\chi_i^2|\nabla^2 W(\cdot,t)|^2
\bigg)^{\frac{1}{2}}
\overset{\eqref{pri:3.6}}{\lesssim} \rho^{-3} + \rho^{-1}t^{-\frac{p}{3p-2}}
\lesssim \rho^{-3}t^{-\frac{p}{3p-2}},
\end{equation*}
which is the stated estimate $\eqref{pde2.25}$.

Then, we turn to study the estimate $\eqref{pde2.23}$.
Set $v=\text{curl~}W$ in $\omega\times[0,T]$. By recalling
the equations \eqref{pde2.1}, there holds a new parabolic system:
\begin{equation*}
\left\{
\begin{aligned}
(\partial_t-\Delta)v&=0&\qquad&\text{in}\quad Y_f\times(0,T];\\
v|_{\partial Y_f}&=\text{curl~} W&\qquad&\text{for}\quad  0<t\leq T;\\
v|_{t=0}&=0.&\qquad&
\end{aligned}
\right.
\end{equation*}

For we are interested in the interior estimates, by setting $u=\chi v$,
from the above parabolic system we get
\begin{equation}\label{pde2.20}
\left\{
\begin{aligned}
(\partial_t-\Delta)u&=-\Delta \chi v-2\nabla\chi\cdot\nabla v:=g&\qquad&\text{in}\quad Y_f\times(0,T];\\
u|_{\partial Y_s}&=0,&\qquad&\text{for}\quad 0<t<T;\\
u|_{t=0}&=\chi v|_{t=0}=0.&\qquad&\\
\end{aligned}
\right.
\end{equation}
It is well known from the energy estimate $\eqref{pde2.14}$ that
$v_j\in L^2(0,T;L^2_{per}(Y_f)^d)$,
(where $v_j = \text{curl~}W_j$ represents
the $j^{\text{th}}$ component of $v$.) whereupon we can verify that
each of the components of $g$ belongs to
$L^2(0,T; H^{-1}(Y_{f}))$. Moreover, there holds
\begin{equation*}
\begin{aligned}
\|g(\cdot,t)\|_{H^{-1}(Y_f)}&\lesssim
\|\nabla\chi\cdot\nabla v(\cdot,t)\|_{H^{-1}(\text{supp}(\chi))}
+\|\Delta\chi\|_{L^{\infty}(Y)}
\|v(\cdot,t)\|_{L^2(\text{supp}(\chi))}\\
&\leq\|\nabla^2\chi\|_{L^{\infty}(Y)}
\|v(\cdot,t)\|_{L^2(\text{supp}(\chi))}
\end{aligned}
\end{equation*}
for any $t>0$, and we get
\begin{equation*}
\int_{0}^{T}dt\|g(\cdot,t)\|^2_{H^{-1}
(Y_f)}\lesssim
\|\nabla^2\chi\|^2_{L^{\infty}(Y)}\int_{0}^{T}dt
\int_{2B}|v(\cdot,t)|^2.
\end{equation*}
This together with the energy estimates for \eqref{pde2.20} gives us
\begin{equation}\label{pde2.21}
\|\nabla u\|_{L^2(0,T;L^2(Y_f))}
\lesssim
\|\nabla^2\chi\|_{L^{\infty}(Y)}\bigg(\int_{0}^{T}
\int_{2B}|v|^2\bigg)^{1/2}.
\end{equation}
Recalling the definition of $u$ and $v$, it is not hard to see that
\begin{equation*}
\begin{aligned}
\bigg(\int_{0}^{T}\int_{Y_f}\chi^2|\nabla v|^2\bigg)^{1/2}
&\lesssim
\bigg(
\int_{0}^{T}\int_{Y_f}|\nabla u|^2
\bigg)^{1/2}+
\|\nabla\chi\|_{L^{\infty}(Y)}
\bigg(
\int_{0}^{T}\int_{2B}|v|^2
\bigg)^{1/2}\\
&\overset{\eqref{pde2.21}}{\lesssim}
\|\nabla^2\chi\|_{L^{\infty}(Y)}
\bigg(
\int_{0}^{T}\int_{2B}|v|^2
\bigg)^{1/2},
\end{aligned}
\end{equation*}
and therefore, we have
 \begin{equation}\label{pde2.22}
\bigg(\int_{0}^{T}\int_{B}|\nabla v|^2\bigg)^{1/2}
\lesssim
\|\nabla^2\chi\|_{L^{\infty}(Y)}
\bigg(
\int_{0}^{T}\int_{2B}|v|^2
\bigg)^{1/2}.
\end{equation}

By noting that $v=\text{curl~}W$,
we have $-\Delta W=\text{curl~} v$ in $Y_f\times[0,T]$, where we employ
the vector identity $-\Delta W =\text{curl~}\text{curl~}W-\nabla(\nabla\cdot W)$ (see e.g.  \cite[pp.9]{Tsai18}).
Thus, from the interior $H^2$ estimates for elliptic equation (see e.g. \cite[Chapter 4]{Giaquinta-Martinazzi12}), it follows that
\begin{equation*}
\|\nabla^2 W(\cdot,t)\|_{L^2(B/2)}
%\lesssim\|\nabla \chi\|_{L^{\infty}}\|\nabla W(\cdot,t)\|_{L^2(B)}+\|\nabla\times v\|_{L^2(B)}
\lesssim
\|\nabla \chi\|_{L^{\infty}(Y)}\|\nabla W(\cdot,t)\|_{L^2(B)}+
\|\text{curl~} v(\cdot,t)\|_{L^2(B)}
\end{equation*}
holds for any $t>0$. Hence, integrating both sides above with respect to  $t$ from 0 to $T$, we can derive that
\begin{equation*}
\begin{aligned}
\int_{0}^Tdt\|\nabla^2 W(\cdot,t)\|^2_{L^2(B/2)}
&\lesssim
\|\nabla \chi\|^2_{L^{\infty}(Y)}
\int_{0}^{T}dt\|\nabla W(\cdot,t)\|^2_{L^2(B)}
+\int_{0}^{T}dt\|\nabla v(\cdot,t)\|^2_{L^2(B)}\\
&\overset{\eqref{pde2.22}}{\lesssim}
\|\nabla^2\chi\|^2_{L^{\infty}(Y)}
\bigg\{\int_{0}^{T}dt\|\nabla W(\cdot,t)\|^2_{L^2(B)}
+\int_{0}^{T}dt\|v(\cdot,t)\|^2_{L^2(2B)}\bigg\}\\
&\lesssim\|\nabla^2\chi\|^2_{L^{\infty}(Y)}
\int_{0}^{T}dt\|\nabla W(\cdot,t)\|^2_{L^2(2B)},
\end{aligned}
\end{equation*}
which is the stated estimate $\eqref{pde2.23}$.

\textbf{Step 4.} Show the estimate $\eqref{pde2.26}$.
We start from rewriting the equations $\eqref{pde2.1}$ as follows: for any $t>0$,
\begin{equation*}
\left\{
\begin{aligned}
-\Delta W(\cdot,t)+\nabla \pi(\cdot,t)&=-\partial_t W(\cdot,t) &\quad&\text{in}\quad Y_f;\\
\nabla\cdot W(\cdot,t)&=0&\quad&\text{in}\quad Y_f;\\
W(\cdot,t)|_{\partial Y_s}&=0.&\quad&
\end{aligned}\right.
\end{equation*}

For any $1<p<\infty$, from the $L^p$ theory for
stationary Stokes system (see e.g. \cite[Chapter IV]{Galdi11}), it follows that
\begin{equation}\label{f:3.12}
\|\nabla^2W(\cdot,t)\|_{L^p(Y_f)}
+\|\nabla\pi(\cdot,t)\|_{L^p(Y_f)}
\lesssim
\|\partial_tW(\cdot,t)\|_{L^p(Y_f)},
\end{equation}
where the up to constant depends on $p,d$ and the character of $Y_f$.

Then, for any $p\geq 2$, by using H\"{o}lder's inequality,
the estimate $\eqref{f:3.12}$, and the observation $\partial_t W
=\mathcal{A}W$, as well as,
Stokes semigroup estimates (see e.g. \cite[Chapter 5]{Tsai18}), in the order, we obtain that
\begin{equation*}
\begin{aligned}
\bigg(
\int_{Y_f\setminus(Y_f)_{\rho}}|\Delta W(\cdot,t)|^2
\bigg)^{1/2}
&\lesssim\rho^{\frac{1}{2}-\frac{1}{p}}
\bigg(
\int_{Y_f\setminus(Y_f)_{\rho}}|\nabla^2 W(\cdot,t)|^p
\bigg)^{1/p}\\
&\lesssim
\rho^{\frac{1}{2}-\frac{1}{p}}
\bigg(
\int_{Y_f}|\partial_t W(\cdot,t)|^p
\bigg)^{1/p}
=\rho^{\frac{1}{2}-\frac{1}{p}}
\bigg(
\int_{Y_f}|\mathcal{A} W(\cdot,t)|^p
\bigg)^{1/p}\lesssim \rho^{\frac{1}{2}-\frac{1}{p}}t^{-1},
\end{aligned}
\end{equation*}
which is the desired estimate $\eqref{pde2.26}$.
We have completed the whole proof.
\end{proof}

\subsection{\centering Proof of Proposition $\ref{prop2.1}$}\label{subsec:3.3}

\begin{corollary}[weighted estimates]\label{lemma3.4}
Let $0<T<\infty$, and $2<p<\infty$ with $\vartheta:= \frac{p}{3p-2}$.
Suppose that $(W_j,\pi_j)$ is a weak solution of \eqref{pde2.1}
with $j=1,\cdots,d,$ and the condition $\int_{Y_f}\pi(\cdot,t)=0$ for any $t> 0$.
Then, for any $\alpha>2\vartheta$,
we have a refined weighted estimate
\begin{equation}\label{pri:3.9}
\bigg(\int_{0}^{T}dt\int_{Y_f}|\partial_t W_j(\cdot, t)|^2t^{\alpha}\bigg)^{1/2}
+ \bigg(\int_{0}^{T}dt\int_{Y_f}|\pi_j(\cdot, t)|^2t^{\alpha}\bigg)^{1/2}
\lesssim 1,
\end{equation}
where the up to constant depends on $d$, $\alpha$, $T$, and the character of
$Y_f$.
\end{corollary}

\begin{proof}
The advantage of the present proof avoids using advanced analysis results\footnote{The case $\alpha =1$ simply come from a testing function argument, and
we pursue a consistency even in this refined estimate, although there exists
a way to save effort on the pressure term.}.
We firstly establish the weighted estimate for $\partial_t W_j$, and then translate
the same type estimate to the pressure term $\pi_j$. To do so, we take
$\partial_t W_j$ as the test function for the equations $\eqref{pde2.1}$, and for any $t>0$ there holds:
\begin{equation*}
\int_{Y_f}|\partial_t W_j(\cdot,t)|^2
+ \frac{1}{2}\frac{d}{dt}\int_{Y_f}|\nabla W_j(\cdot,t)|^2 = 0.
\end{equation*}
Integrating both sides of the above equation with respect to temporal variable from $\tau$ to $T$,
and then appealing to Lemma $\ref{lemma3.3}$, we obtain that
\begin{equation*}
 \int_{\tau}^{T}\int_{Y_f}|\partial_t W_j|^2
 \leq \frac{1}{2}\int_{Y_f}|\nabla W_j(\cdot,\tau)|^2
 \lesssim \tau^{-2\vartheta}.
\end{equation*}
Multiplying the factor $\tau^{\alpha-1}$ on the both sides above, and then integrating with respect to
$\tau$ from $0$ to $T$, one can further derive that
\begin{equation*}
\int_{0}^{T}d\tau
\bigg(\tau^{\alpha-1}\int_{\tau}^{T}\int_{Y_f}|\partial_t W_j|^2\bigg)
 \lesssim T^{\alpha-2\vartheta}.
\end{equation*}
By using Fubini's theorem in the left-hand side above, i.e.,
\begin{equation*}
\int_{0}^{T}d\tau
\bigg(\tau^{\alpha-1}\int_{\tau}^{T}\int_{Y_f}|\partial_t W_j|^2\bigg)
= \int_{0}^{T}dt\int_{Y_f}|\partial_t W_j(\cdot,t)|^2t^{\alpha},
\end{equation*}
we have established the following estimate
\begin{equation}\label{f:3.13}
 \int_{0}^{T}dt\int_{Y_f}|\partial_t W_j(\cdot,t)|^2t^{\alpha} \lesssim T^{\alpha-2\vartheta}.
\end{equation}

We now continue to study the pressure term, and begin with constructing
an auxiliary function associated with Bogovskii's operator, i.e., for any $t>0$, we have
\begin{equation*}
\left\{
\begin{aligned}
\nabla\cdot v(\cdot,t)&=\pi(\cdot,t),&\quad&\text{in}~Y_f;\\
v(\cdot,t)&=0,&\quad&\text{on}~\partial Y_s,\quad
y\mapsto
v(y,t)~\text{is 1-periodic}
\end{aligned}\right.
\end{equation*}
with the estimate $\|v(\cdot,t)\|_{H^1(Y_f)}
\lesssim\|\pi(\cdot,t)\|_{L^2(Y_f)}$.
By taking $v$ as a test function acting on \eqref{pde2.1}, it is not hard to
get
\begin{equation*}%\label{pde2.17}
\int_{Y_f}|\pi_j(\cdot,t)|^2\lesssim\int_{Y_f}|\partial_t W_j(\cdot,t)|^2.
\end{equation*}

Consequently, multiplying $t^\alpha$ on the both sides above and then integrating it
from $0$ to $T$, we obtain
\begin{equation*}
\int_{0}^{T}dt\int_{Y_f}|\pi_j(\cdot,t)|^2t^\alpha
\lesssim \int_{0}^{T}dt\int_{Y_f}|\partial_t W_j(\cdot,t)|^2t^\alpha
\overset{\eqref{f:3.13}}{\lesssim} T^{\alpha-2\vartheta}.
\end{equation*}
This together with $\eqref{f:3.13}$ leads to
the desired estimate $\eqref{pri:3.9}$, and ends the whole proof.
\end{proof}

\begin{corollary}\label{cor:1}
Let $\vartheta$ be given as in Corollary $\ref{lemma3.4}$.
Suppose that $(W_j,\pi_j)$ is a weak solution of \eqref{pde2.1}
with $j=1,\cdots,d$. Then,
for any $1<q<\infty$,
there exists a constant $\lambda\in(0,1]$ as the same as in
Lemma $\ref{lemma2.2}$ such that
a refined decay estimate
\begin{equation}\label{pri:3.10}
\bigg(\int_{Y_f}|\nabla W_j(\cdot, t)|^q\bigg)^{1/q}
\lesssim t^{-\frac{1}{2}+\lambda(\frac{1}{2}-\vartheta)}
\end{equation}
holds for any $t>0$, which further leads to
\begin{equation}\label{pri:3.11}
\|W_j\|_{L^{2}(0,T;W^{1,q}(Y_f))}\lesssim 1,
\end{equation}
where the up to constant depends on $d$, $q$, $\lambda$, $T$, and the character of $Y_f$.
\end{corollary}

\begin{proof}
Using the same arguments as given for Lemma $\ref{lemma2.2}$, we can
derive the estimate $\eqref{pri:3.10}$ from the interpolation between the
refined estimate $\eqref{pri:3.6}$ and the semigroup estimates (see e.g. 
\cite[Lemma 5.1]{Tsai18}), while the stated estimate $\eqref{pri:3.11}$ directly follows from
$\eqref{pri:3.10}$ coupled with Poincare's inequality (by noting zero-boundary
condition of $W_j$ for $t>0$), and this completes the proof.
\end{proof}

\begin{lemma}[antisymmetry and regularities]\label{flux}
For any $0<t<T$, let $B(\cdot,t)=\{b_{ij}(\cdot,t)\}_{1\leq i,j\leq d}$
be given as in Propostion $\ref{prop2.1}$. Then there hold
the structure properties:
\begin{equation}\label{eq:3.1}
\emph{(i)}~~\nabla\cdot B(\cdot,t)=0;\qquad\qquad
\emph{(ii)}~~\int_{Y}B(\cdot,t)=0.
\end{equation}
Moreover, there exists $\Phi(\cdot,t)=\{\Phi_{ki,j}(\cdot,t)\}_{1\leq i,j,k\leq d}$
with $\Phi_{ki,j}(\cdot,t)\in H^{1}_{\emph{loc}}(\mathbb{R}^d)$
being 1-periodic, and satisfying
\begin{equation}\label{pde:3.5}
 \nabla \cdot \Phi(\cdot, t) =
 B(\cdot,t)\quad \text{in}\quad Y
\end{equation}
under the antisymmetry condition (i.e., $\Phi_{ki,j}=-\Phi_{ik,j}$).
Also, for any $1<q<\infty$, we have the regularity estimate
\begin{equation}\label{pri:3.8}
\|\Phi(\cdot,t)\|_{W^{1,q}(Y)}\lesssim \|B(\cdot,t)\|_{L^q(Y)},
\end{equation}
where the up to constant is independent of $t$.
\end{lemma}

\begin{proof}
Here we merely treat temporal variable as a parameter. In this regard,
once the properties $\eqref{eq:3.1}$ had been verified, the existence of
the solution $\Phi$ of the equation $\eqref{pde:3.5}$ would be established as
the same as the stationary case
%\footnote{See Z. Shen' work
%in \cite{Shen20}.}
under the antisymmetry condition. As a consequence,
it would similarly satisfies the regularity
estimate $\eqref{pri:3.8}$. Thus, the main job is only left to check the equalities
in $\eqref{eq:3.1}$. Recalling the formula of $B$ in Proposition $\ref{prop2.1}$,
the equality (i) in $\eqref{eq:3.1}$ follows from the divergence-free condition of $\tilde{W}$,
while the equality (ii) directly comes from the definition of the effective matrix $A$ (see $\eqref{pde2.2}$).
\end{proof}

\emph{The main structure} of the proof of Proposition $\ref{prop2.1}$
is presented by the following flow chart.
\begin{figure}[H]%[!htb]
  \centering
\begin{tikzpicture}[->,>=stealth',global scale=0.8]

\node[state,text width=3cm] (Lemma1)
{
Lemma $\ref{lemma3.3}$
};

\node[state,
  below of=Lemma1,
  yshift=-3cm,
  anchor=center,
  text width=3cm] (Lemma2)
{
Lemma $\ref{lemma2.2}$
};

\node[state,
  right of=Lemma1,
  node distance=7cm,
  anchor=center,
  text width=3cm] (Lemma3)
{
Corollary $\ref{lemma3.4}$.
};

\node[state,
  below of=Lemma3,
  node distance=2cm,
  anchor=center,
  text width=3cm] (Lemma4)
{
Corollary $\ref{cor:1}$.
};

\node[state,
  below of=Lemma4,
  node distance=2cm,
  anchor=center,
  text width=3cm] (Lemma5)
{
Lemma $\ref{flux}$.
};

%\node[state,
%  right of=Lemma3,
%  node distance=6cm,
%  anchor=center,
%  text width=3cm] (Lemma6)
%{
%Proposition $\ref{prop2.1}$.
%};

\node[state,
  right of=Lemma5,
  node distance=7cm,
  anchor=center,
  text width=3cm] (Lemma7)
{
Proposition $\ref{prop2.1}$.
};

\path (Lemma1) edge (Lemma2)
      (Lemma1) edge (Lemma3)
      (Lemma1) edge (Lemma4)
      (Lemma1) edge[bend right = 20] node[anchor=left,right]{
 %     $+$ Lemma $\ref{lemma2.5}$
      }  (Lemma5)
      (Lemma2) edge node[anchor=top,below]{$+$ Lemma $\ref{lemma2.5}$ } (Lemma5)
      (Lemma3) edge [bend left] (Lemma7)
      (Lemma4) edge (Lemma7)
      (Lemma5) edge (Lemma7)
      (Lemma2) edge[bend right=15] (Lemma7);
\end{tikzpicture}
  \caption{\small The proof structure of Proposition \ref{prop2.1}}\label{DD}
\end{figure}

\noindent\textbf{The proof of Proposition \ref{prop2.1}.}
In terms of the equations $\eqref{pde2.1}$ within a finite time,
the existence of the weak solution $(W_j,\pi_j)$, as well as,
the energy estimate $\eqref{pde2.14}$ had been well known
for a long time (see e.g. \cite[Chapter3]{Temam79}).
Based upon the estimate $\eqref{pri:3.6}$ stated in Lemma \ref{lemma3.3},
we have derived weighted estimates $\eqref{pri:3.9}$ in Corollary \ref{lemma3.4},
which is in fact one part of \eqref{pde3.33a}. The other part of $\eqref{pde3.33a}$
concerns the estimate on the quantity $\|\nabla W_j\|_{L^1(0,T;L^q(Y_f))}$
for $1<q<\infty$, which has been shown in Corollary $\ref{cor:1}$.
Then we turn to the estimate $\eqref{pde3.33b}$, and it follows from
the improved estimate $\eqref{pri:3.7}$ stated in Lemma $\ref{lemma2.2}$,
which gives us the estimate on the first term of $\eqref{pde3.33b}$.
Once we had the estimate on $\|\partial_t W_j\|_{L^1(0,T;L^q(Y_f))}$,
the other two terms of $\eqref{pde3.33b}$ would be derived from
the estimate $\eqref{f:3.12}$. In the end, we address the flux corrector.
Its existence and antisymmetry properties in $\eqref{anti-sym}$
have been shown in Lemma $\ref{flux}$. On account of the equation
$\eqref{pde:3.5}$ and the estimate $\eqref{pri:3.8}$,
the desired estimates $\eqref{pri:3.5a}$ and $\eqref{pri:3.5b}$
are reduced to showing the related decay estimates of $\|B(\cdot,t)\|_{L^q(Y_f)}$
and $\|\partial_t B(\cdot,t)\|_{L^q(Y_f)}$. By
the definition of $B=\{b_{ij}(\cdot,t)\}$ in Proposition \ref{prop2.1}, it suffices to show the associated
decay estimates on $\|W(\cdot,t)\|_{L^q(Y_f)}$,
$\|\partial_tW(\cdot,t)\|_{L^q(Y_f)}$, which can been found in
Corollary $\ref{cor:1}$ and
Lemma $\ref{lemma2.2}$, respectively. The conclusion
of $\Phi_{ki,j}\in C([0,T];C_{per}(Y))$ follows from the
estimate $\eqref{pri:3.5b}$ for the case of $q>d$ (see e.g.
\cite[Theorem 2 in Subsection 5.9.2]{Evans10}).
We have completed the whole proof.
\qed

\section{Boundary-layer correctors}\label{sec4}

\noindent
Before proceeding the concrete results of boundary-layer correctors that we have introduced in Subsection $\ref{subsec:1.2}$,
we further explain the source of $J_1$ and $J_2$ presented in  $\eqref{pde3.8}$. Recalling the formula $\eqref{pde:B}$,
if computed $w_\varepsilon^{(2)}$ directly, we would find the following formula:
\begin{equation*}
\nabla \cdot w_\varepsilon^{(2)}
= \underbrace{\nabla\psi_{\varepsilon}\cdot(W^{\varepsilon}\ast G+\varepsilon\phi^{\varepsilon}\ast_2\partial G)}_{\text{layer~part}}
+\underbrace{\psi_{\varepsilon}
\frac{A}{|Y_f|}\ast_2\partial G
+\varepsilon\psi_{\varepsilon}
\phi^{\varepsilon}\ast_3\partial^2G.}_{\text{co-layer~part}}
\end{equation*}
By virtue of $\nabla\cdot u_0 = 0$ in $\Omega\times(0,T)$, there is no big problem on co-layer part, while the layer part is problematic.
Fortunately having \emph{flux corrector}, it inspires us to introduce the quantity $\nabla\psi_\varepsilon\cdot A*G$ into the layer part
(see Remark $\ref{remark:3.1}$). Therefore, the present form of $J_1$ and $J_2$ in $\eqref{pde3.8}$ is cogent.

We now state the results of boundary-layer correctors:

\begin{proposition}[Boundary-layer corrector I]\label{lemma3.6-1}
Let $0<T<\infty$. Given $f\in L^2(0,T;C^{1,\frac{1}{2}}(\bar\Omega)^d)$, assume
the same geometry assumptions on perforated domains as in Theorem
$\ref{thm1}$.
Let $J_1$ and $J_2$ be given as in $\eqref{pde3.8}$. Then, for a.e. $t\geq 0$,
there exists at least
one weak solution to
\begin{equation}\label{divc2}
\left\{
\begin{aligned}
\nabla\cdot\xi_{\varepsilon}(\cdot,t)&=\frac{\partial J_2}{\partial t} +\sum_{i} (\dashint_{O_{\varepsilon}^i}\frac{\partial J_1}{\partial t})1_{O_{\varepsilon}^i},&
\text{in}&\quad \Omega_{\varepsilon};\\
\xi_{\varepsilon}(\cdot,t)&=0,&\text{on}&\quad \partial \Omega_{\varepsilon},
\end{aligned}\right.
\end{equation}
where we recall that $1_{O_{\varepsilon}^i}$ is the indicator function of $O_{\varepsilon}^i$, and $\{O_{\varepsilon}^i\}$ is a family of the non-overlap subsets of $O_\varepsilon$, satisfying $\eqref{k-5}$. Also,
the solution satisfies the following estimate
\begin{equation}\label{pri:4.1}
 \|\xi_{\varepsilon}\|_{L^2(0,T;L^2(\Omega_\varepsilon))}
 + \varepsilon\|\nabla\xi_{\varepsilon}\|_{L^2(0,T;L^2(\Omega_\varepsilon))}
 \lesssim \varepsilon^{1/2}\|f\|_{L^2(0,T;C^{1,1/2}(\bar{\Omega}))},
\end{equation}
in which the up to constant depends on $d$, $|Y_f|$, $T$, and the characters of $Y_f$ and $\Omega$.
\end{proposition}

\begin{proposition}[Boundary-layer corrector II]\label{lemma3.6-2}
Assume the same conditions as in Theorem $\ref{thm1}$.
Let $J_1$ be given as in $\eqref{pde3.8}$. Then,
for a.e. $t\geq 0$,  there exists at least
one weak solution to
\begin{equation}\label{divc3}
\left\{
\begin{aligned}
\nabla\cdot\eta_{\varepsilon}(\cdot,t)&=\frac{\partial J_1}{\partial t} -\sum_{i} (\dashint_{O_{\varepsilon}^i}\frac{\partial J_1}{\partial t})1_{O_{\varepsilon}^i},&
\text{in}&\quad O_{\varepsilon};\\
\eta_{\varepsilon}(\cdot,t)&=0,&\text{on}&\quad \partial O_{\varepsilon},
\end{aligned}\right.
\end{equation}
satisfying the following estimate
\begin{equation}\label{pri:5.2}
 \|\eta_{\varepsilon}\|_{L^2(0,T;L^2(\Omega_\varepsilon))}
 + \varepsilon\|\nabla\eta_{\varepsilon}\|_{L^2(0,T;L^2(\Omega_\varepsilon))}
 \lesssim \varepsilon^{1/2}\|f\|_{L^2(0,T;C^{1,1/2}(\bar{\Omega}))},
\end{equation}
where the up to constant depends on $d$, $|Y_f|$, $T$, and the characters of $Y_f$ and $\Omega$.
\end{proposition}

\begin{proposition}[Boundary-layer corrector III]\label{lemma3.6-3}
Suppose $\xi_{\varepsilon}$ and $\eta_{\varepsilon}$ are the two solutions of $\eqref{divc2}$ and $\eqref{divc3}$ given in Propositions $\ref{lemma3.6-1}$ and $\ref{lemma3.6-2}$, respectively.
Let $\hat{\xi_{\varepsilon}}$ and $\hat{\eta_{\varepsilon}}$ be defined as in $\eqref{pde3.11}$\footnote{Note that
$\eta_{\varepsilon}$ can be trivially zero-extended to the whole region $\Omega$.}:
Moreover, there also hold the regularity estimates:
\begin{subequations}
\begin{align}
\|\hat\xi_{\varepsilon}\|_{L^2(0,T;L^2(\Omega_\varepsilon))}
+ \varepsilon\|\nabla\hat\xi_{\varepsilon}\|_{L^2(0,T;L^2(\Omega_\varepsilon))}
&\lesssim \varepsilon^{1/2}\|f\|_{L^2(0,T;C^{1,1/2}(\bar\Omega))}; \label{pri:5.3a} \\
\|\hat\eta_{\varepsilon}\|_{L^2(0,T;L^2(\Omega_\varepsilon))}
+ \varepsilon\|\nabla\hat\eta_{\varepsilon}\|_{L^2(0,T;L^2(\Omega_\varepsilon))}
&\lesssim \varepsilon^{1/2}\|f\|_{L^2(0,T;C^{1,1/2}(\bar\Omega))}, \label{pri:5.3b}
\end{align}
\end{subequations}
where the multiplicative constant depends on $d$, $|Y_f|$, $T$, and the characters of
$Y_f$ and $\Omega$.
\end{proposition}

It is well known that for the problem: $\nabla\cdot v = g$ in $\Omega$;
and $v=0$ on $\partial\Omega$ with the compatibility condition $\int_{\Omega}f=0$,
there exists at least one solution $v\in H^1_0(\Omega)^d$ together with a constant $C$, depending on $\Omega$, such that
$\|v\|_{H^1_0(\Omega)}\leq C\|f\|_{L^2(\Omega)}$
(see \cite[lemma III.3.1]{Galdi11}).
Therefore, if replaced by a perforated domain,
the constant $C$ will additionally depend on the size $\varepsilon$ of holes in the perforated domain,
which has been stated below.
\begin{theorem}[Bogovskii's operator on perforated domains
\cite{Chechkin-Piatnitski-Shamaev2007}]\label{thm4}
Assume the same geometry assumptions on the perforated domain
as in Theorem $\ref{thm1}$. Then, for any $g\in L^2(\Omega_{\varepsilon})$ with
$\int_{\Omega_{\varepsilon}}g =0$,
there exists a vector-valued function
$v_{\varepsilon}\in H^{1}_0(\Omega_{\varepsilon})^d$
such that $\nabla\cdot v_{\varepsilon} =g$ in $\Omega_\varepsilon$, satisfying
\begin{equation}\label{pde3.21}
\| v_{\varepsilon}
\|_{L^2(\Omega_{\varepsilon})}
\leq C\varepsilon\|\nabla v_{\varepsilon}
\|_{L^2(\Omega_{\varepsilon})}
\leq C
\|g\|_{L^2(\Omega_{\varepsilon})},
\end{equation}
where the constant $C$ is independent of $\varepsilon$ and $g$.
\end{theorem}

By virtue of Theorem $\ref{thm4}$,
the existence of the solution $\xi_{\varepsilon}$ to $\eqref{divc2}$ is reduced to
verifying compatibility condition, while the desired estimate $\eqref{pri:4.1}$  follows from the estimates on the quantities
\begin{equation}\label{f:1}
\bigg(
\int_{0}^{T}dt\int_{\Omega_{\varepsilon}}
\big|\frac{\partial J_2}{\partial t}\big|^2\bigg)^{1/2}
\quad\text{and}\quad
\bigg(\int_{0}^{T}dt\int_{\Omega_{\varepsilon}}\Big|
\sum_{i}(\dashint_{O_{\varepsilon}^i}\frac{\partial J_1}{\partial t}) 1_{O_{\varepsilon}^i}\Big|^2
\bigg)^{1/2},
\end{equation}
which will be addressed in Lemmas \ref{lemma:5.1} and \ref{*lemma4.1},   separately.

\emph{The key ideas} are summarized as follows:
\begin{itemize}
  \item By introducing \textbf{radial cut-off function}, together with the special structure
of the effective solution $\eqref{pde1.2}$ on the boundary, i.e.,
$\vec{n}\cdot u_0 = 0$ on $\partial\Omega\times(0,T)$, it is possible to
produce a desired smallness near the boundary, simply by using Poincar\'e's inequality (see Lemma \ref{lemma:5.1}), which is a crucial observation
for Proposition $\ref{lemma3.6-1}$.
  \item By \textbf{decomposing the boundary layer region},
      the proof of Proposition $\ref{lemma3.6-2}$
\textbf{does not} rely on Theorem $\ref{thm4}$.
The solution $\eta_{\varepsilon}$ to the equation
$\eqref{divc3}$ consists of ``piecewise'' solutions of cell problems,
according to the decomposed element $O_\varepsilon^i$.
On the other hand,  due to $O_\varepsilon^{i}$ obtained by cutting
the boundary layer region $O_\varepsilon$ in the normal direction $\vec{n}$ associated with $\partial\Omega$, this enables us to consistently apply the antisymmetry property of the flux corrector (see Lemma $\ref{*lemma4.1}$), which is also important to the proof of Proposition $\ref{lemma3.6-1}$.
\end{itemize}

%\begin{remark}
%Although the estimates $\eqref{pri:5.3a}$ and $\eqref{pri:5.3b}$ can not treated as corollaries of Propositions $\ref{lemma3.6-1}$ and $\ref{lemma3.6-2}$,
%we use the same philosophy in the related computations (see Subsection $\ref{subsec:4.4}$).
%\end{remark}

\subsection{\centering Radial cut-off functions}\label{subsec:4.2}

\noindent
We first describe the important concept of this paper: radial cut-off  function, and then we will discuss the decomposition of the boundary layer $O_\varepsilon$ with details.

\begin{lemma}[radial cut-off functions]\label{lemma:cut-off}
Let $\Omega\subset\mathbb{R}^d$ be a bounded $C^2$ domain. Then, for any $0<\varepsilon\ll 1$,
there exists a cut-off function $\psi_{\varepsilon}$ satisfying
the following properties:
\emph{\begin{itemize}
\item[\text{(a).}] $\psi_{\varepsilon}\in C^\infty_0(\Omega)$
with supp$(\psi_{\varepsilon})=\Sigma_{\varepsilon}$,
$\psi_{\varepsilon}=1$ on $\Sigma_{2\varepsilon}$ and
$0\leq \psi_{\varepsilon}\leq 1$ (where $\Sigma_\varepsilon$ is defined in Subsection $\ref{notation}$);
\item[\text{(b).}] For any $x$ with $\text{dist}(x,\partial\Omega)<10\varepsilon$, there exists
a unique element $\tilde{x}\in\partial\Omega$ such that dist$(x,\partial\Omega)=|x-\tilde{x}|$ and $\tilde{x}\in \partial\Omega$, and there also holds
\begin{equation}\label{eq:5.1}
\nabla \psi_{\varepsilon}(x)=-|\nabla \psi_{\varepsilon}(x)|\vec{n}(\tilde{x})
\quad \text{on}~O_\varepsilon,
\end{equation}
where $O_{\varepsilon}:=$supp$(\nabla\psi_{\varepsilon})\subset \Omega\setminus\Sigma_{2\varepsilon}$; Also, we have
$|\nabla\psi_{\varepsilon}(x)|=|\nabla\psi_{\varepsilon}(y)|$, provided
dist$(x,\partial\Omega)=$dist$(y,\partial\Omega)$ for $x,y\in O_{\varepsilon}$;
\item[\text{(c).}] $|\nabla\psi_{\varepsilon}|
    \lesssim\frac{1}{\varepsilon}$ and $|\nabla^2\psi_{\varepsilon}|\lesssim
    \frac{1}{\varepsilon^2}$.
\end{itemize}}
\end{lemma}

\begin{proof}
We initially construct the radial function as follows:
\begin{equation*}
\phi_\varepsilon(r) =\left\{\begin{aligned}
&0, &&0\leq r\leq (4/3)\varepsilon;\\
&\frac{3}{\varepsilon}\big(r-\frac{4}{3}\varepsilon\big),
&&(4/3)\varepsilon< r\leq (5/3)\varepsilon;\\
&1, &&r>(5/3)\varepsilon.
\end{aligned}\right.
\end{equation*}
Let $g_\varepsilon(r):={\varrho}_{\frac{\varepsilon}{3}}\ast\phi_\varepsilon$,
where the kernel $\varrho\in C^\infty(\mathbb{R})$ is the 1-dimensional standard
mollifier\footnote{The mollifier $\varrho(r):=C\exp\{-\frac{1}{1-r^2}\}$
for $r<1$, and vanishes
for $r\geq 1$, in which the constant $C$ is such that
$\int\varrho = 1$.}. Thus, we have $g_\varepsilon\in C^\infty(\mathbb{R})$, and
$0\leq g_\varepsilon \leq 1$ with $g_\varepsilon(r) = 0$ if $r\leq \varepsilon$;
and $g_\varepsilon(r)=1$ if $r>2\varepsilon$. Also, we have
$|\frac{d^kg_\varepsilon}{dr^k}|\lesssim 1/\varepsilon^k$ for any positive integer $k$.
Now, we introduce the distance function $\delta(x) :=\text{dist}(x,\partial\Omega)$
and then the desired radial cut-off function
can be defined by $\psi_\varepsilon(x) := g_\varepsilon(\delta(x))$ on $\Omega$.
It is not hard to see that $\text{supp}(\psi_\varepsilon) = \Sigma_\varepsilon$,
$\psi_\varepsilon \equiv 1$ on $\Sigma_{2\varepsilon}$, and
$0\leq \psi_\varepsilon\leq 1$. Also, we denote the compact support of
$\nabla\psi_\varepsilon$ by $O_\varepsilon$, which is included in
the layer type region $\Sigma_{2\varepsilon}\setminus\Sigma_\varepsilon$.
Since $\Omega$ is assumed to be a bounded $C^2$ domain
and $0<\varepsilon\ll 1$, for any $x\in O_\varepsilon$
there exists a unique point $\tilde{x}\in\partial\Omega$ such that $\delta(x)
=|x-\tilde{x}|$. Thus, we can find that
\begin{equation}\label{f:4.7}
  \nabla\psi_\varepsilon(x) = g_\varepsilon'(\delta(x))\nabla\delta(x)
  = -{g}_\varepsilon'(\delta(x))\vec{n}(\tilde{x})
\end{equation}
for any $x\in O_\varepsilon$, where $g_\varepsilon'$ represents the derivative of $g_\varepsilon$.
This further implies
\begin{equation}\label{f:4.8}
 |g_\varepsilon'(\delta(x))| =|\nabla\psi_\varepsilon(x)|
 \quad \text{and} \quad
 |\nabla\psi_\varepsilon(x)| = |g_\varepsilon'(\delta(x))|
 = |g_\varepsilon'(\delta(y))|
 = |\nabla\psi_\varepsilon(y)|,
\end{equation}
whenever $\delta(x) = \delta(y)$. Combining the equalities
$\eqref{f:4.7}$ and $\eqref{f:4.8}$ leads to the stated conclusions
in (b). Moreover, it is clear to see $|\nabla\psi_\varepsilon|\lesssim 1/\varepsilon$,
and there holds
\begin{equation*}
 \nabla^2\psi_\varepsilon
 = g_\varepsilon''(\delta)\nabla\delta \otimes\nabla\delta
% + \frac{d^2g_\varepsilon}{dr^2}(\delta(x))\nabla^2\delta(x)
 \quad\text{in}\quad O_\varepsilon,
\end{equation*}
where $g_\varepsilon''$ represents the second-order derivative of $g_\varepsilon$, and this gives us $|\nabla^2\psi_\varepsilon(x)| \leq |g_\varepsilon''(\delta(x))|
 \lesssim 1/\varepsilon^2$. We have completed the proof.
\end{proof}

\begin{remark}\label{remark:4.1}
\emph{As mentioned in Proposition $\ref{subsec:1.2}$,
$\{O_{\varepsilon}^i\}$ is a family of the non-overlap subsets of $O_\varepsilon$, such that
\begin{equation*}
\begin{aligned}
&O_\varepsilon = \bigcup_{i}O_\varepsilon^{i}
\quad\text{and}\quad
|O_\varepsilon^i| \sim \varepsilon^d,
\end{aligned}
\end{equation*}
in which $O_\varepsilon^{i}$ is an approximately $d$-dimensional cube obtained by cutting $O_\varepsilon$ in the normal direction $\vec{n}$ associated with $\partial\Omega$ (see Figure $\ref{subfig_1}$ for an example). The boundary of $O_\varepsilon^i$ consists of two parts:
\begin{itemize}
  \item[(a).] $\{R_j^i\}:= \partial O_\varepsilon^i\cap \partial O_\varepsilon$. Let $\vec{n}_R$ be the unit normal vector of $R_j^i$, satisfying $\vec{n}_R \perp R_j^i$. Then, there holds
      \begin{equation*}
         \vec{n}_R \pll \vec{n} \quad\text{on}\quad R_j^i.
      \end{equation*}
  \item[(b).]
$\{S_j^i\}:= \partial O_\varepsilon^i\cap O_\varepsilon$. Let $\vec{n}_S$ be the unit normal vector of $S_j^i$,
satisfying $\vec{n}_S \perp S_j^i$. Due to $S_j^i$ obtained by the cutting in the normal direction $\vec{n}$, we have the following relationship:
\begin{equation}\label{k-6}
 \vec{n}_S \perp \vec{n} \quad\text{on}\quad S_j^i.
\end{equation}
\end{itemize}}
\end{remark}

%\begin{figure}[H]
%\centering  %图片全局居中
%\subfigure[$d=2$]{
%\label{Fig.sub.1}
%\includegraphics[width=0.7\textwidth]{P6}}
%\subfigure[$d=3$]{
%\label{Fig.sub.2}
%\includegraphics[width=0.8\textwidth]{P2}}
%\caption{\small Typical cases}
%\label{Fig.main}
%\end{figure}

\begin{figure}
  \centering
  \includegraphics[width=0.65\textwidth]{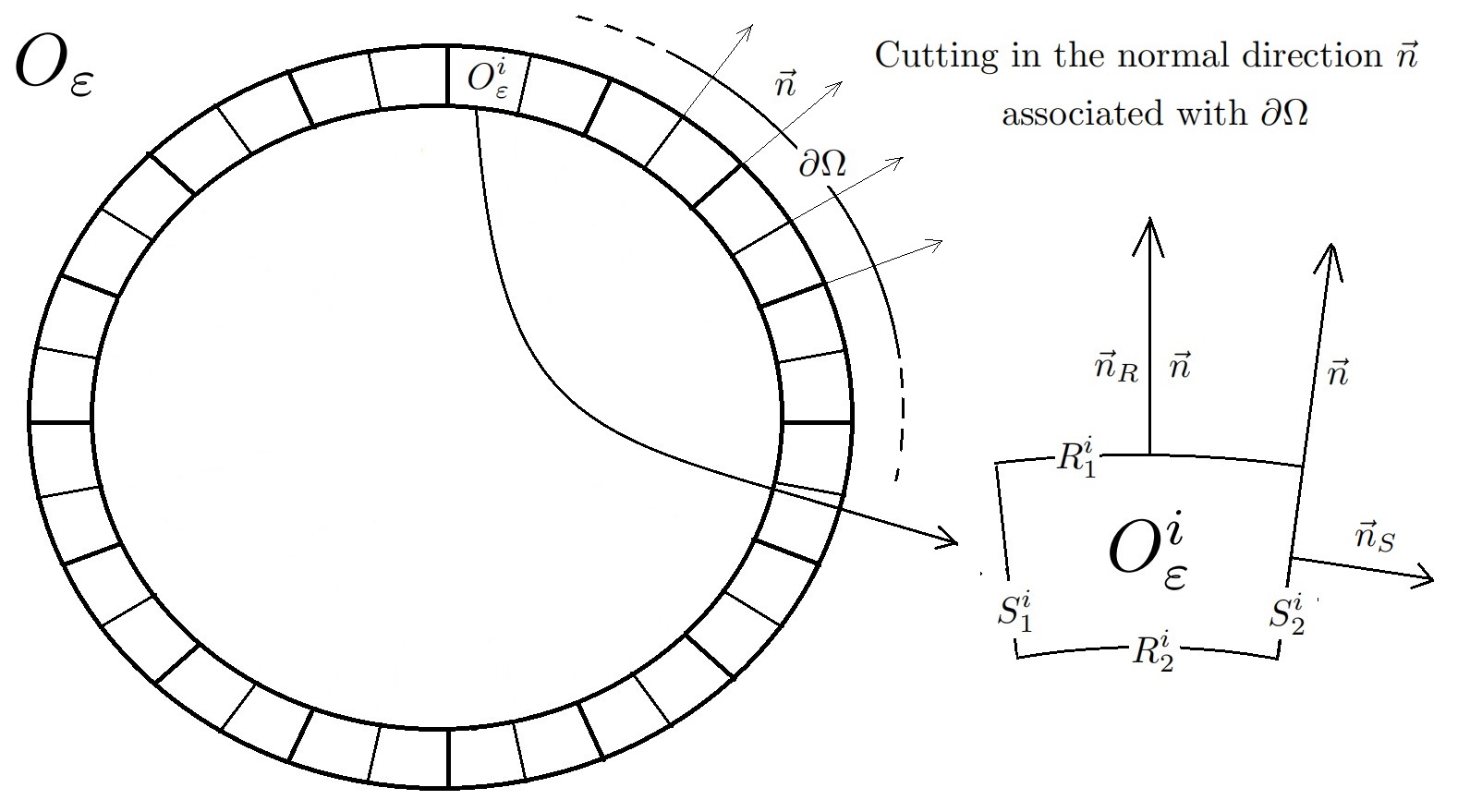}
  \caption{\small A typical example for $d=2$}\label{subfig_1}
\end{figure}

\subsection{\centering Proof of Proposition $\ref{lemma3.6-1}$}\label{subsec:4.3}

\begin{lemma}[estimates for $J_2$]\label{lemma:5.1}
Let $\psi_\varepsilon$ be the radial cut-off function defined in
Lemma $\ref{lemma:cut-off}$, and $J_2$ be given
as in $\eqref{pde3.8}$. Assume the same conditions as in
Proposition $\ref{lemma3.6-1}$. Then,
the first term in the right-hand side of $\eqref{divc2}$ satisfies the following estimate
\begin{equation}\label{pri:5.1}
\begin{aligned}
\bigg(
\int_{0}^{T}dt\int_{\Omega_{\varepsilon}}
&\big|\frac{\partial J_2}{\partial t}\big|^2
\bigg)^{1/2}
\lesssim
\varepsilon^{\frac{1}{2}}
\Big(\|\partial_t \phi\|_{L^1(0,T;L^2(Y))}
+\|\phi(\cdot,0)\|_{L^2(Y_f)}
\Big)
\|(\nabla G,\varepsilon^{\frac{1}{2}}\nabla^2G)\|_{L^2(0,T;
L^{\infty}(\text{supp}(\psi_{\varepsilon})))}\\
&+
\Big(\|\partial_tA\|_{L^1(0,T)}+|A(0)|\Big)
\bigg\{\varepsilon^{\frac{1}{2}}\|\nabla F\|_{L^2(0,T;C(\bar{\Omega}))}
+\|\nabla(G-F)\|_{
L^2(0,T;L^2{(\text{supp}(\psi_{\varepsilon}))})}\bigg\},\\
\end{aligned}
\end{equation}
where the multiplicative constant depends on
$d$, $|Y_f|$, and the character of $\Omega$,
but independent of $\varepsilon$.
\end{lemma}

\begin{proof}
By virtue of the equality $\eqref{eq:5.1}$, together with
the structure of the effective solution, we can find a cancellation near the boundary with respect to
outward normal direction, which further provides us with the desired smallness.
For the sake of the convenience, we recall the expression of $J_2$, as well as, the convention in $\eqref{notation:4.2}$,
and denote it by
\begin{equation}\label{f:5.11}
\begin{aligned}
  J_2&=\nabla \psi_{\varepsilon}
  \cdot (A\ast G)
  +\varepsilon\nabla \psi_{\varepsilon}\cdot
  \big(\phi_{\varepsilon}\ast_2 \partial G\big)
  +\psi_{\varepsilon}\frac{A}{|Y_f|}\ast_2 \partial G
  +\varepsilon\psi_{\varepsilon}\phi_{\varepsilon}
  \ast_3\partial^2G\\
  &:=J_{21}+J_{22}+J_{23}+J_{24}.
  \end{aligned}
\end{equation}
We complete the whole arguments by three steps, according to the similarity of the computations.

\textbf{Step 1}. We start from the term $\frac{\partial J_{21}}{\partial t}$,
and it follows from the equality $\eqref{eq:5.1}$ that
\begin{equation}\label{f:5.9}
\begin{aligned}
\frac{\partial J_{21}}{\partial t}
= \nabla \psi_{\varepsilon}\cdot \frac{\partial }{\partial t}( A\ast G)
&= -|\nabla\psi_\varepsilon|\vec{n}\cdot \frac{\partial }{\partial t}( A\ast G)\\
&=-|\nabla\psi_\varepsilon|\vec{n}\cdot \frac{\partial }{\partial t}\big[A\ast (G-F)\big] -|\nabla\psi_\varepsilon|\vec{n}\cdot \frac{\partial }{\partial t}( A\ast F).
\end{aligned}
\end{equation}
By noting the fact that $|\nabla\psi_{\varepsilon}|\lesssim
\varepsilon^{-1}$ (see Lemma $\ref{lemma:cut-off}$), there holds
\begin{equation}\label{f:5.10}
\begin{aligned}
\int_{0}^{T}dt\int_{\Omega_{\varepsilon}}\Big|\frac{\partial J_{21}}{\partial t}\Big|^2
&\overset{\eqref{f:5.9}}{\lesssim} \varepsilon^{-2}
\int_{0}^{T}dt\int_{O_{\varepsilon}}
|\vec{n}\cdot\frac{\partial }{\partial t}( A\ast F)|^2
+\varepsilon^{-2}
\int_{0}^{T}dt\int_{O_{\varepsilon}}
|\vec{n}\cdot\frac{\partial }{\partial t}[ A\ast (G-F)]|^2.
\end{aligned}
\end{equation}
To handle the first term in the right-hand side above,
recalling the effective equations $\eqref{pde1.2}$, we have
$u_0=A\ast F$ with $\vec{n}\cdot u_0=0$ on the boundary $\partial\Omega\times(0,T)$,
whereupon for any $x\in O_\varepsilon$ and $t\geq 0$ there holds
\begin{equation*}
\vec{n}\cdot
\frac{\partial}{\partial t}\big(A\ast F\big)(x,t)
=\vec{n}\cdot \int_{0}^{1}ds\nabla
\frac{\partial}{\partial t}\big(A\ast F\big)(\tilde{x}+s(x-\tilde{x}),t)
\cdot(x-\tilde{x}),
\end{equation*}
where $\tilde{x}\in\partial\Omega$ is such that $\text{dist}(x,\partial\Omega)
=|x-\tilde{x}|$,
which further implies
\begin{equation*}
\int_{O_\varepsilon}\Big|\vec{n}\cdot \frac{\partial}{\partial t}
\big(A\ast F\big)(\cdot,t)\Big|^2
\lesssim \varepsilon^2
\int_{O_\varepsilon}\int_{0}^{1}ds
\big|\nabla \frac{\partial}{\partial t}\big(A\ast F\big)
(\tilde{x}+s(\cdot
-\tilde{x}),t)\big|^2
\leq \varepsilon^2\int_{\Omega\setminus \Sigma_{2\varepsilon}}
|\nabla \frac{\partial}{\partial t}\big(A\ast F\big)(\cdot,t)|^2.
\end{equation*}
Plugging this back into $\eqref{f:5.10}$, and using
Young's inequality,  we obtain that
\begin{equation*}
\begin{aligned}
&\quad\int_{0}^{T}dt\int_{\Omega_{\varepsilon}}\Big|\frac{\partial J_{21}}{\partial t}\Big|^2
\lesssim
\int_{0}^{T}dt
\int_{\Omega}
|\nabla\frac{\partial }{\partial t}( A\ast F)|^2
+\varepsilon^{-2}
\int_{0}^{T}dt
\int_{O_{\varepsilon}}
\big(|(\partial_tA)\ast (G-F)|^2+|A(0)(G-F)|^2\big)\\
&\lesssim
\int_{0}^{T}dt\int_{\Omega\setminus\Sigma_{2\varepsilon}}
\big(|(\partial_tA)\ast \nabla F|^2+|A(0)\nabla F|^2\big)+\varepsilon^{-2}
\int_{0}^{T}dt\int_{O_{\varepsilon}}
\big(|(\partial_tA)\ast (G-F)|^2+|A(0)(G-F)|^2\big)\\
&\lesssim
\big(\|\partial_tA\|_{L^1(0,T)}^2+|A(0)|^2\big)
\Big\{\varepsilon\|\nabla F\|_{L^2(0,T;C(\bar{\Omega}))}^2
+ \varepsilon^{-1}\|G-F\|_{L^2(0,T;L^\infty(\text{supp}
(\psi_{\varepsilon})))}^2 \Big\}.
\end{aligned}
\end{equation*}
This, together with the estimate (Recall the notations $G,F$ in Subsection $\ref{notation}$.)
\begin{equation*}
\|G(\cdot,t)
-F(\cdot,t)\|_{L^\infty(\text{supp}
(\psi_{\varepsilon}))}\lesssim\varepsilon
\|\nabla F(\cdot,t)\|_{C(\bar\Omega)},
\end{equation*}
immediately yields
\begin{equation}\label{*xi2}
\int_{0}^{T}dt\int_{\Omega_{\varepsilon}}\Big|\frac{\partial J_{21}}{\partial t}\Big|^2
\lesssim \varepsilon
\big(\|\partial_tA\|_{L^1(0,T)}^2+|A(0)|^2\big)
\|\nabla F\|_{L^2(0,T;C(\bar{\Omega}))}^2.
\end{equation}

\textbf{Step 2}. We turn to study the term $J_{23}$.
Appealing to
the fact that $\nabla\cdot u_0=A\ast_2\partial F=0$ in $\Omega\times(0,T)$,
there holds
\begin{equation*}
\begin{aligned}
  \int_{0}^{T}dt\int_{\Omega_{\varepsilon}}\Big|
  \frac{\partial J_{23}}{\partial t}
  \Big|^2
  &=
  \int_{0}^{T}dt\int_{\Omega_{\varepsilon}}\Big|
  \frac{\psi_{\varepsilon}}{|Y_f|}
  \frac{\partial}{\partial t}(A\ast_2\partial G)
  \Big|^2
  =\int_{0}^{T}dt\int_{\Omega_{\varepsilon}}\Big|
  \frac{\psi_{\varepsilon}}{|Y_f|}
  \frac{\partial}{\partial t}[A\ast_2\partial (G-F)]\Big|^2\\
  &=\int_{0}^{T}dt\int_{\Omega_{\varepsilon}}\Big|
  \frac{\psi_{\varepsilon}}{|Y_f|}
  \Big(\big[(\partial_tA)\ast_2\partial (G-F)\big]
  +A(0):\partial (G-F)\Big)\Big|^2.
\end{aligned}
\end{equation*}
By using Minkowski's inequality and Young's inequality, the above equality
gives us
\begin{equation}\label{*xi3}
\int_{0}^{T}dt\int_{\Omega_{\varepsilon}}\Big|
\frac{\partial J_{23}}{\partial t}
\Big|^2\lesssim
\big(\|\partial_tA\|^2_{L^1(0,T)}+|A(0)|^2\big)
\|\nabla(G-F)\|^2_{
L^2(0,T;L^2{(\text{supp}(\psi_{\varepsilon}))})}.
\end{equation}

\textbf{Step 3}.
Due to the analogous computations, we consider the terms $J_{22}$ and $J_{24}$ together, and start from
the following estimate:
\begin{equation*}
  \begin{aligned}
  \int_{0}^{T}dt\int_{\Omega_{\varepsilon}}
  \bigg|\frac{\partial J_{22}}{\partial t}+\frac{\partial J_{24}}{\partial t}\bigg|^2
  &\lesssim
  \int_{0}^{T}\int_{\Omega_{\varepsilon}}
  \bigg|\varepsilon\nabla\psi_{\varepsilon}(x)\cdot
  \frac{\partial}{\partial t}(\phi_{\varepsilon}\ast_2\partial G)\bigg|^2 +
  \int_{0}^{T}\int_{\Omega_{\varepsilon}}
 \bigg|\varepsilon\psi_{\varepsilon}\frac{\partial}
  {\partial t}(\phi_{\varepsilon}\ast_3\partial^2 G)\bigg|^2 \\
  &\lesssim
  \int_{0}^{T}\int_{O_{\varepsilon}}
  \big(|(\partial_t\phi_{\varepsilon})\ast_2\partial G|^2
  +|\phi_{\varepsilon}(\cdot,0):\partial G|^2\big) \\
  &\qquad+\varepsilon^2
  \int_{0}^{T}\int_{\text{supp}(\psi_{\varepsilon})}
  \Big(|(\partial_t\phi_{\varepsilon})\ast_3\partial^2G|^2
  +|\phi_{\varepsilon}(\cdot,0)|^2|\nabla^2G|^2\Big).
  \end{aligned}
\end{equation*}
Using Minkowski's inequality and Young's inequality, we have
\begin{equation*}
\begin{aligned}
\int_{0}^{T}dt\int_{O_{\varepsilon}}
  \Big(|(\partial_t\phi_{\varepsilon})\ast_2\partial G(\cdot,t)|^2
&  +|\phi_{\varepsilon}(\cdot,0):\partial G(\cdot,t)|^2\Big)\\
&\lesssim
\varepsilon\Big\{\|\partial_t\phi\|^2_{
L^1(0,T;L^2(Y_f))}
+\|\phi(\cdot,0)\|^2_{L^2(Y_f)}\Big\}
\|\nabla G\|^2_{L^2(0,T;L^{\infty}
(\text{supp}(\psi_{\varepsilon})))},
\end{aligned}
\end{equation*}
and it similarly follows that
\begin{equation*}
\begin{aligned}
\int_{0}^{T}dt\int_{\text{supp}(\psi_{\varepsilon})}
  \Big(|(\partial_t\phi_{\varepsilon})\ast_3\partial^2
  &G(\cdot,t)|^2
  +|\phi_{\varepsilon}(\cdot,0)|^2|\nabla^2G(\cdot,t)|^2\Big)\\
&\lesssim
\varepsilon^2
\Big\{\|\partial_t\phi\|^2_{L^1(0,T;L^2(Y_f))}
+\|\phi(\cdot,0)
\|^2_{L^2(Y_f)}\Big\}
\|\nabla^2G\|^2_{L^2(0,T;L^{\infty}
(\text{supp}(\psi_{\varepsilon})))}.
\end{aligned}
\end{equation*}
Thus, combining the three estimates above, we derive that
\begin{equation}\label{*xi1}
\begin{aligned}
&\quad\int_{0}^{T}dt\int_{\Omega_{\varepsilon}}\Big|\frac{\partial J_{22}}{\partial t}
+\frac{\partial J_{24}}{\partial t}\Big|^2 \\
&\lesssim
\varepsilon\Big(\|\partial_t\phi\|^2_{L^1(0,T;L^2(Y_f))}
+\|\phi(\cdot,0)\|^2_{L^2(Y_f)}\Big)
\|(\nabla G,\varepsilon^{\frac{1}{2}}\nabla^2G)\|^2_{L^2(0,T;L^{\infty}
(\text{supp}(\psi_{\varepsilon})))}.
  \end{aligned}
\end{equation}

Consequently, the desired estimate $\eqref{pri:5.1}$
follows from the estimates $\eqref{f:5.11}$, $\eqref{*xi2}$, $\eqref{*xi3}$,
and $\eqref{*xi1}$. This ends the proof.
\end{proof}

\medskip

\begin{lemma}[estimates for $J_1$]\label{*lemma4.1}
Let $J_1$ be given as in $\eqref{pde3.8}$. Assume the same conditions as in
Proposition $\ref{lemma3.6-1}$. Then, there holds
the estimate
\begin{equation}\label{pri:5.0}
\begin{aligned}
\bigg(\int_{0}^{T}dt
&\int_{\Omega_{\varepsilon}}\Big|
\sum_{i}(\dashint_{O_{\varepsilon}^i}\frac{\partial J_1}{\partial t}) 1_{O_{\varepsilon}^i}\Big|^2
\bigg)^{1/2} \\
\lesssim
&\varepsilon^{\frac{1}{2}}
\bigg\{\|\partial_t \Phi\|_{L^1(0,T;L^{\infty}(Y))}+
\|\Phi(\cdot,0)\|_{L^{\infty}(Y)}
\bigg\}\|\nabla G\|_{L^2(0,T;L^{\infty}
(O_\varepsilon))},
\end{aligned}
\end{equation}
where the multiplicative constant depends on $d$ and the character of $\Omega$, but independent of $\varepsilon$.
\end{lemma}

\begin{proof}
\textbf{Step 1.} Reduction. For a.e. $t\geq 0$,
we manage to establish the following uniformly
bounded estimate with respect to the index $i$, i.e.,
\begin{equation}\label{*pde4.8}
\begin{aligned}
\bigg|\dashint_{O_{\varepsilon}^i}\frac{\partial J_1}{\partial t}(\cdot,t)\bigg|
&\lesssim
\underbrace{\int_{0}^{t}ds\|\partial_t
\Phi(\cdot,t-s)\|_{L^{\infty}(Y)}\|\nabla G(\cdot,s)\|_{L^{\infty}
(O_\varepsilon)}
+\|\Phi(\cdot,0)\|_{L^{\infty}(Y)}\|\nabla G(\cdot,t)\|_{L^{\infty}
(O_\varepsilon)}}_{=:Q(t)}.
\end{aligned}
\end{equation}
Then, it is reduced to estimating the left-hand side of $\eqref{pri:5.0}$ below
\begin{equation*}
\begin{aligned}
&
\int_{0}^{T}dt\int_{\Omega_{\varepsilon}}\Big|
\sum_{i}(\dashint_{O_{\varepsilon}^i}\frac{\partial J_1}{\partial t}) 1_{O_{\varepsilon}^i}\Big|^2
\lesssim  \int_{0}^{T}dt |Q(t)|^2\int_{\Omega_{\varepsilon}}\Big|
\sum_{i} 1_{O_{\varepsilon}^i}
\Big|^2
\lesssim \varepsilon\int_{0}^{T}dt |Q(t)|^2\\
&\lesssim
\varepsilon\bigg\{
\|\partial_t\Phi\|_{L^1(0,T;L^\infty(Y))}^2
\|\nabla G\|_{L^2(0,T;L^\infty(O_\varepsilon))}^2
+
\|\Phi(\cdot,0)\|_{L^{\infty}(Y)}^2
\|\nabla G\|_{L^2(0,T;L^{\infty}(O_\varepsilon))}^2
\bigg\},
\end{aligned}
\end{equation*}
where we employ Young's inequality for the second line.
This gives the desired estimate $\eqref{pri:5.0}$.

\textbf{Step 2.} Arguments for $\eqref{*pde4.8}$.
On account of the antisymmetric property of flux corrector $\Phi$ in Proposition \ref{prop2.1},
integrating by parts we obtain that, for any $t\in(0,T]$,
\begin{equation}\label{*pde4.9}
\begin{aligned}
  \int_{O_{\varepsilon}^i}\frac{\partial J_1}{\partial t}(\cdot,t)
  &=\int_{O_{\varepsilon}^i}
  \frac{\partial}{\partial t}
  [(W^{\varepsilon}-A)\ast G](\cdot,t)\cdot\nabla \psi_{\varepsilon}
  =\int_{O_{\varepsilon}^i}
  \frac{\partial}{\partial t}
  \big[(\nabla \cdot\Phi)^\varepsilon\ast G\big](\cdot,t)\cdot\nabla \psi_{\varepsilon}\\
  &\overset{\eqref{anti-sym}}{=}
  -\varepsilon\int_{O_{\varepsilon}^i}
  \frac{\partial}{\partial t}\big[
  \Phi^\varepsilon\ast_2\partial G\big](\cdot,t)\cdot\nabla\psi_{\varepsilon}
  +\varepsilon
  \int_{\partial O_{\varepsilon}^i}dS\frac{\partial}{\partial t}
  \big[\Phi^\varepsilon\ast_3 (G\otimes\nabla\psi_{\varepsilon}\otimes
  \vec{n}_s)\big](\cdot,t)\\
  :&=E_{1}^i(t)+E_{2}^i(t),
  \end{aligned}
\end{equation}
in which $\vec{n}_s$ is the unit outward normal vector of the boundary $\partial O_{\varepsilon}^i$.

In fact, we can establish the following estimates:
\begin{subequations}
\begin{align}
&\big|E_1^i(t)\big|\lesssim |O_{\varepsilon}^i|
  Q(t); \label{*pde3.19} \\
& E_2^i(t) = 0 \quad (\text{uniformly~with~respect~to~the index}~i~\text{and}~ t\in(0,T]).
\label{pri:4.4}
\end{align}
\end{subequations}
Plugging the above two estimates back into
$\eqref{*pde4.9}$, we obtain the stated estimate
$\eqref{*pde4.8}$.

\textbf{Step 3.} Arguments for $E_1^i(t)$ in $\eqref{*pde3.19}$.
Recalling the properties of the cut-off
function $\psi_{\varepsilon}$ stated in Lemma $\ref{lemma:cut-off}$, a direct computation
leads to
\begin{equation*}
\begin{aligned}
 |E_{1}(t)|
 &\lesssim\int_{O_{\varepsilon}^i}\bigg|\frac
 {\partial}{\partial t}[\Phi^\varepsilon\ast_2\partial G](\cdot,t)\bigg|
 =\int_{O_{\varepsilon}^i}
 \bigg|(\partial_t\Phi^\varepsilon
 \ast_2\partial G)(\cdot,t)+\Phi^\varepsilon(\cdot,0):\partial G(\cdot,t)\bigg|\\
 &\lesssim \int_{O_{\varepsilon}^i}
 \big|\partial_t\Phi^\varepsilon\ast_2\partial G\big|(\cdot,t)
+ |O_{\varepsilon}^i|
  \|\Phi(\cdot,0)\|_{L^{\infty}(Y)}\|\nabla G(\cdot,t)\|_{L^{\infty}(O_\varepsilon)}.
 \end{aligned}
\end{equation*}
This together with the estimate
\begin{equation*}
\begin{aligned}
\int_{O_{\varepsilon}^i}
 \big|\partial_t\Phi^\varepsilon\ast_2\partial G\big|(\cdot,t)
\leq
\int_{0}^tds\int_{O_{\varepsilon}^i}
\big|\partial_s\Phi^\varepsilon(\cdot,s):\partial G(\cdot,t-s)\big|
\lesssim 
|O_{\varepsilon}^i|
\int_{0}^{t}ds\|\partial_t
\Phi(\cdot,t-s)\|_{L^{\infty}(Y)}\|\nabla G(\cdot,s)\|_{L^{\infty}
(O_\varepsilon)}
\end{aligned}
\end{equation*}
leads to the stated estimate $\eqref{*pde3.19}$,
where we recall the definition of the notation
presented in  $\eqref{notation:4.2}$.

\textbf{Step 4.} Arguments for $E_2^i(t)$ in $\eqref{pri:4.4}$.
According to Remark $\ref{remark:4.1}$,
we rewrite $E_{2}^i(t)$ as follows:
\begin{equation}\label{}
\begin{aligned}
E_{2}^i(t)
  =\frac{\partial}{\partial t}
  \int_{\partial O_{\varepsilon}^i\cap \partial O_\varepsilon}dS
  &\big[\Phi^\varepsilon\ast_3 \big(G
  \otimes(\varepsilon\nabla\psi_{\varepsilon})\otimes
  \vec{n}_R\big)\big](\cdot,t) \\
&+ \frac{\partial}{\partial t}
  \int_{\partial O_{\varepsilon}^i\cap O_\varepsilon}dS
  \big[\Phi^\varepsilon\ast_3
  \big(G\otimes(\varepsilon\nabla\psi_{\varepsilon})\otimes
  \vec{n}_S\big)\big](\cdot,t).
\end{aligned}
\end{equation}

On the one hand, it follows from
the property (a) presented in Remark $\ref{remark:4.1}$ that
\begin{equation*}
(\varepsilon\nabla\psi_\varepsilon)\otimes \vec{n}_R =
\vec{n}_R\otimes(\varepsilon\nabla\psi_\varepsilon)
  \qquad \text{on}\quad \partial O_\varepsilon^i\cap \partial O_\varepsilon,
\end{equation*}
and this together with the antisymmetric property of $\Phi$ leads to
\begin{equation}\label{k-9}
\frac{\partial}{\partial t}
  \int_{\partial O_{\varepsilon}^i\cap \partial O_\varepsilon}dS
  \big[\Phi^\varepsilon\ast_3 \big(G
  \otimes(\varepsilon\nabla\psi_{\varepsilon})\otimes
  \vec{n}_R\big)\big](\cdot,t) = 0.
\end{equation}

On the other hand, using the antisymmetric property again,
we first notice that
\begin{equation*}
\big[\Phi^\varepsilon\ast_3 (G\otimes
  \vec{n}_S\otimes
  \vec{n}_S)\big](\cdot,t)
  =\int_{0}^{t}ds\Phi_{ij,k}(\cdot/\varepsilon,t-s)\big(G_k
  \vec{n}_S^i\vec{n}_S^j\big)(\cdot,s)=0
  \qquad\text{on}\quad \partial O_\varepsilon^i\cap O_\varepsilon,
\end{equation*}
where we employ the fact that $\vec{n}_S\otimes
  \vec{n}_S$ is a symmetric matrix.
Therefore, for any $N>0$, it follows that
\begin{equation*}
\begin{aligned}
&E_{2}^i(t)
  \overset{\eqref{k-9}}{=} \frac{\partial}{\partial t}
  \int_{\partial O_{\varepsilon}^i\cap O_\varepsilon}dS
  \big[\Phi^\varepsilon\ast_3
  \big(G\otimes(\varepsilon\nabla\psi_{\varepsilon})\otimes
  \vec{n}_S\big)\big](\cdot,t) \\
  &= \frac{\partial}{\partial t}
  \int_{\partial O_{\varepsilon}^i\cap O_\varepsilon}dS
  \big[\Phi^\varepsilon\ast_3 (G\otimes
  \big(\varepsilon\nabla\psi_{\varepsilon}+N\vec{n}_S\big)\otimes
  \vec{n}_S)\big](\cdot,t)
  -N\varepsilon \frac{\partial}{\partial t}
  \int_{\partial O_{\varepsilon}^i\cap O_\varepsilon}dS
  \big[\Phi^\varepsilon\ast_3 (G\otimes
  \vec{n}_S\otimes
  \vec{n}_S)\big](\cdot,t) \\
  &= \frac{\partial}{\partial t}\int_{\partial O_{\varepsilon}^i
  \cap O_\varepsilon}dS
  \big[\Phi^\varepsilon\ast_3 (G\otimes
  \big(\varepsilon\nabla\psi_{\varepsilon}+N\vec{n}_S\big)\otimes
  \vec{n}_S)\big](\cdot,t).
\end{aligned}
\end{equation*}

Then, taking $N\to\infty$,
if the right-hand side above vanishes, one can obtain
the desired equality $\eqref{pri:4.4}$. This is true
and we show it in the following.

On the one hand, we can start from the following observation
\begin{equation}\label{pri:4.3}
\lim_{N\to\infty}
  \big[\Phi^\varepsilon\ast_3 (G\otimes
  \big(\varepsilon\nabla\psi_{\varepsilon}+N\vec{n}_S\big)\otimes
  \vec{n}_S)\big](\cdot,t) = 0
  \qquad\text{on}~~\partial O_\varepsilon^i\cap O_\varepsilon,
\end{equation}
which comes from the fact that $\varepsilon\nabla\psi_{\varepsilon}+N\vec{n}_S$  becomes linearly dependent on $\vec{n}_S$, i.e.,
\begin{equation*}
  \cos\theta = \frac{\big(\varepsilon\nabla\psi_\varepsilon+N\vec{n}_S\big)\cdot
  \vec{n}_S}{\sqrt{N^2+c}} = \frac{N}{\sqrt{N^2+c}} \to 1
\qquad\text{as}
\qquad N\to\infty,
\end{equation*}
where we employ the geometry fact $\eqref{k-6}$ by noting that
$\varepsilon\nabla\psi_\varepsilon\sim \vec{n}$.
In other words, we have
\begin{equation*}
\big(\varepsilon\nabla\psi_{\varepsilon}+N\vec{n}_S\big)\otimes
  \vec{n}_S -
  \vec{n}_S\otimes
  \big(\varepsilon\nabla\psi_{\varepsilon}+N\vec{n}_S\big)\to 0,
\qquad\text{as}\qquad N\to\infty.
\end{equation*}
and this together with the antisymmetric property of $\Phi$ gives
$\eqref{pri:4.3}$.

One the other hand, for any fixed $i$ and $\varepsilon$, $E_{2}^i(t)$ is unform bounded with
respect to any $t\in[0,T]$ and $N>0$. Therefore, we have
\begin{equation*}
\begin{aligned}
\lim_{N\to\infty}\frac{\partial}{\partial t}\int_{\partial O_{\varepsilon}^i\cap O_\varepsilon}
&dS
  \big[\Phi^\varepsilon
\ast_3 (G
\otimes
  \big(\varepsilon\nabla\psi_{\varepsilon}
+N\vec{n}_S\big)\otimes
  \vec{n}_S)\big](\cdot,t)\\
&=
\frac{\partial}{\partial t}\Big(
\lim_{N\to\infty}
\int_{\partial O_{\varepsilon}^i\cap O_\varepsilon}dS
  \big[\Phi^\varepsilon\ast_3 (G\otimes
  \big(\varepsilon\nabla\psi_{\varepsilon}+N\vec{n}_S\big)\otimes
  \vec{n}_S)\big](\cdot,t)\Big)\\
&\qquad\qquad\quad=\frac{\partial}{\partial t}
\int_{\partial O_{\varepsilon}^i\cap O_\varepsilon}dS\Big(
\lim_{N\to\infty}
  \big[\Phi^\varepsilon\ast_3 (G\otimes
  \big(\varepsilon\nabla\psi_{\varepsilon}+N\vec{n}_S\big)\otimes
  \vec{n}_S)\big](\cdot,t)\Big) \overset{\eqref{pri:4.3}}{=}0,
\end{aligned}
\end{equation*}
where we employed Lebesgue's dominated convergence theorem
for the second equality. This gives $\eqref{pri:4.4}$.

Finally, we note that the fact that $\eqref{pri:4.4}$ holds uniformly with respect to $i$ and $t\in(0,T]$ is based on consistent geometric features
(i.e., $\nabla\psi_\varepsilon\cdot n_S =0$ on $\partial O_\varepsilon^i
\cap O_\varepsilon$ for each $i$ due to $\eqref{k-6}$) and the same algebraic principles
(i.e., ``the second-order inner product of an antisymmetric matrix and a symmetric matrix is equal to zero.''). This completes all the proofs.
\end{proof}

\begin{remark}\label{remark:4.2}
\emph{In the previous work \cite{Wang-Xu-Zhang22}, to estimate $E_2^i(t)$, we mainly relied on the periodicity of the flux corrector. Therefore, we made a more detailed division for $O_\varepsilon$. However, this method becomes very complicated when addressing the decomposition of the three-dimensional boundary layer. Another obvious drawback is that it can only indicate the existence of the required decomposition for specific regions in $\mathbb{R}^3$. As given in the above proof, we now turn to make more in-depth use of the antisymmetry property of the flux corrector.  This improvement greatly simplifies the decomposition process for boundary layers (see Remark $\ref{remark:4.1}$).}
\end{remark}

\begin{lemma}[existences]\label{lemma3.6}
There exists at least one weak solution for \eqref{divc2} and \eqref{divc3}, respectively.
\end{lemma}

\begin{proof}
The existence of weak solution for \eqref{divc2} and \eqref{divc3} follows from compatibility condition for divergence operator and
it is obvious for \eqref{divc3}.
Recalling the definition of $J_1$ and $J_2$ in \eqref{pde3.8}, and the fact
that
\begin{equation*}
\nabla\cdot\bigg\{
\psi_{\varepsilon}[W^{\varepsilon}\ast G+\varepsilon\phi^{\varepsilon}_{k,j}\ast \nabla_k G_j]
\bigg\}=J_1+J_2,
\end{equation*}
it follows that, for a.e. $t\geq0$,
\begin{equation*}
\int_{\Omega_{\varepsilon}}(J_1+J_2)(\cdot,t)
=\int_{\Omega_{\varepsilon}}
\nabla\cdot\bigg\{
\psi_{\varepsilon}[W^{\varepsilon}\ast G+\varepsilon\phi^{\varepsilon}\ast \nabla G]
\bigg\}(\cdot,t)=\int_{\partial \Omega_{\varepsilon}}dS\vec{n}\cdot\bigg\{
\psi_{\varepsilon}[W^{\varepsilon}\ast G+\varepsilon\phi^{\varepsilon}\ast \nabla G]
\bigg\}(\cdot,t)=0,
\end{equation*}
where, by abusing the notation, $\vec{n}$ represents the unit outward normal vector of
$\partial\Omega_\varepsilon$.
Therefore, due to the fact that supp$(J_1)=O_{\varepsilon}=
\sum\limits_{i}O_{\varepsilon}^i$ (see Remark $\ref{remark:4.1}$),
there holds
\begin{equation*}
\begin{aligned}
\int_{\Omega_{\varepsilon}}
\Big[ J_2+\sum_{i} (\dashint_{O_{\varepsilon}^i} J_1)1_{O_{\varepsilon}^i} \Big](\cdot,t)=\int_{\Omega_{\varepsilon}}
\big(J_1+J_2\big)(\cdot,t)=0.
\end{aligned}
\end{equation*}

Consequently, we have
\begin{equation*}
\int_{\Omega_{\varepsilon}}
\Big[\frac{\partial J_2}{\partial t} +\sum_{i} (\dashint_{O_{\varepsilon}^i}\frac{\partial J_1}{\partial t})1_{O_{\varepsilon}^i}\Big](\cdot,t)=
\int_{\Omega_{\varepsilon}}
 \partial_t \big(J_1+J_2\big)(\cdot,t)=0.
\end{equation*}
This is the compatibility condition for \eqref{divc2}.
\end{proof}

\medskip

Consequently, the structure of the proof of Proposition $\ref{lemma3.6-1}$
can be presented by the following flow chart.
\begin{center}
\begin{tikzpicture}[<-,>=stealth', global scale =0.8]
\node[state,
      text width = 3cm] (l1)
{ Proposition $\ref{lemma3.6-1}$
};

\node[state,
      left of= l1,
      text width = 3cm,
      yshift = 1cm,
      node distance=5cm] (l2)
{ Lemma $\ref{lemma3.6}$
};

\node[state,
      left of= l1,
      text width = 3cm,
      yshift = -1cm,
      node distance=5cm] (l3)
{ Theorem $\ref{thm4}$
};

\node[state,
      right of = l1,
      yshift=1cm,
      text width = 3cm,
      node distance=5cm] (m1)
{Lemma $\ref{lemma:5.1}$};

\node[state,
      right of = l1,
      yshift=-1cm,
      text width = 3cm,
      node distance=5cm] (m2)
{Lemma $\ref{*lemma4.1}$};

\node[state,
      right of = m1,
      text width = 3cm,
      node distance = 5cm] (r1)
{Lemma $\ref{lemma:cut-off}$
};

\node[state,
      right of = m2,
      text width = 3cm,
      node distance = 5cm] (r2)
{Subection $\ref{subsec:4.2}$
};

\path (l1) edge (l2)
      (l1) edge (l3)
      (l1) edge (m1)
      (l1) edge (m2)
      (m1) edge (r1)
      (m2) edge (r2);
\end{tikzpicture}
\end{center}

\noindent
\textbf{The proof of Proposition $\ref{lemma3.6-1}$.}
The existence of the solution of $\eqref{divc2}$ has been established
in Lemma $\ref{lemma3.6}$, while the main job is to show the
estimate $\eqref{pri:4.1}$. To do so, applying the result presented in Theorem
$\ref{thm4}$ to the solution of the equations $\eqref{divc2}$, we have
\begin{equation*}
\begin{aligned}
\|\xi_{\varepsilon}\|_{L^2(\Omega_{\varepsilon,T})}\lesssim
\varepsilon \|\nabla\xi_{\varepsilon}\|_{L^2(\Omega_{\varepsilon,T})}
&\overset{\eqref{pde3.21}}{\lesssim} \big\|\frac{\partial}{\partial t}(J_2+\sum_{i}(\dashint_{O_{\varepsilon}^i}
J_1)1_{O_{\varepsilon}^i})
\big\|_{L^2(\Omega_{\varepsilon,T})}\\
&\leq
\big\|\partial_tJ_2\|_{L^2(\Omega_{\varepsilon,T})}
+
\|\sum_{i}(\dashint_{O_{\varepsilon}^i}\partial_tJ_1)1_{O_{\varepsilon}^i})\big\|_{L^2(\Omega_{\varepsilon,T})}.
\end{aligned}
\end{equation*}
In view of Lemmas $\ref{lemma:5.1}$ and $\ref{*lemma4.1}$,
the right-hand side above can be controlled by
\begin{equation*}
\begin{aligned}
\varepsilon^{\frac{1}{2}}
&\bigg\{\|\partial_t \phi\|_{L^1(0,T;L^2(Y))}
+\|\phi(\cdot,0)\|_{L^2(Y_f)}
\bigg\}
\|(\nabla G,\varepsilon^{\frac{1}{2}}\nabla^2G)\|_{L^2(0,T;L^{\infty}
(O_\varepsilon))}\\
& +\varepsilon^{\frac{1}{2}}
\bigg\{\|\partial_t \Phi\|_{L^1(0,T;L^{\infty}(Y))}+
\|\Phi(\cdot,0)\|_{L^{\infty}(Y)}
\bigg\}
\|\nabla G\|_{L^2(0,T;L^{\infty}
(O_\varepsilon))}\\
+&\Big\{\|\partial_tA\|_{L^1(0,T)}+|A(0)|\Big\}
\bigg\{\varepsilon^{\frac{1}{2}}\|\nabla F\|_{L^2(0,T;C(\bar{\Omega}))}
+\|\nabla(G-F)\|_{
L^2(0,T;L^2{(\text{supp}(\psi_{\varepsilon}))})}\bigg\}.
\end{aligned}
\end{equation*}

Applying Propositions $\ref{prop2.1}$, $\ref{prop3.1}$, and $\ref{P:1}$
to the above expression,  we can further derive that
\begin{equation}\label{f:5.15}
\begin{aligned}
\varepsilon \|\nabla\xi_{\varepsilon}\|_{L^2(\Omega_{\varepsilon,T})}
&\lesssim \varepsilon^{\frac{1}{2}}
\bigg\{
\|\nabla G\|_{L^2(0,T;L^{\infty}
(\text{supp}(\psi_{\varepsilon})))}
+\|G\|_{L^2(0,T;L^{\infty}
(\text{supp}(\psi_{\varepsilon})))}
\bigg\}\\
&\qquad\qquad+\varepsilon^{\frac{1}{2}}
\bigg\{
\|\nabla G\|_{L^2(0,T;L^{\infty}
(\text{supp}(\psi_{\varepsilon})))}
+\varepsilon^{\frac{1}{2}}\|\nabla^2G\|_{L^2(0,T;L^{\infty}
(\text{supp}(\psi_{\varepsilon})))}\bigg\}\\
&\qquad\qquad\qquad\qquad+
\varepsilon^{\frac{1}{2}}\|\nabla F\|_{L^2(0,T;C(\bar{\Omega}))}
+\|\nabla(G-F)\|_{L^2(0,T;L^2{(\text{supp}(\psi_{\varepsilon}))})}.
\end{aligned}
\end{equation}

By using  Lemma $\ref{lemma3.2}$, the right-hand side of
$\eqref{f:5.15}$
can be governed by $\varepsilon^{\frac{1}{2}}\|f\|_{L^2(0,T;C^{1,1/2}(\bar{\Omega}))}$,
and the desired estimate $\eqref{pri:4.1}$ follows. This ends the whole proof.
\qed

\subsection{\centering Proof of Proposition $\ref{lemma3.6-2}$}\label{subsec:4.4}

\noindent
%\textbf{The proof of Proposition $\ref{lemma3.6-2}$.}
We construct $\eta_{\varepsilon}$ according to the decomposition
introduced in Subsection $\ref{subsec:4.2}$. By virtue of
Remark $\ref{remark:4.1}$, we have the decomposition of $O_\varepsilon$. For each $O_\varepsilon^i$ and any $t\geq 0$ fixed,
we can get a $\eta_{\varepsilon}^{i}$ which satisfies the following equation \eqref{*pde3.9}. The desired solution $\eta_{\varepsilon}$ follows from sticking these $\eta_{\varepsilon}^{i}$ together piece by piece.
\begin{equation}\label{*pde3.9}
\left\{
\begin{aligned}
  \nabla\cdot\eta_{\varepsilon}^{i}(\cdot,t) &=\frac{\partial J_1}{\partial t}
  -(\dashint_{O_{\varepsilon}^i}\frac{\partial J_1}{\partial t})1_{O_{\varepsilon}^i},&\text{in}&\quad O_{\varepsilon}^i;\\
  \eta_{\varepsilon}^{i}(\cdot,t)&=0,&\text{on}&\quad\partial O_{\varepsilon}^i,
\end{aligned}\right.
\end{equation}

Moreover, we have the following estimate
\begin{equation}\label{*pde4.14}
\|\nabla\eta_{\varepsilon}^{i}(\cdot,t)\|_{L^2(O_{\varepsilon}^i)}\leq C \|\frac{\partial J_1}{\partial t}-(
\dashint_{O_{\varepsilon}^i}\frac{\partial J_1}{\partial t}
)1_{O_{\varepsilon}^i}\|_{L^2(O_{\varepsilon}^i)}
\leq C
\|\partial_t J_1\|_{L^2(O_{\varepsilon}^i)},
\end{equation}
where the constant $C$ \textbf{does not} depend on $\varepsilon$ and $i$ (see e.g. \cite[Chapter III.3]{Galdi11}).

Let $\eta_{\varepsilon}:=\sum_{i}\eta_{\varepsilon}^{i}$,
and it is not hard to observe that
%for a.e. $t\geq 0$, (The following expression is for emphasizing temporal variable.)
\begin{equation*}
 \|\nabla \eta_{\varepsilon}(\cdot,t)\|_{L^2(O_\varepsilon)}^2
 = \sum_{i}\|\nabla\eta_{\varepsilon}^{i}(\cdot,t)\|_{L^2(O_\varepsilon^i)}^2.
\end{equation*}
This together with $\eqref{*pde4.14}$ leads to
\begin{equation}\label{f:5.13}
 \|\nabla \eta_{\varepsilon}(\cdot,t)\|_{L^2(O_\varepsilon)}^2
 \lesssim \|\partial_t J_1\|_{L^2(O_\varepsilon)}^2.
\end{equation}

Now, dealing with the term
in the right-hand side of $\eqref{f:5.13}$,
it follows from the definition of $J_1$ in
\eqref{pde3.8} and Minkowski's inequality  that
\begin{equation}\label{f:5.12}
\begin{aligned}
&\int_{0}^{T}dt\int_{O_{\varepsilon}}
\Big|\frac{\partial J_1}{\partial t}\Big|^2
\lesssim\varepsilon^{-2}
\int_{0}^{T}dt\int_{O_{\varepsilon}}
\Big|\frac{\partial }{\partial t}
\big[(W^{\varepsilon}-A)\ast G\big]\Big|^2\\
&\lesssim\varepsilon^{-2}
\|\partial_tW^{\varepsilon}-\partial_tA\|^2_{L^{1}
(0,T;L^2(O_{\varepsilon}))}
\|G\|^2_{L^{2}(0,T;L^{\infty}(O_{\varepsilon}))}
+\varepsilon^{-2}
\|W^{\varepsilon}(\cdot,0)-A(0)\|^2_{L^{2}
(O_{\varepsilon})}
\|G\|^2_{L^{2}(0,T;L^{\infty}(O_{\varepsilon}))}.
\end{aligned}
\end{equation}
By a rescaling argument used for $\partial_tW^{\varepsilon}$
and using its periodicity,
we note that
\begin{equation*}
\begin{aligned}
\|\partial_tW^{\varepsilon}-\partial_tA\|_{L^{1}
(0,T;L^2(O_{\varepsilon}))}
=&\int_0^{T}dt\bigg(\int_{O_{\varepsilon}}dx
|\partial_tW(x/\varepsilon,t)-\partial_tA(t)|^2\bigg)^{1/2}\\
&\lesssim \varepsilon^{1/2}\int_{0}^Tdt\big|\partial_t A(t)\big|
+\int_{0}^Tdt\bigg(\int_{O_{\varepsilon}}dx
\big|\partial_tW(x/\varepsilon,t)\big|^2\bigg)^{1/2}\\
%&\lesssim
%\varepsilon^{1/2}\int_{0}^Tdt|A'(t)|
%+\varepsilon^{1/2}\int_{0}^Tdt\bigg(\int_{Y_f}dy
%|\partial_tW(y,t)|^2\bigg)^{1/2}\\
&=\varepsilon^{1/2}
\Big\{\|\partial_tA\|_{L^1(0,T)}+
\|\partial_tW\|_{L^1(0,T;L^2(Y_f))}\Big\};
\end{aligned}
\end{equation*}
By the same token, we have
\begin{equation*}
\|W^{\varepsilon}(\cdot,0)-A(0)\|_{L^{2}
(O_{\varepsilon})}
\lesssim\varepsilon^{1/2}.
\end{equation*}

Inserting the above two estimates back into $\eqref{f:5.12}$,
and then together with $\eqref{f:5.13}$, we obtain that
\begin{equation*}
\begin{aligned}
\varepsilon^2\int_{0}^{T}dt\int_{O_\varepsilon}
\big|\nabla \eta_{\varepsilon}(\cdot,t)\big|^2
\lesssim \varepsilon
\Big\{\|\partial_tA\|_{L^1(0,T)}
+
\|\partial_tW\|_{L^1(0,T;L^2(Y_f))}+ 1\Big\}^2
\|G\|^2_{L^{2}(0,T;L^{\infty}(O_{\varepsilon}))}.
\end{aligned}
\end{equation*}

Consequently, appealing to Proposition $\ref{prop2.1}$ and Lemma $\ref{lemma2.5}$,
we have derived the main part of the stated estimate $\eqref{pri:5.2}$, and
the remainder of the proof follows from Poincar\'e's inequality. This
completes the whole proof.
\qed

%\section{Asymptotic expansions}\label{sec3}

%Before giving a formal proof, let us introduce the concrete expression of
%the notations appeared in Lemma $\ref{lemma:4.3}$.
%The notation on convolutions was defined in Subsection $\ref{notation}$,
%which will be heavily used in later proofs.

\section{Proof of Theorem $\ref{thm1}$}\label{sec5}

\begin{lemma}\label{lemma:4.3}
Let $p_0$ and $f$ be associated by $\eqref{pde1.2}$.
Suppose that $(u_\varepsilon,p_\varepsilon)$ satisfies $\eqref{pde1.1}$.
Let $(W,\pi)$ be the corrector given in
$\eqref{pde2.1}$, while the correctors $(\phi,\hat{\xi_{\varepsilon}},\hat{\eta_{\varepsilon}})$
are related to $\eqref{divc1*}$ and $\eqref{pde3.11}$.
Define the error term $(w_{\varepsilon},q_{\varepsilon})$ as in $\eqref{eq:A}$.
%$(w_{\varepsilon},q_{\varepsilon})$ as follows,
%\begin{equation}\label{pde3.17}
%\left\{
%  \begin{aligned}
%&w_{\varepsilon}=u_{\varepsilon}-\psi_{\varepsilon}
%\big(W^{\varepsilon}\ast G+\varepsilon
%\phi^\varepsilon\ast_{2} \partial G\big)
%+\hat{\xi}
%+\hat{\eta};\\
%&q_{\varepsilon}=p_\varepsilon-p_0-
%\varepsilon\psi_{\varepsilon}\pi^{\varepsilon}\ast_{1} G,
%\end{aligned}\right.
%\end{equation}
%where $\psi_{\varepsilon}$ is a radial cut-off function defined in Lemma $\ref{lemma:cut-off}$
%and the definition of the convolutions $\ast,\ast_1,\ast_2$ are shown in Subsection $\ref{notation}$.
Then, the pair $(w_{\varepsilon},q_{\varepsilon})$ satisfies the following equations:
\begin{equation}\label{pde3.13}
\left\{
\begin{aligned}
\partial_t w_{\varepsilon}
-\varepsilon^2\Delta w_{\varepsilon}+\nabla q_{\varepsilon}&=I_1+\varepsilon I_2+\varepsilon^2I_3+\varepsilon^3 I_4,&\text{in}&\quad\Omega_{\varepsilon}\times(0,T];\\
\nabla\cdot w_{\varepsilon}&=0,&\text{in}&\quad
\Omega_{\varepsilon}\times(0,T];\\
w_{\varepsilon}&=0,&\text{on}&\quad
\partial\Omega_{\varepsilon}\times(0,T];\\
w_{\varepsilon}|_{t=0}&=0,&\text{on}
&\quad\Omega_{\varepsilon},
\end{aligned}
\right.
\end{equation}
in which we adopt the convention presented in $\eqref{notation:4.1}$ to have
the expressions of $I_1,I_2,I_3$, and $I_4$ by
\begin{equation}\label{notation:5.2}
\begin{aligned}
I_1:=&f-\nabla p_0-\psi_{\varepsilon}W^{\varepsilon}(\cdot,0)G
+\xi_{\varepsilon}+\eta_{\varepsilon};\\
I_2:=&-\psi_{\varepsilon}
\bigg[
\partial_t\phi^{\varepsilon}\ast_2\partial G+
\phi^{\varepsilon}(\cdot,0):\partial G
\bigg]+2\psi_{\varepsilon}(\partial W)^{\varepsilon}\ast_2\partial G
-
\nabla\psi_{\varepsilon}
\pi^{\varepsilon}\ast_1 G
-\psi_{\varepsilon}\nabla G\ast\pi^{\varepsilon};\\
I_3:=&\psi_{\varepsilon}
\nabla\cdot
\Big\{
(\nabla\phi)^{\varepsilon}\ast_2\partial G
\Big\}
+\nabla\cdot
\bigg\{
\nabla\psi_{\varepsilon}\otimes \big(W^{\varepsilon}\ast G\big)
\bigg\}
+\nabla\big(W^{\varepsilon}\ast G\big)\nabla
\psi_{\varepsilon}
+\psi_{\varepsilon} W^{\varepsilon}\ast\Delta G
-\Delta\hat{\xi_{\varepsilon}}
-\Delta\hat{\eta_{\varepsilon}};\\
I_4:=&\nabla
\big(\phi^{\varepsilon}\ast_2\partial G\big)\nabla\psi_{\varepsilon}+
\nabla\cdot\bigg\{
\nabla\psi_{\varepsilon}\otimes
\big(\phi^{\varepsilon}\ast_2\partial G\big)
\bigg\}
+\psi_{\varepsilon}
\nabla\cdot\bigg\{\phi^{\varepsilon}\ast_2\nabla\partial G\bigg\},
\end{aligned}
\end{equation}
where $G$ is associated with the quantity
$f-\nabla p_0$, defined in Subsection $\ref{notation}$.
\end{lemma}

%\medskip
%\noindent
%\textbf{The proof of Lemma $\ref{lemma:4.3}$.}
\begin{proof}
To obtain the equations $\eqref{pde3.13}$, we merely insert
the expression $\eqref{eq:A}$ into the left-hand sides of
$\eqref{pde3.13}$, and compute it term by term, directly.
By noticing the boundary conditions of $u_\varepsilon$,
$\psi_\varepsilon W^\varepsilon$, $\psi_\varepsilon\phi^\varepsilon$, $\hat{\xi_{\varepsilon}}$ and
$\hat{\eta_{\varepsilon}}$, it is not hard to verify that $w_\varepsilon(\cdot,t)$ vanishes on the
boundary of $\Omega_\varepsilon$ for $t\geq 0$, which satisfies
the third line of $\eqref{pde3.13}$. From the initial value
of $u_\varepsilon$, the definition of the convolution with respect to the temporal variable, as well as, $\hat{\xi_{\varepsilon}}$ and
$\hat{\eta_{\varepsilon}}$, the last line of $\eqref{pde3.13}$ can be simply checked. Therefore,
the main job is devoted to deriving the first and second line of $\eqref{pde3.13}$.

\medskip

\textbf{Part 1.}
We firstly address the first line of the equations $\eqref{pde3.13}$,
and start from dealing with the term $\partial_t w_{\varepsilon}$,
by noting the notation presented in the first and third line of $\eqref{notation:4.1}$,
\begin{equation}\label{f:4.2}
\begin{aligned}
\frac{\partial w_{\varepsilon}}{\partial t}
&=
\frac{\partial u_{\varepsilon}}{\partial t}-
\psi_{\varepsilon}\frac{\partial}{\partial t}
\big(W^{\varepsilon}\ast G+\varepsilon \phi^{\varepsilon}\ast_2
\partial G\big)(\cdot,t)
+\xi_{\varepsilon}(\cdot,t)
+\eta_{\varepsilon}(\cdot,t)\\
&=\frac{\partial u_{\varepsilon}}{\partial t}
+\xi_{\varepsilon}(\cdot,t)
+\eta_{\varepsilon}(\cdot,t)
-\psi_{\varepsilon}
\bigg\{\partial_tW^{\varepsilon}\ast G +W^{\varepsilon}(\cdot,0)G+\varepsilon \big(\partial_t\phi^{\varepsilon}\ast_2\partial G\big)
+\varepsilon \phi^{\varepsilon}(\cdot,0):\partial G
\bigg\}.
\end{aligned}
\end{equation}

\medskip

Then, we turn to the term $\varepsilon^2\Delta w_{\varepsilon}$,
\begin{equation}\label{f:4.1}
\begin{aligned}
&-\varepsilon^2\Delta w_{\varepsilon}=
-\varepsilon^2\Delta u_{\varepsilon}+\varepsilon^2\Delta\bigg\{
\psi_{\varepsilon}\big(W^{\varepsilon}\ast G+\varepsilon\phi^{\varepsilon}\ast_2\partial G\big)
\bigg\}
-\varepsilon^2\Delta\hat{\xi_{\varepsilon}}
-\varepsilon^2\Delta\hat{\eta_{\varepsilon}}\\
=&-\varepsilon^2\Delta u_{\varepsilon}
-\varepsilon^2\Delta\hat{\xi_{\varepsilon}}
-\varepsilon^2\Delta\hat{\eta_{\varepsilon}}
+\varepsilon^2\nabla\cdot
\bigg\{
\nabla\psi_{\varepsilon}\otimes\big(W^{\varepsilon}\ast G+\varepsilon\phi^{\varepsilon}\ast_2\partial G\big)+\psi_{\varepsilon}\nabla \big(W^{\varepsilon}\ast G+\varepsilon\phi^{\varepsilon}\ast_2\partial G\big)
\bigg\},
\end{aligned}
\end{equation}
while the last term above leads to three terms:
\begin{equation*}
\underbrace{\varepsilon^2\nabla\cdot
\bigg\{
\nabla\psi_{\varepsilon}\otimes\big(W^{\varepsilon}\ast G+\varepsilon\phi^{\varepsilon}\ast_2\partial G\big)
\bigg\}}_{R_1}
+\underbrace{\varepsilon^2\nabla
\big(W^{\varepsilon}\ast G+\varepsilon\phi^{\varepsilon}\ast_2\partial G\big)\nabla\psi_{\varepsilon}}_{R_2}
+\underbrace{\varepsilon^2\psi_{\varepsilon}
\Delta\big(W^{\varepsilon}\ast G+\varepsilon\phi^{\varepsilon}\ast_2\partial G\big)}_{R_3}.
\end{equation*}
The first term $R_1$ and the second term $R_2$ are quite easy, and we merely rearrange
them in terms of the power of $\varepsilon$, and it follows that
\begin{equation*}
\begin{aligned}
 R_1 &= \varepsilon^2\nabla\cdot\Big\{
 \nabla\psi_\varepsilon\otimes\big(W^\varepsilon\ast G\big)\Big\}
 + \varepsilon^3\nabla\cdot\Big\{
 \nabla\psi_\varepsilon\otimes\big(\phi^\varepsilon\ast_2\partial G\big)\Big\};\\
 R_2 &= \varepsilon^2\nabla\big(W^\varepsilon\ast G\big)\nabla\psi_\varepsilon
 +\varepsilon^3\nabla\big(\phi^\varepsilon\ast_2\partial G\big)\nabla\psi_\varepsilon.
\end{aligned}
\end{equation*}
We continue to handle $R_3$. According to the power of $\varepsilon$, there holds
\begin{equation*}
\begin{aligned}
R_3 &= \varepsilon^2\psi_{\varepsilon}
\bigg\{
\varepsilon^{-2}(\Delta W)^{\varepsilon}\ast G
+2\varepsilon^{-1}(\partial W)^{\varepsilon}\ast_2\partial G+W^{\varepsilon}\ast\Delta G
\bigg\}
+\varepsilon^{2}\psi_{\varepsilon}\nabla\cdot
\bigg\{
(\nabla\phi)^{\varepsilon}\ast_2\partial G
+\varepsilon\phi^{\varepsilon}\ast_{2}\nabla\partial G
\bigg\}\\
&= \psi_{\varepsilon}(\Delta W)^{\varepsilon}\ast G
+2\varepsilon\psi_{\varepsilon}
\big[(\partial W)^{\varepsilon}\ast_2\partial G\big]
+\varepsilon^2\psi_{\varepsilon}
W^{\varepsilon}\ast\Delta G
+\varepsilon^2\psi_{\varepsilon}
\nabla\cdot\Big(
(\nabla\phi)^{\varepsilon}\ast_2\partial G
\Big)
+\varepsilon^3\psi_{\varepsilon}\nabla\cdot
\Big(\phi^{\varepsilon}\ast_2\nabla\partial G\Big),
\end{aligned}
\end{equation*}
where we substituted $\nabla$ with $\partial$ at the corresponding positions to highlight
that the components of the gradient are involved in the inner product of tensors
(see the convention presented in the part of ``notation for convolution'' in Subsection $\ref{notation}$).
Therefore, plugging the terms $R_1$, $R_2$ and $R_3$ above back into
$\eqref{f:4.1}$ we obtain that
\begin{equation}\label{f:4.3}
\begin{aligned}
-\varepsilon^2\Delta w_{\varepsilon}
=&\psi_{\varepsilon}(\Delta W)^{\varepsilon}\ast G
+2\varepsilon\psi_{\varepsilon}
\big[(\partial W)^{\varepsilon}\ast_2\partial G\big]
-\varepsilon^2\Delta u_{\varepsilon}
-\varepsilon^2\Delta\hat{\xi_{\varepsilon}}
-\varepsilon^2\Delta\hat{\eta_{\varepsilon}}
+\varepsilon^2\psi_{\varepsilon}
\nabla\cdot\Big(
(\nabla\phi)^{\varepsilon}\ast_2\partial G
\Big)\\
&+\varepsilon^2\nabla\cdot\bigg\{
\nabla\psi_{\varepsilon}\otimes \big(W^{\varepsilon}\ast G\big)
\bigg\}
+\varepsilon^2\nabla
\big(W^{\varepsilon}\ast G\big)\nabla\psi_{\varepsilon}
+\varepsilon^2\psi_{\varepsilon}
W^{\varepsilon}\ast\Delta G\\
&+\varepsilon^{3}\nabla\cdot
\bigg\{
\nabla\psi_{\varepsilon}\otimes
\big(\phi^{\varepsilon}\ast_2\partial G\big)
\bigg\}
+\varepsilon^3
\nabla\big(\phi^{\varepsilon}\ast_2\partial G\big)\nabla\psi_{\varepsilon}
+\varepsilon^3\psi_{\varepsilon}\nabla\cdot
\Big(\phi^{\varepsilon}\ast_2\nabla\partial G\Big).
\end{aligned}
\end{equation}

\medskip

Now, we turn to the pressure term
\begin{equation}\label{f:4.4}
\nabla q_{\varepsilon}=\nabla p_{\varepsilon}-\nabla p_0-\varepsilon\nabla\psi_{\varepsilon}
\pi^{\varepsilon}\ast_1 G
-\psi_{\varepsilon}(\nabla \pi)^{\varepsilon}\ast G
-\varepsilon\psi_{\varepsilon}
\nabla G\ast\pi^{\varepsilon},
\end{equation}
where we refer the reader to the convention on ``$\ast_1$''
and ``$\ast$'' presented in $\eqref{notation:4.2}$.

Combining the equalities $\eqref{f:4.2}$, $\eqref{f:4.3}$ and $\eqref{f:4.4}$,
we have
\begin{equation*}
\begin{aligned}
\frac{\partial w_{\varepsilon}}{\partial t}
&-\varepsilon^2\Delta w_{\varepsilon}
+\nabla q_{\varepsilon}=
\frac{\partial u_{\varepsilon}}{\partial t}
+\xi_{\varepsilon}+\eta_{\varepsilon}-
\psi_{\varepsilon}\bigg[
\partial_tW^{\varepsilon}\ast G +W^{\varepsilon}(\cdot,0)G+\varepsilon \big(\partial_t\phi^{\varepsilon}\ast_2\partial G\big)
+\varepsilon \phi^{\varepsilon}(\cdot,0):\partial G
\bigg]\\
&+\psi_{\varepsilon}(\Delta W)^{\varepsilon}\ast G
+2\varepsilon\psi_{\varepsilon}
\big[(\partial W)^{\varepsilon}\ast_2\partial G\big]
-\varepsilon^2\Delta u_{\varepsilon}
-\varepsilon^2\Delta\hat{\xi_{\varepsilon}}
-\varepsilon^2\Delta\hat{\eta_{\varepsilon}}
+\varepsilon^2\psi_{\varepsilon}
\nabla\cdot\Big(
(\nabla\phi)^{\varepsilon}\ast_2\partial G
\Big)\\
&+\varepsilon^2\nabla\cdot\bigg\{
\nabla\psi_{\varepsilon}\otimes \big(W^{\varepsilon}\ast G\big)
\bigg\}
+\varepsilon^2\nabla
\big(W^{\varepsilon}\ast G\big)\nabla\psi_{\varepsilon}
+\varepsilon^2\psi_{\varepsilon}
W^{\varepsilon}\ast\Delta G\\
&+\varepsilon^{3}\nabla\cdot
\bigg\{
\nabla\psi_{\varepsilon}\otimes
\big(\phi^{\varepsilon}\ast_2\partial G\big)
\bigg\}
+\varepsilon^3
\nabla\big(\phi^{\varepsilon}\ast_2\partial G\big)\nabla\psi_{\varepsilon}
+\varepsilon^3\psi_{\varepsilon}\nabla\cdot
\Big(\phi^{\varepsilon}\ast_2\nabla\partial G\Big)\\
&+\nabla p_{\varepsilon}
-\nabla p_0-\varepsilon
\nabla\psi_{\varepsilon}
\pi^{\varepsilon}\ast_1 G-\psi_{\varepsilon}
(\nabla \pi)^{\varepsilon}\ast G
-\varepsilon\psi_{\varepsilon}\nabla G\ast\pi^{\varepsilon},
\end{aligned}
\end{equation*}
whose right-hand side further equals to
\begin{equation*}
\begin{aligned}
&\frac{\partial u_{\varepsilon}}{\partial t}
-\varepsilon^2\Delta u_{\varepsilon}+\nabla p_{\varepsilon}
-\nabla p_0-\psi_{\varepsilon}W^{\varepsilon}(\cdot,0)G+\xi_{\varepsilon}+\eta_{\varepsilon}
-\psi_{\varepsilon}\partial_tW^{\varepsilon}\ast G
+\psi_{\varepsilon}(\Delta W)^{\varepsilon}\ast G-\psi_{\varepsilon}(\nabla\pi)^{\varepsilon}\ast G\\
&-\varepsilon\psi_{\varepsilon}
\bigg[
\partial_t\phi^{\varepsilon}
\ast_2\partial G+\phi^{\varepsilon}(\cdot,0):\partial G
\bigg]
+2\varepsilon\psi_{\varepsilon}\big[(\partial W)^{\varepsilon}\ast_2\partial G\big]
-\varepsilon
\nabla\psi_{\varepsilon}
\pi^{\varepsilon}\ast_1 G
-\varepsilon\psi_{\varepsilon}\nabla G\ast\pi^{\varepsilon}\\
&+\varepsilon^2\psi_{\varepsilon}
\nabla\cdot
\Big(
(\nabla\phi)^{\varepsilon}\ast_2\partial G
\Big)
+\varepsilon^2\nabla\cdot
\bigg\{
\nabla\psi_{\varepsilon}\otimes \big(W^{\varepsilon}\ast G\big)
\bigg\}
+\varepsilon^2\nabla(W^{\varepsilon}\ast G)\nabla
\psi_{\varepsilon}
+\varepsilon^2
\psi_{\varepsilon}W^{\varepsilon}\ast\Delta G
-\varepsilon^2
\Delta\hat{\xi_{\varepsilon}}-\varepsilon^2\Delta\hat{\eta_{\varepsilon}}\\
&+\varepsilon^3\nabla
\big(\phi^{\varepsilon}\ast_2\partial G\big)\nabla\psi_{\varepsilon}+
\varepsilon^3\nabla\cdot\bigg\{
\nabla\psi_{\varepsilon}\otimes
\big(\phi^{\varepsilon}\ast_2\partial G\big)
\bigg\}
+\varepsilon^3\psi_{\varepsilon}
\nabla\cdot\Big(\phi^{\varepsilon}\ast_2\nabla\partial G\Big).
\end{aligned}
\end{equation*}

On account of the equations $\eqref{pde1.1}$ and $\eqref{pde2.1}$, respectively,
the above expression can be rewritten as, in terms of the order of the power of $\varepsilon$,
\begin{equation*}
\begin{aligned}
&f-\nabla p_0-\psi_{\varepsilon}W^{\varepsilon}(\cdot,0)G
+\xi_{\varepsilon}+\eta_{\varepsilon}\\
&-\varepsilon\psi_{\varepsilon}
\bigg[
\partial_t\phi^{\varepsilon}\ast_2\partial G+
\phi^{\varepsilon}(\cdot,0):\partial G
\bigg]+2\varepsilon\psi_{\varepsilon}\big[(\partial W)^{\varepsilon}\ast_2\partial G\big]
-\varepsilon
\nabla\psi_{\varepsilon}
\pi^{\varepsilon}\ast_1 G
-\varepsilon\psi_{\varepsilon}\nabla G\ast\pi^{\varepsilon}\\
&+\varepsilon^2\psi_{\varepsilon}
\nabla\cdot
\Big(
(\nabla\phi)^{\varepsilon}\ast_{2}\partial G
\Big)
+\varepsilon^2\nabla\cdot
\bigg\{
\nabla\psi_{\varepsilon}\otimes \big(W^{\varepsilon}\ast G\big)
\bigg\}
+\varepsilon^2\nabla\big(W^{\varepsilon}\ast G\big)\nabla
\psi_{\varepsilon}
+\varepsilon^2
\psi_{\varepsilon} W^{\varepsilon}\ast\Delta G
-\varepsilon^2
\Delta\hat{\xi_{\varepsilon}}-\varepsilon^2\Delta\hat{\eta_{\varepsilon}}\\
&+\varepsilon^3\nabla
\big(\phi^{\varepsilon}\ast_2\partial G\big)\nabla\psi_{\varepsilon}+
\varepsilon^3\nabla\cdot\bigg\{
\nabla\psi_{\varepsilon}\otimes
\big(\phi^{\varepsilon}\ast_2\partial G\big)
\bigg\}
+\varepsilon^3\psi_{\varepsilon}
\nabla\cdot\Big(\phi^{\varepsilon}\ast_2\nabla\partial G\Big)
:=I_1+\varepsilon I_2+\varepsilon^2 I_3+\varepsilon^3 I_4,
\end{aligned}
\end{equation*}
which have proved the first line of the equations $\eqref{pde3.13}$.

\medskip

\textbf{Part 2.} We now check the divergence-free condition of $w_{\varepsilon}$.
It follows that
\begin{equation*}
\nabla\cdot w_{\varepsilon}=\nabla\cdot u_{\varepsilon}-\nabla\cdot
\bigg\{
\psi_{\varepsilon}\big[W^{\varepsilon}\ast G+\varepsilon\phi^{\varepsilon}
\ast_2 \partial G\big]
\bigg\}
+\nabla\cdot\hat{\xi_{\varepsilon}}+\nabla\cdot\hat{\eta_{\varepsilon}},
\end{equation*}
and it suffices to show
\begin{equation}\label{f:4.6}
\nabla\cdot \bigg\{
\psi_{\varepsilon}[W^{\varepsilon}\ast G+\varepsilon\phi^{\varepsilon}\ast_2 \partial G]
\bigg\} = \nabla\cdot\hat{\xi_{\varepsilon}} + \nabla\cdot\hat{\eta_{\varepsilon}}.
\end{equation}

\medskip

By divergence-free condition of $\eqref{pde2.1}$ and the equations $\eqref{divc1*}$
that $\phi$ satisfies, we obtain that
\begin{equation*}
\begin{aligned}
&\quad\nabla\cdot
\bigg\{
\psi_{\varepsilon}[W^{\varepsilon}\ast G+\varepsilon\phi^{\varepsilon}\ast_2 \partial G]
\bigg\}\\
&=\nabla\psi_{\varepsilon}\cdot[W^{\varepsilon}\ast G+\varepsilon\phi^{\varepsilon}\ast_2\partial G]
+\psi_{\varepsilon}\nabla\cdot[W^{\varepsilon}\ast G+\varepsilon\phi^{\varepsilon}\ast_2\partial G]\\
&=\nabla\psi_{\varepsilon}\cdot[W^{\varepsilon}\ast G+\varepsilon\phi^{\varepsilon}\ast_2\partial G]
+\psi_{\varepsilon}\bigg[
\varepsilon^{-1}(\nabla\cdot W)^{\varepsilon}\ast_1 G+W^{\varepsilon}\ast_2\partial G
+(\nabla\cdot\phi)^{\varepsilon}\ast_2\partial G
+\varepsilon\phi^{\varepsilon}\ast_{3}\partial^2 G
\bigg]\\
&\overset{\eqref{divc1*}}{=}\nabla\psi_{\varepsilon}\cdot[W^{\varepsilon}\ast G+\varepsilon\phi^{\varepsilon}\ast_2\partial G]
+\psi_{\varepsilon}W^{\varepsilon}\ast_2\partial G
+\psi_{\varepsilon}
(\frac{A}{|Y_f|}-W^{\varepsilon}
)\ast_2\partial G
+\varepsilon\psi_{\varepsilon}
\phi^{\varepsilon}\ast_3\partial^2 G\\
&=
\nabla\psi_{\varepsilon}\cdot[W^{\varepsilon}\ast G+\varepsilon\phi^{\varepsilon}\ast_2\partial G]
+\psi_{\varepsilon}
\frac{A}{|Y_f|}\ast_2\partial G
+\varepsilon\psi_{\varepsilon}
\phi^{\varepsilon}\ast_3\partial^2G,
\end{aligned}
\end{equation*}
and this further implies
\begin{equation}\label{f:4.5}
\begin{aligned}
&\quad\nabla\cdot
\bigg\{
\psi_{\varepsilon}[W^{\varepsilon}\ast G+\varepsilon\phi^{\varepsilon}\ast_2 \partial G]
\bigg\}\\
&=\nabla\psi_{\varepsilon}\cdot\big[(W^{\varepsilon}
-A)\ast G\big]
+\nabla\psi_{\varepsilon}\cdot(A\ast G)
+\psi_{\varepsilon}
\frac{A}{|Y_f|}\ast_2\partial G
+\varepsilon\nabla\psi_{\varepsilon}
\cdot\big(\phi^{\varepsilon}\ast_2\partial G\big)
+\varepsilon\psi_{\varepsilon}
\phi^{\varepsilon}\ast_3\partial^2G.
\end{aligned}
\end{equation}

\medskip

Consequently, combining \eqref{pde3.8}, \eqref{divc2}, \eqref{divc3},
$\eqref{pde3.11}$
and $\eqref{f:4.5}$, we have proved
the equality $\eqref{f:4.6}$, which gives us
the divergence-free condition of $\eqref{pde3.13}$,
and ends the whole proof.
\end{proof}

\medskip
Without a proof, we state the following energy estimate.

\begin{lemma}[energy estimates]\label{lemma3.1}
Let $0<T<\infty$ and $d\geq 2$.
Suppose that
$\Omega\subset\mathbb{R}^d$ is a bounded Lipschitz domain, and
the perforated one $\Omega_\varepsilon$
satisfies the geometrical hypothesis $\eqref{RA}$.
Let $\psi_{\varepsilon}$ be the radial cut-off function defined in Lemma $\ref{lemma:cut-off}$.
%at the start of Section \ref{sec3},
Given $\Theta\in L^2(0,T;L^2(\Omega)^d)$ and $\Lambda$, $\Xi\in L^2(0,T;L^2(\Omega)^{d\times d})$, assume
that $u_{\varepsilon}\in L^2(0,T;H^1(\Omega_{\varepsilon})^d)$ is a weak solution of
\begin{equation}\label{pde3.14}
\left\{\begin{aligned}
\partial_t u_{\varepsilon}-\varepsilon^2\Delta u_{\varepsilon}+\nabla p_{\varepsilon}&=\Theta+\varepsilon\nabla\cdot \Lambda+
\varepsilon\psi_{\varepsilon}
\nabla\cdot \Xi &\qquad& \text{in}\quad \Omega_{\varepsilon}\times (0,T];\\
\nabla\cdot u_{\varepsilon}&=0 &\qquad&\text{in}\quad \Omega_{\varepsilon}\times (0,T];\\
u_{\varepsilon}&=0 &\qquad&\text{on}\quad
\partial\Omega_{\varepsilon}\times(0,T];\\
u_{\varepsilon}&=0 &\qquad&\text{in}
\quad\Omega_{\varepsilon}\times\{t=0\}.
\end{aligned}\right.
\end{equation}
Then, for any $\varepsilon>0$, one can derive that
\begin{equation}\label{pde3.7}
\begin{aligned}
\|u_{\varepsilon}\|_{L^{\infty}
(0,T;L^2(\Omega_{\varepsilon}))}
&+\varepsilon
\|\nabla u_{\varepsilon}\|_{L^{2}
(0,T;L^2(\Omega_{\varepsilon}))}\\
&\lesssim
\bigg\{\|\Theta\|_{L^2(0,T;L^2(\Omega_{\varepsilon}))}
+\|\Lambda\|_{L^2(0,T;L^2(\Omega_{\varepsilon}))}
+\|\Xi\|_{L^2(0,T;L^2(\text{supp}
(\psi_{\varepsilon})))}\bigg\},
\end{aligned}
\end{equation}
where the multiplicative constant
depends on $d$ and the characters of
$\Omega$ and $Y_s$, but independent of $\varepsilon$.
\end{lemma}

\subsection{\centering Some auxiliary results}\label{subsec:5.1}

\begin{lemma}\label{lemma3.2}
Let $0<\varepsilon\ll 1$ and $\psi_\varepsilon$ be defined in Lemma $\ref{lemma:cut-off}$.
Given $f\in L^2(0,T;C^{1,1/2}(\bar{\Omega})^d)$, suppose that
$p_0$ is associated with $f$ by the equation \eqref{pde1.2}.
Let $F=f-\nabla p_0$, and
\begin{equation*}
G(x)=S_{\delta}(\varphi_\varepsilon F)(x) = \int_{\mathbb{R}^d}
dy\zeta_\delta(x-y)(\varphi_\varepsilon F)(y),
\qquad \zeta_\delta = \delta^{-d}\zeta(\cdot/\delta),
\end{equation*}
with $\delta = \varepsilon/4$, where $\zeta\in C_0^\infty(B(0,1/2))$
satisfies $\zeta\geq 0$ and $\int_{\mathbb{R}^d}\zeta =1$
(see also in Subsection $\ref{notation}$).
Then, there hold the following estimates
\begin{equation}\label{pde3.1}
\begin{aligned}
\|G-F\|_{L^2(0,T;
L^{\infty}(\text{supp}(\psi_{\varepsilon})))}
\lesssim \varepsilon
\|f\|_{L^2(0,T;C^{1,1/2}(\bar{\Omega}))};\\
\|\nabla(G-F)\|_{
L^2(0,T;L^{\infty}(\text{supp}(\psi_{\varepsilon})))}
\lesssim \varepsilon^{1/2}
\|f\|_{L^2(0,T;C^{1,1/2}(\bar{\Omega}))},
\end{aligned}
\end{equation}
and
\begin{equation}\label{pde3.4}
\begin{aligned}
\|G\|_{L^{2}(0,T;L^{\infty}(\Omega))}+
\|\nabla G\|_{L^{2}(0,T;L^{\infty}
(\text{supp}(\psi_{\varepsilon})))}
&\lesssim\|f\|_{L^{2}(0,T;C^{1,1/2}
(\bar{\Omega}))};\\
\|\nabla^2 G\|_{L^2(0,T;L^{\infty}(\text{supp}(\psi_{\varepsilon})))}
&\lesssim\varepsilon^{-1/2}\|f\|_{L^{2}(0,T;C^{1,1/2}
(\bar{\Omega}))}.
\end{aligned}
\end{equation}
where the multiplicative constant depends only on $d$ and $\zeta$.
\end{lemma}

\begin{remark}
\emph{Once the estimate $\eqref{pri:3.1}$ is established,
a routine computation leads to the desired estimates $\eqref{pde3.1}$
and $\eqref{pde3.4}$. We mention that the cut-off function $\varphi_\varepsilon$ in $G$ is not so crucial as
$\psi_\varepsilon$, and the purpose of introducing $G$ is to weaken the assumption on the smoothness of $f$. (See e.g. \cite{Shen18,Xu16}.)}
\end{remark}

\subsection{\centering Error estimate on velocity}\label{subsec:5.2}

\noindent
\textbf{The proof of the estimate $\eqref{es1.1}$.}
Based on Lemma $\ref{lemma:4.3}$, the first two steps listed in Subsection $\ref{subsec:1.4}$ have been completed. Now, we will show a detailed proof for Step 3 therein.
The main task in this step will be reduced to estimating
each term presented in $\eqref{notation:5.2}$
according to the form of the right-hand side of the equations
$\eqref{pde3.14}$ given in Lemma $\ref{lemma3.1}$. Therefore, we mainly
divide this step into four sub-steps.

\textbf{Step 3-1.}
We start from the first line of $\eqref{notation:5.2}$ that
\begin{equation*}
I_1=f-\nabla p_0-\psi_{\varepsilon}W^{\varepsilon}(\cdot,0)G
+\xi_{\varepsilon}+\eta_{\varepsilon},
\end{equation*}
and we need to estimate these quantities:
$\|f-\nabla p_0-\psi_{\varepsilon}G\|_{
L^2(\Omega_{\varepsilon,T})}$,
$\|\xi_{\varepsilon}\|_{L^2(\Omega_{\varepsilon,T})}
,\|\eta_{\varepsilon}\|_{L^2(\Omega_{\varepsilon,T})}$,
where the last two terms had been done by Propositions
$\ref{lemma3.6-1}$ and $\ref{lemma3.6-2}$, respectively.
Note that $W^\varepsilon(\cdot,0)$ is the identity matrix.
For the first one, we have
\begin{equation*}
\begin{aligned}
\|f-\nabla p_0-\psi_{\varepsilon}G\|_{L^2(\Omega_{\varepsilon,T})}
&\leq
\|F-\psi_{\varepsilon}G+\psi_{\varepsilon}F
-\psi_{\varepsilon}F
\|_{L^2(\Omega_{\varepsilon,T})}\\
&\leq \|\psi_{\varepsilon}(G-F)\|_{
L^2(\Omega_{\varepsilon,T})}+
\|(1-\psi_{\varepsilon})F\|_{
L^2(\Omega_{\varepsilon,T})}\\
&\lesssim \|G-F\|_{L^2(0,T;L^2(\text{supp}(\psi_{\varepsilon})))}+\|F\|_{
L^2(0,T;L^2(O_{\varepsilon}))}.
\end{aligned}
\end{equation*}
Hence, it follows that
\begin{equation}\label{pde5.3}
\begin{aligned}
\|I_1\|_{L^2(\Omega_{\varepsilon,T})}
&\lesssim \|G-F\|_{L^2(0,T;L^2(\text{supp}(\psi_{\varepsilon})))}+\|F\|_{
L^2(0,T;L^2(O_{\varepsilon}))}+\|\xi_{\varepsilon}\|_{L^2(\Omega_{\varepsilon,T})}
+\|\eta_{\varepsilon}\|_{L^2(\Omega_{\varepsilon,T})}\\
&\overset{\eqref{pde3.1},\eqref{pri:3.1},\eqref{pri:4.1},
\eqref{pri:5.2}}{\lesssim}
\varepsilon^{1/2}\|f\|_{L^2(0,T;C^{1,1/2}(\bar{\Omega}))}.
\end{aligned}
\end{equation}

\textbf{Step 3-2.} Recalling the second line of $\eqref{notation:5.2}$, we have
\begin{equation*}
\begin{aligned}
\varepsilon I_2=&-\varepsilon\psi_{\varepsilon}
\bigg[
\partial_t\phi^{\varepsilon}\ast_2\partial G+
\phi^{\varepsilon}(\cdot,0):\partial G
\bigg]
+2\varepsilon\psi_{\varepsilon}(\partial W)^{\varepsilon}\ast_2\partial G
-
\varepsilon\nabla\psi_{\varepsilon}
\pi^{\varepsilon}\ast_1 G
-\varepsilon\psi_{\varepsilon}\nabla G\ast\pi^{\varepsilon}\\
&:=I_{21}+I_{22}+I_{23}+I_{24}+I_{25}.
\end{aligned}
\end{equation*}
By using Minkowski's inequality and Young's inequality, and then the periodicity
of $W$ and $\phi$, we obtain that
\begin{equation}\label{pde5.2}
\begin{aligned}
\int_{\Omega_{\varepsilon,T}}|I_{21}+I_{23}|^2
&\lesssim\varepsilon^2
\int_{0}^{T}\int_{\Omega_{\varepsilon}}\psi^2_{\varepsilon}
| \partial_t\phi^{\varepsilon} \ast_2 \partial G|^2 +
\varepsilon^2
\int_{0}^{T}\int_{\Omega_{\varepsilon}}\psi^2_{\varepsilon}
|(\partial W)^{\varepsilon}\ast_2 \partial G|^2 \\
&\lesssim\varepsilon^2
\bigg\{\|\partial_t\phi\|^2_{L^1(0,T;L^2(Y_f))}+
\|\nabla W\|^2_{L^1(0,T;L^2(Y_f))}\bigg\}\|\nabla G\|^2_{L^2(0,T;L^{\infty}
(\text{supp}(\psi_{\varepsilon})))}.
\end{aligned}
\end{equation}
By H\"older's inequality and the periodicity of $\phi$, we have
\begin{equation*}
\int_{\Omega_{\varepsilon,T}}|I_{22}|^2
  =\varepsilon^2\int_{0}^{T}dt\int_{\Omega_{\varepsilon}}
  |\psi_{\varepsilon}
  \phi^{\varepsilon}(\cdot,0):\partial G(\cdot,t)|^2 \lesssim\varepsilon^2
  \|\phi(\cdot,0)\|^2_{L^2(Y_f)}\|\nabla G\|^2_{L^2(0,T;L^{\infty}
  (\text{supp}(\psi_{\varepsilon}))}.
\end{equation*}
In view of Lemma $\ref{lemma:cut-off}$, employing a similar computation as
given for $\eqref{pde5.2}$, we obtain that
\begin{equation*}
\begin{aligned}
\int_{\Omega_{\varepsilon,T}}|I_{24}|^2
&=\varepsilon^2\int_{0}^{T}\int_{\Omega_{\varepsilon}}|
\nabla\psi_{\varepsilon}
\pi^{\varepsilon}\ast_1 G|^2 \\
&\lesssim
\int_{0}^{T}\int_{O_{\varepsilon}}|\pi^{\varepsilon
}\ast_1 G|^2
\lesssim
\varepsilon\|\pi\|^2_{L^1(0,T;L^2(Y_f))}
\|G\|^2_{L^2(0,T;L^{\infty}
(\text{supp}(\psi_{\varepsilon})))},
\end{aligned}
\end{equation*}
and
\begin{equation*}
\begin{aligned}
\int_{\Omega_{\varepsilon,T}}|I_{25}|^2
&=\varepsilon^2\int_{0}^{T}\int_{\Omega_{\varepsilon}}|
\psi_{\varepsilon}
\nabla G \ast \pi^{\varepsilon}|^2
&\lesssim&
\varepsilon^2\|\pi\|^2_{L^1(0,T;L^2(Y_f))}
\|\nabla G\|^2_{L^2(0,T;L^{\infty}
(\text{supp}(\psi_{\varepsilon})))}.
\end{aligned}
\end{equation*}
Therefore, collecting all the estimates of
$I_{21},\cdots, I_{25}$, there holds
\begin{equation}\label{pde5.4}
\begin{aligned}
\|\varepsilon
I_2\|_{L^2(\Omega_{\varepsilon,T})}
& \lesssim
\varepsilon
\bigg\{\|\partial_t\phi\|_{L^1(0,T;L^2(Y_f))}+
\|\nabla W\|_{L^1(0,T;L^2(Y_f))}
+\|\phi(\cdot,0)\|_{L^2(Y_f)}\bigg\}\|\nabla G\|_{L^2(0,T;L^{\infty}
(\text{supp}(\psi_{\varepsilon})))}\\
&\qquad+\varepsilon^{\frac{1}{2}}
\|\pi\|_{L^1(0,T;L^2(Y_f))}
\bigg\{\|G\|_{L^2(0,T;L^{\infty}
(\text{supp}(\psi_{\varepsilon})))}
+\varepsilon^{\frac{1}{2}}\|\nabla G\|_{L^2(0,T;L^{\infty}
(\text{supp}(\psi_{\varepsilon})))}
\bigg\}\\
&\qquad\qquad\qquad
\overset{\eqref{pde2.14},\eqref{pri:4.2},\eqref{pde3.33b},\eqref{pde3.4}}{\lesssim}
\varepsilon^{1/2}\|f\|_{L^2(0,T;C^{1,1/2}(\bar{\Omega}))},
\end{aligned}
\end{equation}
where it is fine to assume $\int_{Y_f}\pi(\cdot,t)=0$ 
for $t>0$, since $\pi\in L^1(0,T;L^2(Y_f)/\mathbb{R})$
due to the estimate $\eqref{pde3.33a}$.

\textbf{Step 3-3.} We now turn to the term $\varepsilon^2I_3$.
By the expression of $I_3$ in $\eqref{notation:5.2}$, we can rewrite it as follows:
\begin{equation*}
\begin{aligned}
\varepsilon^2 I_3
&=
\varepsilon^2\psi_{\varepsilon} W^{\varepsilon}\ast\Delta G
+\varepsilon^2\nabla\big(W^{\varepsilon}\ast G\big)\nabla
\psi_{\varepsilon}\\
&+\varepsilon^2\nabla\cdot
\bigg\{
\nabla\psi_{\varepsilon}\otimes \big(W^{\varepsilon}\ast G\big)
-\nabla\hat{\xi_{\varepsilon}}-\nabla\hat{\eta_{\varepsilon}}\bigg\}
+\varepsilon^2\psi_{\varepsilon}
\nabla\cdot
\Big\{
(\nabla\phi)^{\varepsilon}\ast_2\partial G
\Big\}\\
&:=I_{31}+I_{32}+\varepsilon \nabla\cdot I_{33}
+\varepsilon \psi_{\varepsilon}
\nabla\cdot I_{34}.
\end{aligned}
\end{equation*}
Then, in view of Lemma $\ref{lemma3.1}$, we need to estimate
$I_{31},I_{32},I_{33}$ and $I_{34}$ according to the related form in
the estimate $\eqref{pde3.7}$. We begin with $I_{34}$, and
by using Minkowski's inequality, Young's inequality, and
the periodicity of $\phi$, there holds
\begin{equation*}
\begin{aligned}
\int_{0}^{T}\int_{\text{supp}(\psi_{\varepsilon})}
|I_{34}|^2
&\lesssim \varepsilon^2
\int_{0}^{T}\int_{\text{supp}(\psi_{\varepsilon})}
|(\nabla \phi)^{\varepsilon}\ast_2\partial G|^2
\lesssim \varepsilon^2
\|\nabla\phi\|^2_{L^1(0,T;L^2(Y_f))}\|\nabla G\|^2_{L^2(0,T;L^{\infty}
(\text{supp}(\psi_{\varepsilon})))}.
\end{aligned}
\end{equation*}
We proceed to handle the term $I_{33}$ in the region $\Omega_{\varepsilon,T}$, and a similar computation as given above leads to
\begin{equation*}
\begin{aligned}
\int_{\Omega_{\varepsilon,T}}|I_{33}|^2&
\leq \varepsilon^2\int_{0}^{T}\int_{\Omega_{\varepsilon}}
\big(|\nabla\psi_{\varepsilon}\otimes W_{\varepsilon}\ast G|^2+|\nabla\hat{\xi_{\varepsilon}}|^2
+|\nabla\hat{\eta_{\varepsilon}}|^2\big)\\
&\lesssim \int_{0}^{T}
\int_{O_{\varepsilon}}|W_{\varepsilon}\ast G|^2
+\varepsilon^2\int_{0}^{T}\int_{\Omega_{\varepsilon}}
\big(|\nabla\hat{\xi_{\varepsilon}}|^2
+|\nabla\hat{\eta_{\varepsilon}}|^2\big) \\
&\lesssim\varepsilon
\|W\|^2_{L^1(0,T;L^2(Y_f))}\|G\|^2_{L^2(0,T;
L^{\infty}
(\text{supp}(\psi_{\varepsilon})))}+
\varepsilon^2\|\nabla\hat{\xi_{\varepsilon}}\|^2_{
L^2(\Omega_{\varepsilon,T})}
+\varepsilon^2\|\nabla\hat{\eta_{\varepsilon}}\|^2_{
L^2(\Omega_{\varepsilon,T})}.
\end{aligned}
\end{equation*}
By the same token, we obtain that
\begin{equation*}
\begin{aligned}
\int_{\Omega_{\varepsilon,T}}
|I_{31}|^2&
\lesssim\varepsilon^4
\int_{0}^{T}\int_{\text{supp}(\psi_{\varepsilon})}
|W^{\varepsilon}\ast\Delta G|^2
\lesssim\varepsilon^4\|W\|^2_{L^1(0,T;L^2(Y_f))}
\|\nabla^2 G\|^2_{L^2(0,T;L^{\infty}
(\text{supp}(\psi_{\varepsilon})))},
\end{aligned}
\end{equation*}
and
\begin{equation*}
\begin{aligned}
\int_{\Omega_{\varepsilon,T}}|I_{32}|^2
&
\lesssim\varepsilon^2\int_{0}^{T}\int_{O_{\varepsilon}}
|\nabla\big(W^{\varepsilon}\ast G\big)|^2 \\
&\lesssim\varepsilon
\|\nabla W\|^2_{L^1(0,T;L^2(Y_f))}
\|G\|^2_{L^2(0,T;L^{\infty}
(\text{supp}(\psi_{\varepsilon})))}
+\varepsilon^3\|W\|^2_{L^1(0,T;L^2(Y_f))}
\|\nabla G\|^2_{L^2(0,T;L^{\infty}
(\text{supp}(\psi_{\varepsilon})))}.
\end{aligned}
\end{equation*}
As a result, we have established the following estimate
\begin{equation}\label{pde5.5}
\begin{aligned}
&\|I_{31}+I_{32}+I_{33}\|_{L^2(\Omega_{\varepsilon,T})}
+\|I_{34}\|_{L^2(0,T;L^2(\text{supp}
(\psi_{\varepsilon})))}\\
&\qquad\lesssim
\varepsilon^{1/2}
\bigg\{\|W\|_{L^1(0,T;L^2(Y_f))}
+\|\nabla W\|_{L^1(0,T;L^2(Y_f))}
\bigg\}
\|G\|_{L^2(0,T;L^{\infty}
(\text{supp}(\psi_{\varepsilon})))}\\
&\qquad\quad\quad+\varepsilon\bigg\{
\varepsilon^{1/2}\|W\|_{L^1(0,T;L^2(Y_f))}
+\|\nabla\phi\|_{L^1(0,T;L^2(Y_f))}\bigg\}
\|\nabla G\|_{L^2(0,T;L^{\infty}
(\text{supp}(\psi_{\varepsilon})))}\\
&\qquad\qquad\quad\quad+\varepsilon^2\|W\|_{L^1(0,T;L^2(Y_f))}
\|\nabla^2 G\|_{L^2(0,T;
L^{\infty}(\text{supp}(\psi_{\varepsilon})))}
+\varepsilon\|\nabla\hat{\xi_{\varepsilon}}\|_{
L^2(\Omega_{\varepsilon,T})}
+\varepsilon\|\nabla\hat{\eta_{\varepsilon}}\|_{
L^2(\Omega_{\varepsilon,T})}\\
&\qquad\qquad\qquad\qquad\quad
\overset{\eqref{pde2.14},\eqref{pri:4.2},\eqref{pde3.4},
\eqref{pri:5.3a},\eqref{pri:5.3b}}{\lesssim}
\varepsilon^{1/2}\|f\|_{L^2(0,T;C^{1,1/2}(\bar{\Omega}))}.
\end{aligned}
\end{equation}

\textbf{Step 3-4.} Recall the last line of $\eqref{notation:5.2}$, and
we rewrite it as follows:
\begin{equation*}
\begin{aligned}
\varepsilon^3 I_4
&=
\varepsilon^3\nabla
\big(\phi^{\varepsilon}\ast_2\partial G\big)\nabla\psi_{\varepsilon}
+\varepsilon^3
\nabla\cdot\bigg\{
\nabla\psi_{\varepsilon}\otimes
\big(\phi^{\varepsilon}\ast_2\partial G\big)
\bigg\}
+\varepsilon^3\psi_{\varepsilon}
\nabla\cdot\bigg\{\phi^{\varepsilon}\ast_2\nabla\partial G\bigg\}\\
&=: I_{41}+\varepsilon \nabla\cdot I_{42}+\varepsilon \psi_{\varepsilon}\nabla\cdot
I_{43}.
\end{aligned}
\end{equation*}
By using Minkowski's inequality and Young's inequality, as well as,
the periodicity of $\phi$, we derive that
\begin{equation*}
\begin{aligned}
\int_{\Omega_{\varepsilon,T}}|I_{41}|^2
&\leq\varepsilon^4
\int_{0}^{T}\int_{O_{\varepsilon}}|\nabla
\big(\phi^{\varepsilon}\ast_2\partial G\big)|^2
\lesssim\varepsilon^4
\int_{0}^{T}\int_{O_{\varepsilon}}|\phi^{\varepsilon}\ast_2\nabla\partial G|^2
+\varepsilon^2
\int_{0}^{T}\int_{O_{\varepsilon}}
|(\nabla\phi)^{\varepsilon}\ast_2\partial G|^2 \\
&\lesssim\varepsilon^5
\|\phi\|^2_{L^1(0,T;L^2(Y_f))}\|\nabla^2 G\|^2_{L^2(0,T;L^{\infty}
(\text{supp}(\psi_{\varepsilon})))}
+\varepsilon^3
\|\nabla\phi\|^2_{L^1(0,T;L^2(Y_f))}\|\nabla G\|^2_{L^2(0,T;L^{\infty}
(\text{supp}(\psi_{\varepsilon})))}.\\
\end{aligned}
\end{equation*}
By an analogous argument employed in Step 3-3, there holds
\begin{equation*}
\begin{aligned}
\int_{\Omega_{\varepsilon,T}}|I_{42}|^2
&=\varepsilon^4\int_{0}^{T}\int_{\Omega_{\varepsilon}}
|\nabla\psi_{\varepsilon}\otimes
\phi^{\varepsilon}\ast_2\partial G|^2
\lesssim\varepsilon^{2}\int_{0}^{T}\int_{O_{\varepsilon}}|
\phi^{\varepsilon}\ast_2\partial G|^2\\
&\lesssim\varepsilon^{3}
\|\phi\|^2_{L^1(0,T;L^2(Y_f))}\|\nabla G\|^2_{L^2(0,T;L^{\infty}
(\text{supp}(\psi_{\varepsilon})))},
\end{aligned}
\end{equation*}
and
\begin{equation*}
\begin{aligned}
\int_{0}^{T}\int_{\text{supp}(\psi_{\varepsilon})}|I_{43}|^2
&=\varepsilon^4\int_{0}^{T}\int_{\text{supp}(\psi_\varepsilon)}
|\phi^{\varepsilon}\ast_2\nabla\partial G|^2dxdt
\lesssim \varepsilon^4\|\phi\|^2_{L^1(0,T;L^2(Y_f))}\|\nabla^2 G\|^2_{L^2(0,T;L^{\infty}
(\text{supp}(\psi_{\varepsilon})))}.
\end{aligned}
\end{equation*}
We now collect the above estimates and obtain
\begin{equation}\label{pde5.6}
\begin{aligned}
&\|I_{41}\|_{L^2(\Omega_{\varepsilon,T})}
+\|I_{42}\|_{L^2(\Omega_{\varepsilon,T})}
+\|I_{43}\|_{L^2(0,T;L^2(\text{supp}
(\psi_{\varepsilon})))}\\
&\qquad\qquad\lesssim
\varepsilon^{2}
\|\phi\|_{L^1(0,T;L^2(Y_f))}\|\nabla^2 G\|^2_{L^2(0,T;L^{\infty}
(\text{supp}(\psi_{\varepsilon})))}\\
&\qquad\qquad\qquad\qquad+\varepsilon^{3/2}
\bigg\{\|\phi\|_{L^1(0,T;L^2(Y_f))}+
\|\nabla\phi\|_{L^1(0,T;L^2(Y_f))}
\bigg\}
\|\nabla G\|_{L^2(0,T;L^{\infty}
(\text{supp}(\psi_{\varepsilon})))}\\
&\qquad\qquad\qquad\qquad\qquad\qquad
\overset{{\eqref{pri:4.2}, \eqref{pde3.4}}}{\lesssim}
\varepsilon^{3/2}\|f\|_{L^2(0,T;C^{1,1/2}(\bar{\Omega}))}.
\end{aligned}
\end{equation}

\medskip

In the end, in view of the equations $\eqref{pde3.13}$, we can rewrite its right-hand side
as
\begin{equation}\label{f:6.1}
I_1+\varepsilon I_2 + \varepsilon^2I_3
+\varepsilon^3 I_4
= \underbrace{I_1 + \varepsilon I_2 + I_{31}+I_{32} + I_{41}}_{\Theta}
+\varepsilon\nabla\cdot\underbrace{\big(I_{33}+I_{42}\big)}_{\Lambda}
+\varepsilon\psi_\varepsilon\nabla\cdot
\underbrace{\big(I_{34}+I_{43}\big)}_{\Xi},
\end{equation}
which coincides with the form of the counterpart of
the equations $\eqref{pde3.14}$. Thus, appealing to the estimate
$\eqref{pde3.7}$, together with  \eqref{pde5.3}, \eqref{pde5.4}, \eqref{pde5.5} and \eqref{pde5.6},
we obtain that
\begin{equation}\label{k-10}
\varepsilon
\|\nabla w_{\varepsilon}\|_{L^{2}
(0,T;L^2(\Omega_{\varepsilon}))}
\lesssim \varepsilon^{1/2}\|f\|_{L^2(0,T;C^{1,1/2}(\bar{\Omega}))},
\end{equation}
which finally leads to the desired estimate \eqref{es1.1}, and complete all the proof.
\qed

\subsection{\centering Error estimate on pressure}\label{subsec:5.3}

\begin{remark}\label{remark:6.1}
\emph{Regarding the equation $\eqref{f:6.1}$,
if the right-hand side of $\eqref{pde3.13}$ is
considered in $L^2(0,T;H^{-1}(\Omega_\varepsilon))$,
it is generally not feasible to obtain a pressure term
in $L^2(\Omega_{\varepsilon,T})$-norm
(see e.g., \cite[Proposition 5]{Simon99}).
However, without aiming to derive smallness from
the right-hand side of $\eqref{pde3.13}$,
one can verify that it belongs to 
$L^2(\Omega_{\varepsilon,T})$-space.
Given the zero initial-boundary data for the equation $\eqref{pde3.13}$,
one can anticipate favorable regularity properties for the pressure term
(see e.g. \cite[Theorem 1]{Wolf18}).
Consequently, it remains reasonable to proceed with our analysis
of the pressure term through meticulous computation.}
\end{remark}

\noindent
\textbf{The proof of the estimate $\eqref{pri:1.3}$.}
As mentioned in Step 5 in Subsection $\ref{subsec:1.4}$, we are required to study the inertial term $\partial_t w_\varepsilon$, and then appeal to a duality argument coupled with the estimate on velocity to derive the desired estimate $\eqref{pri:1.3}$. Therefore, we divide the detailed proof
into two sub-steps.

\medskip

\textbf{Step 5-1.}  We first show the estimate on $\partial_t w_\varepsilon$ by claiming that there holds
\begin{equation}\label{f:6.2}
   \|\partial_t w_\varepsilon\|_{L^2(\Omega_{\varepsilon,T})}
   \lesssim \varepsilon^{1/2}\|f\|_{L^2(0,T;C^{1,1/2}(\bar\Omega))},
\end{equation}
(which have partially proved the estimate $\eqref{pri:1.3}$,) and then recall the estimates on the notations $\Theta, \Lambda$ and $\Xi$
presented in $\eqref{f:6.1}$, i.e.,
\begin{equation}\label{f:6.3}
\begin{aligned}
\|(\Theta,\Lambda,\psi_{\varepsilon}\Xi)\|_{L^2(\Omega_{\varepsilon,T})}
\lesssim \varepsilon^{1/2}\|f\|_{L^2(0,T;C^{1,1/2}(\bar\Omega))}.
\end{aligned}
\end{equation}

\medskip

Based upon the stated estimates $\eqref{f:6.2}$ and
$\eqref{f:6.3}$, we can show the estimate $\eqref{pri:1.3}$ on the pressure.
We start from taking any test function $H\in L^2(\Omega_\varepsilon)$ with
$\int_{\Omega_\varepsilon}H=0$, and constructing test function as follows:
\begin{equation}\label{pde:6.1}
\left\{\begin{aligned}
\nabla\cdot v_\varepsilon &=  H &\quad&\text{in}~\Omega_{\varepsilon};\\
v_\varepsilon &= 0 &\quad&\text{on}~\partial\Omega_{\varepsilon},
\end{aligned}\right.
\end{equation}
whose solution satisfies the estimate $\eqref{pde3.21}$.
In this regard, for any $\varrho\in L^2(0,T)$ and $c\in\mathbb{R}$, there holds
\begin{equation*}
\begin{aligned}
 \int_{\Omega_{\varepsilon,T}}
 \varrho(q_\varepsilon-c)
 &H
 \overset{\eqref{pde:6.1}}{=} -\int_{0}^{T}\varrho \int_{\Omega_{\varepsilon}}\nabla q_\varepsilon\cdot v_\varepsilon \\
 &\overset{\eqref{pde3.13},\eqref{f:6.1}}{=} \varepsilon
 \int_{0}^{T}\varrho\int_{\Omega_\varepsilon}
 (\varepsilon\nabla w_\varepsilon+\Lambda+\psi_\varepsilon\Xi):\nabla v_\varepsilon
 -\int_{0}^{T}\varrho\int_{\Omega_\varepsilon}
 (\Theta-\partial_tw_\varepsilon -\varepsilon\Xi\nabla\psi_\varepsilon)\cdot v_\varepsilon.
\end{aligned}
\end{equation*}

Applying Poincar\'e's inequality to $v_\varepsilon$, as well as, H\"older's inequality, we have
\begin{equation*}
\begin{aligned}
&\big|\int_{\Omega_{\varepsilon,T}}\varrho(q_\varepsilon-c)H\big|\\
&\lesssim \varepsilon\|\varrho\|_{L^2(0,T)}\|\nabla v_\varepsilon\|_{L^2(\Omega_\varepsilon)}
\Big\{\varepsilon\|\nabla w_\varepsilon\|_{L^2(\Omega_{\varepsilon,T})}
+\|\partial_t w_\varepsilon\|_{L^2(\Omega_{\varepsilon,T})}
+\|(\Theta,\Lambda,\psi_\varepsilon\Xi)\|_{L^2(\Omega_{\varepsilon,T})}\Big\},
\end{aligned}
\end{equation*}
and then it follows from the estimate $\eqref{pde3.21}$ and the duality argument that
\begin{equation*}
\begin{aligned}
\|q_\varepsilon-c\|_{L^2(\Omega_{\varepsilon,T})}
&\lesssim \Big\{\varepsilon\|\nabla w_\varepsilon\|_{L^2(\Omega_{\varepsilon,T})}
+\|\partial_t w_\varepsilon\|_{L^2(\Omega_{\varepsilon,T})}
+\|(\Theta,\Lambda,\psi_\varepsilon\Xi)\|_{L^2(\Omega_{\varepsilon,T})}\Big\}\\
&\overset{\eqref{es1.1},\eqref{f:6.2},\eqref{f:6.3}}{\lesssim} \varepsilon^{1/2}\|f\|_{L^2(0,T;C^{1,1/2}(\bar\Omega))}.
\end{aligned}
\end{equation*}

As a result, this
implies the desired estimate $\eqref{pri:1.3}$.

\medskip

\textbf{Step 5-2.} We now show the estimate $\eqref{f:6.2}$.
Taking $\partial_t w_\varepsilon$ as the test function on both sides of
$\eqref{pde3.13}$, we have
\begin{equation}\label{}
 \int_{\Omega_\varepsilon}|\partial_t w_\varepsilon|^2
 + \frac{\varepsilon^2}{2}\frac{d}{dt}\int_{\Omega_\varepsilon}|\nabla w_\varepsilon|^2
 \overset{\eqref{f:6.1}}{=} \int_{\Omega_\varepsilon}\Theta\cdot \partial_tw_\varepsilon
 -\varepsilon\int_{\Omega_\varepsilon}(\Lambda+\psi_\varepsilon\Xi)
 :\partial_t\nabla w_\varepsilon
 -\varepsilon\int_{\Omega_\varepsilon}
 \Xi\nabla\psi_\varepsilon\cdot\partial_t w_\varepsilon.
\end{equation}
Then, integrating both sides above from $0$ to $T$, as well as, integrating by parts with respect to the temporal variable, we can derive that
\begin{equation*}
\begin{aligned}
\int_{0}^{T}\int_{\Omega_\varepsilon}|\partial_t w_\varepsilon|^2
+ \frac{\varepsilon^2}{2}\int_{\Omega_\varepsilon}|\nabla w_\varepsilon(\cdot,T)|^2
&= \int_{0}^{T}\int_{\Omega_\varepsilon}
(\Theta - \varepsilon\Xi\nabla\psi_\varepsilon)\cdot\partial_t w_\varepsilon
+\varepsilon\int_{0}^{T}\int_{\Omega_\varepsilon}
(\partial_t\Lambda+\psi_\varepsilon\partial_t\Xi):\nabla w_\varepsilon \\
&-\varepsilon\int_{\Omega_\varepsilon}
(\Lambda+\psi_\varepsilon\Xi)(\cdot,T):\nabla w_\varepsilon(\cdot,T).
\end{aligned}
\end{equation*}
This together with Young's inequality implies
\begin{equation}\label{f:6.4}
\begin{aligned}
 \int_{\Omega_{\varepsilon,T}}
 |\partial_t w_\varepsilon|^2
 + \varepsilon^2\int_{\Omega_\varepsilon}|\nabla w_\varepsilon(\cdot,T)|^2
 &\lesssim \int_{\Omega_{\varepsilon,T}}
 \big(|\Theta|^2
 + |\psi_\varepsilon\Xi|^2 + |\varepsilon\nabla w_\varepsilon|^2\big)
 + \int_{\Omega_{\varepsilon,T}}
 \big(|\partial_t \Lambda|^2 +|\psi_\varepsilon\partial_t\Xi|^2\big)\\
 &+\int_{\Omega_\varepsilon}\big(|\Lambda(\cdot,T)|^2
 + |\Xi(\cdot,T)|^2\big) =: T_1 + T_2 + T_3.
\end{aligned}
\end{equation}

The relatively easy term is $T_1$, and it follows from the estimates
$\eqref{f:6.3}$ and $\eqref{k-10}$ that
\begin{equation}\label{f:6.5}
 \sqrt{T_1} \lesssim \varepsilon^{1/2}\|f\|_{L^2(0,T;C^{1,1/2}(\bar\Omega))}.
\end{equation}

Then, we turn to the second term $T_2$ by recalling the expression
$\eqref{f:6.1}$, i.e., $\partial_t \Lambda
= \partial_t I_{33} + \partial_t I_{42} $ with
\begin{equation*}
\left\{\begin{aligned}
&\partial_t I_{33}
= \varepsilon\nabla\psi_\varepsilon\otimes
\big((\partial_t W)^\varepsilon\ast G\big)
+ \varepsilon\nabla\psi_\varepsilon\otimes
\big(W^\varepsilon(\cdot,0)G(\cdot,t)\big)
-\varepsilon\nabla\xi_\varepsilon-\varepsilon\nabla\eta_\varepsilon;\\
&\partial_t I_{42}
= \varepsilon^2\nabla\psi_\varepsilon\otimes
\big((\partial_t\phi)^\varepsilon\ast_2\partial G\big)
+ \varepsilon^2\nabla\psi_\varepsilon\otimes
\big(\phi^\varepsilon(\cdot,0):\partial G(\cdot,t)\big);
\end{aligned}\right.
\end{equation*}
and $\partial_t\Xi = \partial_t I_{34}+\partial_tI_{43}$.
A routine computation as used in Subsection $\ref{subsec:5.2}$ leads to
\begin{equation*}
\begin{aligned}
\|\partial_t \Lambda\|_{L^2(\Omega_{\varepsilon,T})}
&\lesssim \varepsilon^{\frac{1}{2}}
\Big\{\|\partial_t W\|_{L^1(0,T;L^2(Y_f))}+1\Big\}
\|G\|_{L^2(0,T;L^\infty(\text{supp}(\psi_\varepsilon)))}
+\varepsilon\|\nabla\xi_\varepsilon
+\nabla\eta_\varepsilon\|_{L^2(\Omega_{\varepsilon,T})}\\
&+\varepsilon^{\frac{3}{2}}
\Big\{\|\partial_t \phi\|_{L^1(0,T;L^2(Y_f))}+\|\phi(\cdot,0)\|_{L^2(Y_f)}\Big\}
\|\nabla G\|_{L^2(0,T;L^\infty(\text{supp}(\psi_\varepsilon)))} \\
&\overset{\eqref{pde3.33b},\eqref{pri:4.2},\eqref{pri:4.1},
\eqref{pri:5.2},\eqref{pde3.4}}{\lesssim} \varepsilon^{1/2}\|f\|_{L^2(0,T;C^{1,1/2}(\bar\Omega))}.
\end{aligned}
\end{equation*}

By the same token, we have $\|\psi_\varepsilon\partial_t
\Xi\|_{L^2(\Omega_{\varepsilon,T})}
\lesssim \varepsilon \|f\|_{L^2(0,T;C^{1,1/2}(\bar\Omega))}$, and
we don't reproduce the proof here.
This coupled with the above estimate provides us with
\begin{equation}\label{f:6.6}
\sqrt{T_2} \lesssim \varepsilon^{1/2}\|f\|_{L^2(0,T;C^{1,1/2}(\bar\Omega))}.
\end{equation}

Finally, we turn to the term $T_3$.
According to the expression of $\eqref{f:6.1}$,
we merely compute the first term in
$\int_{\Omega_{\varepsilon}}|\Lambda(\cdot,T)|^2$ as an example, i.e.,
there holds
\begin{equation*}
\begin{aligned}
\varepsilon^2\int_{\Omega_{\varepsilon}}
|\nabla\psi_\varepsilon
\otimes W^\varepsilon\ast G(\cdot,T)|^2
&\lesssim \varepsilon
\Big(\int_{0}^{T}ds\|W(\cdot,T-s)\|_{L^2(Y_f)}
\|G(\cdot,s)\|_{L^\infty(\Omega)}\Big)^2\\
&\lesssim \varepsilon\|W\|_{L^1(0,T;L^2(Y_f))}^2\|G\|_{L^2(0,T;L^\infty(\Omega))}^2
\overset{\eqref{pde2.14},\eqref{pde3.4}}{\lesssim}\varepsilon \|f\|_{L^2(0,T;C^{1,1/2}(\bar\Omega))}^2.
\end{aligned}
\end{equation*}
and this together with similar computations given for the other terms in $T_3$ provides us with
\begin{equation}\label{f:6.7}
\sqrt{T_3} \lesssim \varepsilon^{\frac{1}{2}}\|f\|_{L^2(0,T;C^{1,1/2}(\bar\Omega))}.
\end{equation}

\medskip
Plugging the estimates $\eqref{f:6.5}$, $\eqref{f:6.6}$ and
$\eqref{f:6.7}$ back into $\eqref{f:6.4}$, we have proved the stated
estimate $\eqref{f:6.2}$, and we have completed the whole proof.
\qed

%%%%%%%%%%%%%%%%%%%%%%%%%%%%%%%%%%%%%%%%%%%%%%%%%%%
%%%%%%%%%%%%%%%%%%%%%%%%%%%%%%%%%%%%%%%%%%%%%%%%%%%
%%%%%%%%%%%%%%%%%%%%%%%%%%%%%%%%%%%%%%%%%%%%%%%%%%%

\medskip

\noindent
\textbf{Proof of Corollary $\ref{cor:2}$.}
Treating time variable as a parameter,
the idea is totally similar to that given in \cite[pp.22]{Shen20}, and
we provide a proof for reader's convenience. The key ingredient is modifying the value of the effective pressure given in $\eqref{pde1.2}$ as follows
\begin{equation*}
\tilde{p}_{0}(x,t):=
\left\{
\begin{aligned}
&p_{0}(x,t)&\quad&\text{if}\quad x\in\Omega_{\varepsilon},\\
&\dashint_{\varepsilon(Y_f+z_{k})}p_{0}(\cdot,t)&\quad&\text{if}
\quad x\in\varepsilon(Y_{s}+z_{k}) \quad\text{and}\quad
\varepsilon(Y+z_k)\subset\Omega \text{~for~some~}z_{k}\in \mathbb{Z}^d.
\end{aligned}\right.
\end{equation*}

By definition, it is not hard to derive that, for any $c\in\mathbb{R}$,
\begin{equation}\label{f:6.8}
\int_{0}^{T}\int_{\Omega\setminus\Omega_\varepsilon}
|\tilde{p}_\varepsilon - \tilde{p}_0-c|^2
\leq \frac{|Y_s|}{|Y_f|}\int_{0}^{T}\int_{\Omega_\varepsilon}|p_\varepsilon-p_0-c|^2
\overset{\eqref{pri:1.3}}{\lesssim}\varepsilon
\|f\|_{L^2(0,T;C^{1,1/2}(\bar\Omega))}^2.
\end{equation}
Moreover, by the smoothness of $p_0$ and Poincer\'e's inequality, we can obtain that
\begin{equation}\label{f:6.9}
\int_{0}^{T}
\int_{\Omega\setminus\Omega_\varepsilon}
|\tilde{p}_0 - p_0|^2
\lesssim \varepsilon^2\int_{0}^{T}\int_{\Omega}|\nabla p_0|^2
\overset{\eqref{pri:3.12}}{\lesssim}
\varepsilon^2\|f\|_{L^2(0,T;L^{2}(\Omega))}^2.
\end{equation}

Then, for any $c\in\mathbb{R}$, it is not hard to derive that
\begin{equation*}
\begin{aligned}
\int_{0}^{T}\int_{\Omega}|\tilde{p}_\varepsilon - p_0 - c|^2
&\lesssim \int_{0}^{T}\int_{\Omega}|\tilde{p}_\varepsilon - \tilde{p}_0 - c|^2
+ \int_{0}^{T}\int_{\Omega\setminus\Omega_\varepsilon}|\tilde{p}_0-p_0|^2 \\
&\lesssim
\int_{0}^{T}\int_{\Omega_\varepsilon}|p_\varepsilon - p_0 - c|^2
+ \int_{0}^{T}\int_{\Omega\setminus\Omega_\varepsilon}|\tilde{p}_\varepsilon - \tilde{p}_0-c|^2
+ \int_{0}^{T}\int_{\Omega\setminus\Omega_\varepsilon}|\tilde{p}_0-p_0|^2\\
&\overset{\eqref{pri:1.3},\eqref{f:6.8},\eqref{f:6.9}}{\lesssim}
\varepsilon \|f\|_{L^2(0,T;C^{1,1/2}(\bar\Omega))}^2.
\end{aligned}
\end{equation*}
Consequently, this implies the desired estimate $\eqref{pri:1.4}$ and
we have completed the proof.
\qed

\section{Appendix}\label{sec3*}

%\subsection{Results and ideas}\label{subsec:3.1*}
\noindent
The primary approach involves reformulating
the equations $\eqref{pde1.2}$ as a fixed-point problem, as presented in J.-L. Lions's work
\cite[pp.170]{Lions81}.
The key components are rooted in Schauder theory for elliptic
equations and refined corrector estimates.
Consequently, we can first establish
the short-time existence of solutions
by invoking the absolute continuity of integrals of
$\partial_t A$, and then extend the solution to
a finite time by induction arguments (see Lemma $\ref{L3.1}$ and $\ref{L3.2}$, respectively).

\begin{remark}
\emph{If replacing
H\"older's norm by Sobolev norm with respect to the spatial variable,
the results similar to $\eqref{pri:3.1}$ would be established by the
same arguments without any real difficulty, i.e.,
\begin{equation}\label{pri:3.12}
\|p_0\|_{L^q(0,T;H^{m+1}(\Omega))}\leq C
\|f\|_{L^q(0,T;H^{m}(\Omega))},
\qquad m\geq 0,
\end{equation}
where we regard $H^0(\Omega)$ as $L^2(\Omega)$.}
\end{remark}

\subsection{\centering Existence of short-time solution}

%%%%%%%%%%%%%%%%%%%%%%%%%%%%%%%%%%%%%%%%%%%%%%%%%%%
%%%%%%%%%%%%%%%%%%%%%%%%%%%%%%%%%%%%%%%%%%%%%%%%%%%
%%%%%%%%%%%%%%%%%%%%%%%%%%%%%%%%%%%%%%%%%%%%%%%%%%%

\begin{lemma}[properties of $A$ \cite{Allaire92,Lions81,Mikelic91,Sandrakov97}]\label{lemma2.5}
The homogenized coefficient $(A_{ij})_{1\leq i,j\leq d}$ which is defined by \eqref{pde2.2} is symmetric, positive defined and exponentially decay in time. Moreover, one can derive
$|\partial_tA|\in L^{1+\beta}(0,T)$ with
$0<\beta<(2/21)$ and $T>0$.
\end{lemma}

\begin{remark}
\emph{Concerned with $|\partial_tA|\in L^1(0,T)$, the proof can be found in
\cite[pp.127]{Sandrakov97} based upon Galerkin's methods, while
the stated result relies on the refined estimate $\eqref{pri:3.7}$.}
\end{remark}

\begin{lemma}[short-time solution]\label{L3.1}
Assume the same conditions as in Proposition $\ref{P:1}$.
There exists $0<\delta_0\ll 1$, depending on $\|\partial_tA\|_{L^1(0,T)}$, such
that the equation \eqref{pde1.2} possesses
the solution $p_0 \in L^q(0,\delta_0;C^{m+1,\alpha}(\bar{\Omega}))$, satisfying the estimate
\begin{equation}\label{pri:3.2}
\|p_0\|_{L^q(0,\delta_0;C^{m+1,\alpha}(\bar{\Omega}))}\lesssim
\|f\|_{L^q(0,\delta_0;C^{m,\alpha}(\bar{\Omega}))},
\end{equation}
where the multiplicative constant depends on $d$, $|Y_f|$, and $\Omega$.
\end{lemma}

\begin{proof}
In view of the divergence-free and boundary conditions of \eqref{pde1.2},
taking t-derivative on its both sides\footnote{To shorten the formula,
we use the notation $A'$ to represent $\partial_t A$ throughout the proof of Lemmas $\ref{L3.1}$
and $\ref{L3.2}$.}, we have
a new form of the equations \eqref{pde1.2}, i.e.,
\begin{equation}\label{rh1}
\left\{
\begin{aligned}
  \nabla\cdot A(0)\nabla p_0(\cdot,t) +
  \nabla\cdot
  \int_{0}^tds
  &A'(t-s)
  \nabla p_0(\cdot,s)
  =
  \nabla\cdot\Big[
  \int_{0}^tdsA'(t-s)f(\cdot,s)+A(0)f(\cdot,t)
  \Big]\quad  \text{in}~\Omega;\\
 \vec{n}\cdot A(0)\nabla p_0(\cdot,t) +
  \vec{n}\cdot
  \int_{0}^tds
  &A'(t-s)\nabla p_0(\cdot,s)
  =
  \vec{n}\cdot\Big[
  \int_{0}^tdsA'(t-s)f(\cdot,s)+A(0) f(\cdot,t)
  \Big]\quad \text{on}~\partial\Omega.
\end{aligned}\right.
\end{equation}
As the argument developed by J.-L. Lions \cite[pp.170]{Lions81}, for a.e. $t>0$, we introduce a function $\hat{p}(\cdot,t)$ as the solution of
\begin{equation}\label{pde:3.1}
\left\{\begin{aligned}
\nabla\cdot\Big[A(0)\nabla\hat{p}(\cdot,t)-\int_{0}^{t}ds
A'(t-s)\nabla p(\cdot,s)
\Big]&=0,&\quad &\text{in}~\Omega;\\
\vec{n}\cdot\Big[A(0)\nabla\hat{p}(\cdot,t)-\int_{0}^{t}ds
A'(t-s)\nabla p(\cdot,s)
\Big]&=0,&\quad &\text{on}~\partial\Omega,\\
\int_{\Omega}\hat{p}(\cdot,t)&=0.
\end{aligned}\right.
\end{equation}
Denote the operator $\nabla\cdot A(0)\nabla =|Y_f|\Delta$
by $\mathcal{L}$, and
\begin{equation}\label{f:3.1}
  K_0(p)(\cdot,t): = \int_{0}^{t}dsA'(t-s)\nabla p(\cdot,s);
  \qquad
  K_1(p): = \mathcal{L}^{-1}\nabla\cdot K_0(p).
\end{equation}
Then, the solution $\hat{p}$ of $\eqref{pde:3.1}$ can be represented by
$\hat{p}(\cdot,t) = K_1(p)(\cdot,t)$ in $\Omega$.
By Schauder estimates (see e.g.
\cite[pp.89]{Giaquinta-Martinazzi12}), there holds that
\begin{equation}\label{pde2.28}
\begin{aligned}
\|\hat{p}(\cdot,t)\|_{C^{2,1/2}(\bar\Omega)}&\leq
C_1
\bigg\|\int_{0}^{t}ds
A'(t-s)\nabla p(\cdot,s)\bigg\|_{C^{1,1/2}(\bar\Omega)}\\
&\leq C_1
\int_{0}^{t}ds
|A'(t-s)|\|\nabla p(\cdot,s)\|_{C^{1,1/2}(\bar\Omega)}
\leq C_1\int_{0}^{t}ds
|A'(t-s)|\| p(\cdot,s)\|_{C^{2,1/2}(\bar\Omega)},
\end{aligned}
\end{equation}
where $C_1$ depends on $|Y_f|$, $d$, and $\Omega$.
In view of $\eqref{f:3.1}$, we observe that $\mathcal{L}K_1(p) =
\nabla\cdot K_0(p)$, and therefore the equations \eqref{rh1} can be rewritten as
\begin{equation*}
\left\{
\begin{aligned}
  \mathcal{L}\Big[p_0(\cdot,t)+K_1(p_0)(\cdot,t)\Big]  &=\nabla\cdot\Big[
  \int_{0}^tdsA'(t-s)f(\cdot,s)+A(0)f(\cdot,t)
  \Big]&\quad &\text{in}~\Omega;\\
 \frac{\partial}{\partial\nu}\Big[p_0(\cdot,t)+K_1(p_0)(\cdot,t)\Big]&=
  \vec{n}\cdot\Big[
  \int_{0}^tdsA'(t-s)f(\cdot,s)+A(0)f(\cdot,t)
  \Big]&\quad &\text{on}~\partial\Omega,
\end{aligned}\right.
\end{equation*}
where the conormal derivative associated with $\mathcal{L}$
is defined by $\partial/\partial\nu:=\vec{n}\cdot A(0)\nabla$.
Similarly, its solution can be expressed by
\begin{equation}\label{rh3}
  \big(p_0+K_1(p_0)\big)(\cdot,t)=\mathcal{L}^{-1}
  \nabla\cdot\Big[
  \int_{0}^tdsA'(t-s)f(\cdot,s)+A(0)f(\cdot,t)
  \Big]=:\tilde{f}_1(\cdot,t).
\end{equation}

Thus, for some $T_1>0$ (which will be fixed later), if constructing the
following map:
\begin{equation}\label{pde:3.2}
 \mathcal{T}_1(p) :=  \tilde{f}_1 - K_1(p)
 \qquad \forall p\in L^q(0,T_1; C^{m+1,\alpha}(\bar\Omega)),
\end{equation}
the unique existence of the solution of $\eqref{rh1}$ in
$L^q(0,T_1; C^{m+1,\alpha}(\bar\Omega))$ is reduced to verifying that the map
$\mathcal{T}_1: L^q(0,T_1; C^{m+1,\alpha}(\bar\Omega))\to L^q(0,T_1; C^{m+1,\alpha}(\bar\Omega))$ is a strict contraction.

To see this, appealing to Schauder estimates again, we firstly derive that
\begin{equation}\label{rh5}
\|\tilde{f}_1(\cdot,t)\|_{C^{m+1,\alpha}(\bar\Omega)}\lesssim
\int_{0}^{t}ds|A'(t-s)|\|f(\cdot,s)\|_{C^{m,\alpha}(\bar\Omega)}
+|A(0)|\|f(\cdot,t)\|_{C^{m,\alpha}(\bar\Omega)}.
\end{equation}
By integrating both sides of $\eqref{pde2.28}$ and $\eqref{rh5}$ with respect to $t$ from
$0$ to $T_1$, Young's inequality, we obtain
\begin{equation}\label{f:3.2}
\begin{aligned}
\|K_1(p)\|_{L^q(0,T_1;C^{m+1,\alpha}(\bar\Omega))}
&\lesssim \|A'\|_{L^1(0,T_1)}
\|p\|_{L^q(0,T_1;C^{m+1,\alpha}(\bar\Omega))};\\
\|\tilde{f}_1\|_{L^q(0,T_1;C^{m+1,\alpha}(\bar\Omega))}
&\lesssim
\|A'\|_{L^1(0,T_1)}
\|f\|_{L^q(0,T_1;C^{m,\alpha}(\bar\Omega))}
  + \|f\|_{L^q(0,T_1;C^{m,\alpha}(\bar\Omega))}.
\end{aligned}
\end{equation}
Then, it follows from the estimates $\eqref{pde:3.2}$ and $\eqref{pde2.28}$ that
\begin{equation*}
\begin{aligned}
\|\mathcal{T}_1(p)(\cdot,t)\|_{C^{m+1,\alpha}(\bar\Omega)}
&\leq \|K_1(p)(\cdot,t)\|_{C^{m+1,\alpha}(\bar\Omega)}+
\|\tilde{f}_1(\cdot,t)\|_{C^{m+1,\alpha}(\bar\Omega)}\\
&\lesssim \int_{0}^{t}ds|A'(t-s)|
\big(\|p(\cdot,s)\|_{C^{m+1,\alpha}(\bar\Omega)}
+\|f(\cdot,s)\|_{C^{m,\alpha}(\bar\Omega)}\big)
+ \|f(\cdot,t)\|_{C^{m,\alpha}(\bar\Omega)},
\end{aligned}
\end{equation*}
and this implies
\begin{equation}\label{rh4}
\begin{aligned}
&\|\mathcal{T}_1(p)\|_{L^q(0,T_1;C^{m+1,\alpha}(\bar\Omega))} \\
&  \lesssim \|A'\|_{L^1(0,T_1)}
  \big(\|p\|_{L^q(0,T_1;C^{m+1,\alpha}(\bar\Omega))}
  +\|f\|_{L^q(0,T_1;C^{m,\alpha}(\bar\Omega))}\big)
  + \|f\|_{L^q(0,T_1;C^{m,\alpha}(\bar\Omega))}.
\end{aligned}
\end{equation}

Also, it is not hard to see that\footnote{By the definition,
the operator
$K_1$ is linear.}
\begin{equation}\label{f:3.6}
 \|\mathcal{T}_1(u)-\mathcal{T}_1(v)
 \|_{L^q(0,T_1;C^{m+1,\alpha}(\bar\Omega))}
 \overset{\eqref{f:3.2}}{\lesssim} \|A'\|_{L^1(0,T_1)}
\|u-v\|_{L^q(0,T_1;C^{m+1,\alpha}(\bar\Omega))}.
\end{equation}
Thus, by Lemma \ref{lemma2.5} we know that $A'\in L^1(0,T)$, and due to the absolute continuity of the integral, for any $\epsilon>0$, there exists $\delta>0$ such that for any interval $I\subset (0,+\infty)$ with $|I|\leq \delta$, we have $\int_{I}|A'|<\epsilon$.
Hence, there exists $\delta_0>0$ such that for $T_1=\delta_0$,
the estimate $\eqref{f:3.6}$ turns to be
\begin{equation*}
\|\mathcal{T}_1(u)-\mathcal{T}_1(v)
 \|_{L^q(0,\delta_0;C^{m+1,\alpha}(\bar\Omega))}
  \leq  \frac{1}{2}\|u-v\|_{L^q(0,\delta_0;C^{m+1,\alpha}(\bar\Omega))}.
\end{equation*}
This coupled with $\eqref{rh4}$ verifies the contraction property of $\mathcal{T}_1$
in $L^q(0,\delta_0;C^{m+1,\alpha}(\bar\Omega))$.

Consequently, it follows from Banach's fixed-point theorem that there exists unique solution $p_0$
such that $\mathcal{T}_1(p_0) = p_0$ in $L^q(0,\delta_0;C^{m+1,\alpha}(\bar\Omega))$, and this together with
the estimate $\eqref{rh4}$ leads to the desired result
\begin{equation*}
\|p_0\|_{L^2(0,\delta_0;C^{2,1/2}(\bar\Omega))}\lesssim
\|f\|_{L^2(0,\delta_0;C^{1,1/2}(\bar\Omega))}.
\end{equation*}
We have completed the proof.
\end{proof}

\subsection{\centering Extension of solution}

\begin{lemma}[inductions]\label{L3.2}
Let $0<\delta_0\ll 1$ be given as in Lemma $\ref{L3.1}$. Assume the same
conditions as in Proposition $\ref{P:1}$. Let $n\geq 2$ be an arbitrary fixed  large integer, and $T_k=k \delta_0$
with $k=1,\cdots,n$. Assume that there exists
a unique solution $p_0\in L^q(0,T_{n-1};C^{m+1,\alpha}(\bar\Omega))$ to the equations $\eqref{pde1.2}$,
satisfying the estimate
\begin{equation}\label{pri:3.3}
\|p_0\|_{L^q(0,T_{n-1};C^{m+1,\alpha}(\bar\Omega))}
\leq C_{n-1}
\|f\|_{L^q(0,T_{n-1};C^{m,\alpha}(\bar\Omega))}.
\end{equation}
Then, there exists a unique extension of the solution $p_0\in L^q(0,T_{n};C^{m+1,\alpha}(\bar\Omega))$ to the equations $\eqref{pde1.2}$,
and satisfies the estimate
\begin{equation}\label{pri:3.4}
\|p_0\|_{L^q(0,T_{n};C^{m+1,\alpha}(\bar\Omega))}
\leq C_{n}
\|f\|_{L^q(0,T_{n};C^{m,\alpha}(\bar\Omega))},
\end{equation}
where $C_n$ is monotonically ascending with respect to $n$.
%\footnote{That's the main reason
%that we merely obtained a finite time solution, compared to Laplace's transform approach.}.
\end{lemma}

\begin{proof}
We continue to adopt the notation used in Lemma $\ref{L3.1}$.
For any fixed $n\geq 2$,
and for any $t\in[T_{n-1},T_n]$,  we start from considering
\begin{equation}\label{pde:3.3}
\left\{
\begin{aligned}
  \mathcal{L}\Big[p_0(\cdot,t)+K_n(p_0)(\cdot,t)\Big]
  =\nabla\cdot\Big[&A'\ast f
  +A(0)f
  -\int_{0}^{T_{n-1}}dsA'(t-s)
  \nabla p_0(\cdot,s)
  \Big]\quad\text{in}~\Omega;\\
 \frac{\partial}{\partial\nu}
 \Big[p_0(\cdot,t)+K_n(p_0)(\cdot,t)\Big]
 = \vec{n}\cdot\bigg\{
 &A'\ast f+A(0)f
 -\int_{0}^{T_{n-1}}dsA'(t-s)
  \nabla p_0(\cdot,s)
  \bigg\}\quad\text{on}~\partial\Omega,
\end{aligned}\right.
\end{equation}
where the auxiliary function $K_n(p_0)$ is given by:
\begin{equation*}
\left\{\begin{aligned}
\nabla\cdot\Big[A(0)\nabla K_n(p_0)(\cdot,t)-\int_{T_{n-1}}^{t}ds
A'(t-s)\nabla p_0(\cdot,s)
\Big]&=0 &\quad &\text{in}~\Omega;\\
\vec{n}\cdot\Big[A(0)\nabla K_n(p_0)(\cdot,t)-\int_{T_{n-1}}^{t}ds
A'(t-s)\nabla p_0(\cdot,s)
\Big]&=0 &\quad &\text{on}~\partial\Omega,\\
\int_{\Omega}K_n(p_0)(\cdot,t)&=0.
\end{aligned}\right.
\end{equation*}
In terms of Schauder estimates, we obtain that
\begin{equation*}
\begin{aligned}
\|K_n(p_0)(\cdot,t)\|_{C^{m+1,\alpha}(\bar\Omega)}
\lesssim
\bigg\|\int_{T_{n-1}}^{t}ds
A'(t-s)\nabla p_0(\cdot,s)\bigg\|_{C^{m,\alpha}(\bar\Omega)}
\lesssim
\int_{T_{n-1}}^{t}ds
|A'(t-s)|\|\nabla p_0(\cdot,s)\|_{C^{m,\alpha}(\bar\Omega)}.
\end{aligned}
\end{equation*}
Let $1/q+1/{q'} = 1$. From H\"older's inequality and Fubini's theorem,
it follows that
\begin{equation*}
\begin{aligned}
&\int_{T_{n-1}}^{T_n}\|K_n(p_0)(\cdot,t)
dt\|^q_{C^{m+1,\alpha}(\bar\Omega)}
\lesssim
\int_{T_{n-1}}^{T_n}dt\Big(\int_{T_{n-1}}^{t}ds
|A'(t-s)|\|\nabla p_0(\cdot,s)\|_{C^{m,\alpha}(\bar\Omega)}\Big)^q\\
&\lesssim\int_{T_{n-1}}^{T_n}dt
\bigg(\Big(\int_{T_{n-1}}^{t}ds|A'(t-s)|\Big)^{\frac{q}{q'}}
\int_{T_{n-1}}^{t}ds
|A'(t-s)|\|\nabla p_0(\cdot,s)\|^q_{C^{m,\alpha}(\bar\Omega)}\bigg)\\
&\lesssim \bigg(\int_{0}^{\delta_0}dt|A'(t)|\bigg)^{\frac{q}{q'}+1}
\int_{T_{n-1}}^{T_n}dt\|p_0(\cdot,t)
\|^q_{C^{m+1,\alpha}(\bar\Omega)}.
\end{aligned}
\end{equation*}
This implies that
\begin{equation}\label{f:3.7}
 \|K_n(p_0)
\|_{L^q(T_{n-1},T_n; C^{m+1,\alpha}(\bar\Omega))}
\leq \frac{1}{2} \|p_0
\|_{L^q(T_{n-1},T_n; C^{m+1,\alpha}(\bar\Omega))}.
\end{equation}

\medskip

Then we introduce the following notation for the ease of the statement.
\begin{equation*}
\tilde{f}_n(\cdot,t):=\mathcal{L}^{-1}\nabla\cdot\Big[
\int_{0}^tdsA'(t-s)f(\cdot,s)+A(0)f(\cdot,t)
-\int_{0}^{T_{n-1}}dsA'(t-s)\nabla p_0(\cdot,s)
\Big]\quad \text{in}~\Omega,
\end{equation*}
which would be treated as the known data in the abstract equation below. For a.e. $t\in[T_{n-1},T_n]$, the equations \eqref{pde:3.3} is equivalent to
\begin{equation}\label{pde:3.4}
\big(p_0+K_n(p_0)\big)(\cdot,t)=\tilde{f}_n(\cdot,t)
\quad \text{in}\quad\Omega.
\end{equation}
Thus, the well-posedness of $\eqref{pde:3.3}$ is reduced to studying  the following contraction map\footnote{Similar to the form of $K_1$,
the operator $K_n$ is linear.}:
\begin{equation}\label{f:3.8}
  \mathcal{T}_n(p_0) := \tilde{f}_n - K_n(p_0).
\end{equation}
Obviously, the strict contraction property of $\mathcal{T}_n$  in $L^q(T_{n-1},T_n;C^{m+1,\alpha}(\bar\Omega))$ is due to the estimate
$\eqref{f:3.7}$.

\medskip

We continue to verify that the range of $\mathcal{T}_n$ is included in $L^q(T_{n-1},T_n;C^{m+1,\alpha}(\bar\Omega))$, and start from
\begin{equation*}
\|\mathcal{T}_n(p_0)(\cdot,t)\|_{C^{m+1,\alpha}(\bar\Omega)}
\leq \|K_n(p_0)(\cdot,t)\|_{C^{m+1,\alpha}(\bar\Omega)}+
\|\tilde{f}_n(\cdot,t)\|_{C^{m+1,\alpha}(\bar\Omega)}.
\end{equation*}
On account of the estimate $\eqref{f:3.7}$, it suffices to estimate
the quantity $\|\tilde{f}_n\|_{L^q(T_{n-1},T_n;C^{m+1,\alpha}(\bar\Omega))}$.
According to the definition of $\tilde{f}_n$ above,
by using Schauder estimates, we first have
\begin{equation*}
\begin{aligned}
&\|\tilde{f}_n(\cdot,t)\|_{C^{m+1,\alpha}(\bar\Omega)}\\
&\lesssim
\int_{0}^{t}ds|A'(t-s)|
\|f(\cdot,s)\|_{C^{m,\alpha}(\bar\Omega)}
+|A(0)|\|f(\cdot,t)\|_{C^{m,\alpha}(\bar\Omega)}
+\int_{0}^{T_{n-1}}ds|A'(t-s)|
\|\nabla p_0(\cdot,s)\|_{C^{m,\alpha}(\bar\Omega)},
\end{aligned}
\end{equation*}
and then appealing to H\"older's inequality and Fubini's theorem, as well as, the inductive assumption, we obtain
\begin{equation*}\label{rh11}
\begin{aligned}
&\int_{T_{n-1}}^{T_n}dt\|\tilde{f}_n(\cdot,t)
\|^q_{C^{m+1,\alpha}(\bar\Omega)} \\
%\lesssim
%\int_{T_1}^{T}\bigg(\int_{0}^{t}|A^{'}(t-s)|
%\|f(\cdot,s)\|_{C^{1,1/2}(\Omega)}ds\bigg)^2dt \\
%&+\int_{T_1}^{T}\bigg(
%|A(0)|\|f(\cdot,t)\|_{C^{1,1/2}(\Omega)}
%\bigg)^2dt
%+\int_{T_1}^{T}\bigg(\int_{0}^{T_1}|A'(t-s)|
%\|\nabla p_0(\cdot,s)\|_{C^{1,1/2}(\Omega)}ds
%\bigg)^2dt\\
&\lesssim\bigg(
\int_{0}^{T_n}dt|A'(t)|\bigg)^{\frac{q}{q'}+1}
\bigg(\int_{0}^{T_n}dt\|f(\cdot,t)
\|^q_{C^{m,\alpha}(\bar\Omega)}
+\int_{0}^{T_{n-1}}dt\|p_0(\cdot,t)\|^q_{C^{m+1,\alpha}(\bar\Omega)}\bigg)
+\int_{0}^{T_n}dt\|f(\cdot,t)
\|^q_{C^{m,\alpha}(\bar\Omega)}\\
&\overset{\eqref{pri:3.3}}{\lesssim}\int_{0}^{T_n}dt\|f(\cdot,t)
\|^q_{C^{m,\alpha}(\bar\Omega)}.
\end{aligned}
\end{equation*}
This coupled with $\eqref{f:3.8}$ leads to
\begin{equation*}
\begin{aligned}
\|\mathcal{T}_n(p_0)\|_{L^q(T_{n-1},T_n;C^{m+1,\alpha}(\bar\Omega))}
&\leq \|K_n(p_0)\|_{L^q(T_{n-1},T_n; C^{m+1,\alpha}(\bar\Omega))}
+\|\tilde{f}_n\|_{L^q(T_{n-1},T_n; C^{m+1,\alpha}(\bar\Omega))}\\
&\overset{\eqref{f:3.7}}{\leq} \frac{1}{2}\|p_0\|_{L^q(T_{n-1},T_n; C^{m+1,\alpha}(\bar\Omega))}
+ C\|f\|_{L^q(0,T_n; C^{m,\alpha}(\bar\Omega))}.
\end{aligned}
\end{equation*}

\medskip

Hence, by Banach's fixed-point theorem, there exists the unique solution
$p_0\in L^q(T_{n-1},T_n;C^{m+1,\alpha}(\bar\Omega))$, satisfying the equation
$\eqref{pde:3.4}$.  Moreover, we have
\begin{equation*}
\begin{aligned}
\|p_0\|_{L^q(T_{n-1},T_n;C^{m+1,\alpha}(\bar\Omega))}
\leq
\frac{1}{2}\|p_0\|_{L^q(T_{n-1},T_n; C^{m+1,\alpha}(\bar\Omega))}
+ C\|f\|_{L^q(0,T_n; C^{m,\alpha}(\bar\Omega))},
\end{aligned}
\end{equation*}
which further implies
\begin{equation*}
\|p_0\|_{L^q(T_{n-1},T_n;C^{m+1,\alpha}(\bar\Omega))}
\lesssim
\|f\|_{L^q(0,T_{n};C^{m,\alpha}(\bar\Omega))}.
\end{equation*}
As a result, it follows from \eqref{pri:3.3} that
\begin{equation*}
\|p_0\|_{L^q(0,T_n;C^{m+1,\alpha}(\bar\Omega))}
\lesssim
\|f\|_{L^q(0,T_n;C^{m,\alpha}(\bar\Omega))},
\end{equation*}
and this completes the whole proof.
\end{proof}

\noindent
\textbf{Proof of Proposition $\ref{P:1}$.} Combining Lemmas $\ref{lemma2.5}$, $\ref{L3.1}$, and $\ref{L3.2}$, we have the desired results.
\qed

\paragraph{Acknowledgements.}
The second author deeply appreciated Prof. Felix Otto for his instruction and encourages when he held a post-doctoral position in the Max Planck Institute for Mathematics in the Sciences (in Leipzig).
The second author was supported by the Young Scientists Fund of the National Natural Science Foundation of China (Grant No. 11901262); the Fundamental Research Funds for the Central Universities (Grant No.lzujbky-2021-51);
The third author was supported by the National Natural Science Foundation of China (Grant No. 12171010).

\paragraph{Data availability statement.}
 Data sharing not applicable to this article as no datasets were generated or analysed during the current study.
 
\paragraph{Conflict of interest.} 
On behalf of all authors, the corresponding author states that there is no conflict of interest.

\addcontentsline{toc}{section}{References}

%\bibliographystyle{plain}
%\bibliography{mybib}

\begin{thebibliography}{000}

%\bibitem{Allaire91}
%G. Allaire,
%Homogenization of the Navier-Stokes equations
%in open sets perforated with tiny holes (I-II).
%Arch. Rational Mech. Anal. 113, no. 3, 209-259;
%261-298
%(1990)


\bibitem{Allaire92}
%Allaire, Gr\'egoire;
Allaire, G.:
Homogenization of the unsteady Stokes equations in porous media.
Progress in partial differential equations: calculus of variations,
applications (Pont-a-Mousson, 1991), 109–123, Pitman Res. Notes Math. Ser., 267,
Longman Sci. Tech., Harlow,
(1992)

\bibitem{Allaire89}
Allaire, G.:  Homogenization of the Stokes flow in a connected porous medium. Asymptotic Anal. 2, no. 3, 203-222
(1989)



\bibitem{Allaire-Mikelic97}
%Allaire, Grégoire; Mikelić, Andro
Allaire, G., Mikeli\'c, A.:
One-phase Newtonian flow.
Homogenization and Porous Media, 45–76, 259–275,
Interdiscip. Appl. Math., 6, Springer, New York,
(1997)

%\bibitem{Armstrong-Dario18}
%Armstrong, S., Dario, P.:
%Elliptic regularity and quantitative homogenization on percolation clusters.
%Comm. Pure Appl. Math. 71, no. 9, 1717–1849
%(2018)

%\bibitem{Armstrong-Kuusi-Mourrat19}
%Armstrong, S., Kuusi, T., Mourrat, J.-C.:
%Quantitative Stochastic Homogenization and
%Large-scale Regularity. Grundlehren der mathematischen Wissenschaften
%[Fundamental Principles of Mathematical Sciences], 352. Springer, Cham, (2019)

%\bibitem{Armstrong-Kuusi-Mourrat-Prange17}
%%Armstrong, Scott; Kuusi, Tuomo; Mourrat, Jean-Christophe; Prange, Christophe
%Armstrong, S., Kuusi, T., Mourrat, J.-C., Prange, C.:
%Quantitative analysis of boundary layers in periodic homogenization.
%Arch. Ration. Mech. Anal. 226, no. 2, 695-741
%(2017)


\bibitem{Balazi-Allaire-Omnes24}
Balazi, L., Allaire, G., Omnes, P.:
Sharp convergence rates for the homogenization of the Stokes
equations in a perforated domain.
Discrete Contin. Dyn. Syst. Ser. B 30, no. 5, 1550-1574 (2025)

\bibitem{Bogovskii79}
Bogovskii, M.E.: 
Solution of the first boundary value problem for an equation of continuity of an incompressible medium. (Russian) Dokl. Akad. Nauk SSSR 248, no. 5, 1037-1040 (1979)

%\bibitem{Avellaneda-Lin87}
%Avellaneda, M., Lin, F.:
%Compactness methods in the theory of homogenization.
%Comm. Pure Appl. Math. 40, no. 6, 803-847
%(1987)

%\bibitem{Bella-Giunti-Otto17}
%Bella, P., Giunti, A., Otto, F.:
%Quantitative stochastic homogenization:
%local control of homogenization error through corrector. Mathematics and materials, 301-327,
%IAS/Park City Math. Ser., 23, Amer. Math. Soc., Providence, RI,
%(2017)

%\bibitem{Bella-Oschmann21}
%Bella, P., Oschmann, F.:
%Homogenization and low Mach number limit of compressible Navier-Stokes
%equations in critically perforated domains, arXiv:2104.05578v1
%(2021)

\bibitem{Chechkin-Piatnitski-Shamaev2007}
%Chechkin, G. A.; Piatnitski, A. L.; Shamaev, A. S.
Checkin, G., Piatnitski, A., Shamaev, A.:
Homogenization. Methods and applications.
Translated from the 2007 Russian original by Tamara Rozhkovskaya.
Translations of Mathematical Monographs, 234.
American Mathematical Society, Providence, RI,
(2007)

\bibitem{Dacorogna02}
Dacorogna, B.: Existence and regularity of solutions of $d\omega=f$ with Dirichlet boundary conditions. Nonlinear problems in mathematical physics and related topics, I, 67–82, Int. Math. Ser. (N. Y.), 1, Kluwer/Plenum, New York, (2002)

\bibitem{Evans10}
Evans, L.: Partial differential equations.
Second edition. Graduate Studies in Mathematics,
19. American Mathematical Society, Providence, RI, (2010)

%\bibitem{Clozeau-Josien-Otto-Xu}
%Clozeau, N., Josien, M., Otto, F., Xu, Q.:
%Bias in the representative volume element method:
%periodize the ensemble instead of its realization.
%(In preparaion)

%\bibitem{Duerinckx-Gloria21}
%Duerinckx, M., Gloria, A.: Corrector equations in fluid mechanics: effective viscosity of colloidal suspensions. Arch. Ration. Mech. Anal. 239, no. 2, 1025–1060
%(2021)

%\bibitem{Duerinckx-Gloria21-1}
%Duerinckx, M., Gloria, A.:
%Quantitative homogenization theory for random suspensions in
%steady Stokes flow, arXiv:2103.06414
%(2021)


\bibitem{Galdi11}
%Galdi, G. P.
Galdi, G.:
An introduction to the mathematical theory of the Navier-Stokes equations.
Steady-state problems. Second edition.
Springer Monographs in Mathematics. Springer, New York,
(2011)

%\bibitem{Gerard-Masmoudi12} %Gérard-Varet, David; Masmoudi, Nader
%G\'erard-Varet, D.,  Masmoudi, N.:
%Homogenization and boundary layers. Acta Math. 209, no.1, 133-178
%(2012)

\bibitem{Giaquinta-Martinazzi12}
Giaquinta, M.,  Martinazzi, L.:
An Introduction to the Regularity Theory
for Elliptic Systems, Harmonic Maps and Minimal Graphs. Second edition. Appunti.
Scuola Normale Superiore di Pisa (Nuova Serie)
[Lecture Notes. Scuola Normale Superiore di Pisa (New Series)],
11. Edizioni della Normale, Pisa,
(2012)

%\bibitem{Giunit20}
%Giunti, A.:
%Derivation of Darcy's law in randomly perforated domains.
%Calc. Var. Partial Differential Equations 60, no. 5, Paper No. 172, 30 pp.
%(2021)

%\bibitem{Giunit19}
%%Giunti, Arianna; Höfer, Richard M.
%Giunti, A., H\"ofer, R.:
%Homogenisation for the Stokes equations in randomly perforated domains under almost minimal
%assumptions on the size of the holes.
%Ann. Inst. H. Poincaré Anal. Non Linéaire 36, no. 7, 1829-1868
%(2019)

%\bibitem{Gloria-Neukamm-Otto20}
%Gloria, A., Neukamm, S., Otto, F.:
%A regularity theory for random elliptic operators.
%Milan J. Math. 88, no. 1, 99-170
%(2020)

%\bibitem{Gloria-Otto14}
%Gloria, A., Otto, F.:
%The corrector in stochastic homogenization: the corrector in
%stochastic homogenization: optimal rates, stochastic integrability and
%fluctuations, arXiv:1510.08290v3
%(2015)

%\bibitem{Heywood-Walsh94}
%%Heywood, John G.; Walsh, Owen D.
%Heywood, J.,  Walsh, O.:
%A counterexample concerning the pressure in the Navier-Stokes equations,
%as $t\to 0^{+}$. Pacific J. Math. 164, no. 2, 351-359
%(1994)

\bibitem{Jankowiak-Lozinski24}
Jankowiak, G., Lozinski, A.: Non-conforming multiscale finite element method for Stokes flows in heterogeneous media. Part II: Error estimates for periodic microstructure. Discrete Contin. Dyn. Syst. Ser. B 29, no. 5, 2298-2332 (2024)

%\bibitem{Josien-Otto20}
%M. Josien, F. Otto,
%The annealed Calder\'on-Zygmund estimate as
%convenient tool in quantiative stochastic homogenization.
%arXiv:2005.08811v1 (2020).

\bibitem{Jikov-Kozlov-Oleinik94}
%Jikov, V. V.; Kozlov, S. M.; Oleĭnik, O. A.
Jikov, V., Kozlov, S.,  Oleinik, O.:
 Homogenization of differential operators and integral functionals. Translated from the Russian by G. A. Yosifian. Springer-Verlag, Berlin, (1994)


\bibitem{Jin13}
Jin, B.: On the Caccioppoli inequality of the unsteady Stokes system.
Int. J. Numer. Anal. Model. Ser. B4, no. 3, 215-223
(2013)

\bibitem{Jing20}
Jing, W.:
A unified homogenization approach for the Dirichlet problem in perforated domains. SIAM J. Math. Anal. 52, no. 2, 1192-1220 (2020)


\bibitem{Jing-Lu-Prange24}
Jing, W., Lu, Y., Prange, C.:
Unified quantitative analysis of the Stokes equations in dilute perforated domains via layer potentials. arXiv:2409.16960 (2024)


%\bibitem{Kenig12}
%C. Kenig, F. Lin, Z. Shen,
%Convergence rates in $L^2$ for elliptic homogenization problems.
%Arch. Ration. Mech. Anal. 203 (2012), no. 3, 1009-1036.
%

\bibitem{Ladyzhenskaya69}
%Ladyzhenskaya, O.A.
Ladyzhenskaya, O.:
The mathematical theory of viscous incompressible flow. Second English edition, revised and enlarged Translated from the Russian by Richard A. Silverman and John Chu Mathematics
and its Applications, Vol. 2 Gordon and Breach, Science Publishers, New York-London-Paris, (1969)

%\bibitem{Lipton-Avellaneda90}
%Lipton, R., Avellaneda,  M.: Darcy's law for slow viscous flow past a stationary array of bubbles. Proc. Roy. Soc. Edinburgh Sect. A 114, no. 1-2, 71-79
%(1990)

\bibitem{Lions81}
Lions, J.-L.:
Some methods in the mathematical analysis of systems and their control. Kexue Chubanshe (Science Press),
Beijing; Gordon $\&$ Breach Science Publishers, New York,
(1981)

%\bibitem{Lu-P21}
%Lu, Y., Pokorn\'y, M.:
%Homogenization of stationary Navier-Stokes-Fourier system in domains with tiny holes. J. Differential Equations 278, 463-492
%(2021)

\bibitem{Lu21}
Lu, Y.:
Uniform estimates for Stokes equations in a domain with a small hole and applications in homogenization problems. Calc. Var. Partial Differential Equations 60, no. 6, Paper No. 228, 31 pp. (2021)


%\bibitem{Masmoudi02}
%Masmoudi, N.: Homogenization of the compressible Navier-Stokes equations in a porous medium. A tribute to J. L. Lions. ESAIM Control Optim. Calc. Var. 8, 885-906
%(2002)

\bibitem{Marusi-Mikelic96}
Maru\v{s}i\'{c}-Paloka, E., Mikeli\'{c}, A.:
An error estimate for correctors in the homogenization
of the Stokes and the Navier-Stokes equations in a porous medium,
Boll. Un. Mat. Ital. A (7) 10, no. 3, 661-671
(1996)


\bibitem{Mikelic94}
Mikeli\'c, A.:
Mathematical derivation of the Darcy-type law with memory effects,
governing transient flow through porous media. Glas. Mat. Ser. III 29(49), no. 1, 57-77
(1994)

\bibitem{Mikelic91}
Mikeli\'c, A.:
Homogenization of nonstationary Navier-Stokes equations in a domain with a grained boundary. Ann. Mat. Pura Appl. (4) 158, 167-179
(1991)


\bibitem{Mikelic-Paoli99}
%Mikelić, Andro; Paoli, Laetitia
Mikeli\'c, A., Paoli, L.:
Homogenization of the inviscid incompressible fluid flow through a 2D porous medium. Proc. Amer. Math. Soc. 127, no. 7, 2019-2028
(1999)




\bibitem{Temam79}
Temam, R.: Navier-Stokes Equations.
Theory and numerical analysis. Revised edition.
With an appendix by F. Thomasset. Studies in Mathematics and its Applications, 2. North-Holland Publishing Co., Amsterdam-New York,
(1979)

\bibitem{Tsai18}
%Tsai, Tai-Peng
Tsai, T.-P.:
Lectures on Navier-Stokes Equations.
Graduate Studies in Mathematics, 192. American Mathematical Society, Providence, RI,
(2018)

%\bibitem{Schmutz08}
%%Schmutz, Eric
%Schmutz, E.:
%Rational points on the unit sphere. Cent. Eur. J. Math. 6, no. 3, 482-487
%(2008)


\bibitem{Sanchez-Palencia80}
%Sánchez-Palencia, Enrique
S\'anchez-Palencia, E.: Nonhomogeneous media and vibration theory. Lecture Notes in Physics, 127. Springer-Verlag, Berlin-New York,
(1980)

\bibitem{Sandrakov07}
Sandrakov, G.:
Estimates for the solutions of a system of Navier-Stokes equations with rapidly oscillating data.
(Russian) Dokl. Akad. Nauk 413 (2007), no. 3, 317-319; translation in Dokl. Math. 75,
no. 2, 252-254
(2007)



\bibitem{Sandrakov97}
%Sandrakov, G. V.
Sandrakov, G.:
Averaging of the nonstationary Stokes system with viscosity in a punctured domain.
(Russian) Izv. Ross. Akad. Nauk Ser. Mat. 61 (1997), no. 1, 113-140;
translation in Izv. Math. 61, no. 1, 113-141
(1997)

%\bibitem{Shen-Zhuge18}
%Shen, Z., Zhuge, J.:
%Boundary layers in periodic homogenization of Neumann problems.
%Comm. Pure Appl. Math. 71, no. 11, 2163-2219
%(2018)


\bibitem{Shen18}
Shen, Z.: Periodic Homogenization of Elliptic Systems.
Operator Theory: Advances and Applications, 269.
Advances in Partial Differential Equations (Basel). Birkhäuser/Springer, Cham, (2018)


\bibitem{Shen20}
Shen, Z.: Sharp convergence rates for Darcy's law.
Comm. Partial Differential Equations 47, no. 6, 1098-1123 (2022)

\bibitem{Shen22}
Shen, Z.: Compactness and large-scale regularity for Darcy's law. J. Math. Pures Appl. (9) 163, 673–701 (2022)

\bibitem{Shen23}
Shen, Z.: Darcy's law for porous media with multiple microstructures. Matematica 2, no. 2, 438-478 (2023)

%\bibitem{Shen21}
%Shen, Z.: Compactness and large-scale regularity for Darcy's law,
%arXiv:2104.05074v1
%(2021)

%\bibitem{Shen91}
%Shen, Z.: Boundary value problems for parabolic Lam\'e systems and a nonstationary
%linearized system of Navier-Stokes equations in Lipschitz cylinders.
%Amer. J. Math. 113, no. 2, 293-373
%(1991)

\bibitem{Simon99}
%Simon, Jacques
Simon, J.:
 On the existence of the pressure for solutions of the variational Navier-Stokes equations. J. Math. Fluid Mech. 1, no. 3, 225-234
(1999)

%\bibitem{Suslina19}
%Suslina, T.: Homogenization of the stationary Maxwell system with periodic coefficients in a bounded domain. Arch. Ration. Mech. Anal. 234, no. 2, 453-507
%(2019)

%\bibitem{Tartar80}
%Tartar, L.: Incompressible fluid flow in a porous medium: convergence of the homogenization process, in Nonhomogeneous media and vibration theory, edited by E. Sanchez–Palencia 368-377
%(1980)

%\bibitem{Wang-Xu-Zhao21}
%Wang, L., Xu, Q., Zhao, P.:
%Convergence rates for linear elasticity systems on perforated domains.
%Calc. Var. Partial Differential Equations 60, no. 2, Paper No. 74, 51 pp.
%(2021)

\bibitem{Wolf18}
%Wolf, Jörg
Wolf, J.: A note on the pressure of strong solutions to the Stokes system in bounded and exterior domains.
J. Math. Fluid Mech. 20, no. 2, 721-731
(2018)

\bibitem{Wang-Xu-Zhang22}
Wang, L., Xu, Q., Zhang, Z.:
Corrector estimates and homogenization error of unsteady flow ruled by Darcy's law. arXiv:2202.04826v1 (2022)

\bibitem{Xu16}
Xu, Q.:
Convergence rates for general elliptic homogenization problems in Lipschitz domains. SIAM J. Math. Anal. 48, no. 6, 3742-3788
(2016)


\end{thebibliography}

%\newpage
%
%
%\noindent Li Wang\\
%School of Mathematic Sciences,
%Peking University, Peking 100871, China.\\
%E-mail:wangli@math.pku.edu.cn\\
%
%\noindent Qiang Xu\\
%School of Mathematics and Statistics,
%Lanzhou University, Lanzhou 730000, China.\\
%E-mail:xuq@lzu.edu.cn\\
%
%\noindent Zhifei Zhang\\
%School of Mathematic Sciences,
%Peking University, Peking 100871, China.\\
%E-mail:zfzhang@math.pku.edu.cn\\

\end{document}